\documentclass[12pt]{article}
\usepackage{fullpage,amsthm,enumitem}
\usepackage{amsmath,amsfonts,amssymb,graphicx,latexsym} 

\usepackage{mathpazo} 
\usepackage[dvipsnames]{xcolor} 

\usepackage[colorlinks,citecolor=blue,urlcolor=blue,hypertexnames=false]{hyperref}

\usepackage{mathrsfs} 

\usepackage{ragged2e}
\usepackage{bm}
\usepackage{mathtools}
\usepackage{mathtools}
\usepackage{blkarray, bigstrut}
\usepackage{soul}

\newtheorem{thm}{Theorem}
\newtheorem{cor}[thm]{Corollary}
\newtheorem{lem}[thm]{Lemma}
\newtheorem{prop}[thm]{Proposition}
\newtheorem{defn}[thm]{Definition}

\newtheorem{rem}[thm]{Remark}

\newtheorem{hypo}[thm]{Hypothesis}

\newcommand{\N}{\mathbb{N}}
\newcommand{\Z}{\mathbb{Z}}
\newcommand{\R}{\mathbb{R}}
\newcommand{\C}{\mathbb{C}}
\renewcommand{\S}{\mathcal{S}}
\renewcommand{\H}{\mathcal{H}}
\renewcommand{\O}{\mathcal{O}}

\newcommand{\mB}{\mathcal{B}}

\newcommand{\vect}[2]{\begin{pmatrix}#1\\#2\end{pmatrix}}
\newcommand{\nrm}[2]{\left\|#1\right\|_{#2}}
\newcommand{\norm}[1]{\left\|#1\right\|}

\newcommand{\absnorm}[1]{\norm{ #1 }_1}

\newcommand{\mW}{\mathscr{W}}

\newcommand{\mJ}{\mathcal{J}}
\newcommand{\bT}{\mathbf{T}}
\newcommand{\BS}{\mathfrak{BS}}

\DeclareMathOperator{\Tr}{Tr}
\DeclareMathOperator{\dom}{dom}
\DeclareMathOperator{\ran}{ran}

\usepackage[autostyle=false, style=english]{csquotes}
\usepackage{ulem}

\MakeOuterQuote{"}

\DeclarePairedDelimiterX{\bnorm}[1]{\big\lVert}{\big\rVert}{#1}
\DeclarePairedDelimiterX{\Bnorm}[1]{\Big\lVert}{\Big\rVert}{#1}

\newcommand{\Nat}{\mathbb{N}}

\renewcommand{\S}{\mathcal{S}}
\newcommand{\A}{\mathcal{A}}

\newcommand{\WT}{\mathscr{W}_{\rm Taylor}(\R)}
\newcommand{\WCT}{\mathscr{W}_{\rm Certainly Taylor}(\R)}
\newcommand{\Q}{\mathfrak{Q}}

\newcommand{\Scal}{\mathcal{S}}
\newcommand{\Hcal}{\mathcal{H}}
\renewcommand{\H}{\mathcal{H}}

\newcommand{\Snp}{\mathcal{S}^{n/p}}

\title{Differentiation, Taylor series, and all order spectral shift functions, for relatively bounded perturbations}
\author{Arup Chattopadhyay\footnote{arupchatt@iitg.ac.in, Department of Mathematics, IIT Guwahati, 781039, India.},~~ Teun D.H. van Nuland\footnote{t.d.h.vannuland@tudelft.nl, TU Delft, EWI/DIAM, 2600 GA Delft, The Netherlands.},~~ Chandan Pradhan\footnote{chandan.pradhan2108@gmail.com, Department of Mathematics, IISc Bangalore, 560012, India.}}

\begin{document}
	\maketitle
	
	\begin{abstract}
		Given $H$ self-adjoint, $V$ symmetric and relatively $H$-bounded, and $f:\R\to\C$ satisfying mild conditions, we show that the Gateaux derivative
		$$\frac{d^n}{dt^n}f(H+tV)|_{t=0}$$
		exists in the operator norm topology, for every natural $n$, and establish perturbation formulas for multiple operator integrals under relatively bounded perturbations. If the $H$-bound of $V$ is less than 1, we obtain sufficient conditions on $f$ which ensure that the Taylor expansion
		$$f(H+V)=\sum_{n=0}^\infty\frac{1}{n!}\frac{d^n}{dt^n} f(H+tV)\big|_{t=0}$$
		exists and converges absolutely in operator norm.
		Assuming that $V(H-i)^{-p}\in\S^{s/p}$ for $p=1,\ldots,s$ for some $s\in\N$, we show that the Krein--Koplienko spectral shift functions $\eta_{k,H,V}$, satisfying 
		$$\Tr\left(f(H+V)-\sum_{m=0}^{k-1}\frac{1}{m!}\frac{d^m}{dt^m} f(H+tV)\big|_{t=0}\right)=\int_\R f^{(k)}(x)\eta_{k,H,V}(x)dx,$$
		exist for every $k=1,2,3,\ldots$, independently of $s$. The latter result (which is significantly stronger than \cite{vNS22}) is new also in the case that $V$ is bounded. The proof is based on \cite{PSS}, combined with a generalisation of the multiple operator integral compatible with \cite{HMvN}.
		We discuss applications of our results to quantum physics and noncommutative geometry.
	\end{abstract}
	
	\section{Introduction}\label{sec:intro}
	\footnote{2010 {\it Mathematics Subject Classification}: 47A55, 46A56; 47B93, 47A60, 58B34.}
	\footnote{{\it Key words and phrases}: Operator derivatives, relatively bounded operator, Taylor series, multiple operator integration, spectral shift functions.}
	The Taylor series is a fundamental tool in real analysis, and its noncommutative or operator-theoretic generalization is useful in quantum mechanics, quantum field theory, and noncommutative geometry.
	It relies firstly on the existence of derivatives of operator functions, for which the most natural setting is arguably the one of Kato and Rellich.
	Indeed, if $H$ is a self-adjoint (possibly unbounded) operator in a separable Hilbert space $\H$, and $V$ is symmetric and relatively $H$-bounded, then the Kato--Rellich theorem states that
	$$H+tV\qquad\text{is self-adjoint for all}\qquad t\in(-\tfrac{1}{a},\tfrac{1}{a}),$$
	where $a$ is the $H$-bound of $V$. Hence, by Borel functional calculus, we may apply a bounded measurable function $f:\R\to\C$ to obtain a mapping
	\begin{align}\label{eq:intro:operator function}
		(-\tfrac1a,\tfrac1a)\to\mB(\H),\qquad t\mapsto f(H+tV).
	\end{align}
	For suitable classes of $f$, the following questions may arise:
	\begin{enumerate}
		\item[Q1.]\label{Q1} Is $t\mapsto f(H+tV)$ $n$-times differentiable?
		\item[Q2.]\label{Q2} Does the Taylor-expansion of $f(H+tV)$ in orders of $t$ converge in norm?
		\item[Q3.]\label{Q3} Do all higher-order spectral shift functions exist?
	\end{enumerate}
	{For reasons of summability, more assumptions on $H$ and $V$ are needed to affirmatively answer Q3, as detailed in the next subsection. 
		It turns out that Q1 and Q2 have an affirmative answer without extra assumptions on $H$ and $V$, and with conditions on $f$ that are relatively easy to verify. 
		
		


		Pioneering results concerning the differentiability of the operator function \eqref{eq:intro:operator function} were obtained in~\cite{DaletskiiKrein1956} by use of the  {\it double operator integral}. The $n$-th order derivative of the operator function~\eqref{eq:intro:operator function}, for bounded perturbations $V$, is naturally expressed in terms of a {\it multiple operator integral} (a concept originating in \cite{BS1,DaletskiiKrein1956,Pavlov1971,SolomyakStenkin1971,Stenkin1977}), which generalizes the notion of a double operator integral to more than two variables (see, e.g.,~\cite{ACDS,CoLemSu21,HMvN,Peller,PSS,PoSkSu14,PoSkSu16,Sk17Adv}). A series of works has since addressed the differentiability of such operator functions under various conditions on $f$, $H$, and $V$ -- see~\cite[Paragraph~5.3]{ST19}, and references therein.

		When differentiating in the direction of an unbounded operator $V$, the operator-norm Frechet derivative of \eqref{eq:intro:operator function} is useless because it contains $\|V\|=\infty$ in the denominator. One is thus led to the Gateaux derivative, for which the following holds.
		\begin{thm}\label{thm intro:higher derivatives in norm}
			Let $H$ be self-adjoint in a separable Hilbert space $\H$, and let $V$ be symmetric and relatively $H$-bounded with $H$-bound $a\in[0,\infty)$. If $n\in\N$ and $f\in C^{n+1}(\R)$ is such that the functions $\frac{d^{p}}{dx^{p}}(f(x)(x-i)^{p})$ are Fourier transforms of finite complex measures for all $p=0,\ldots,n+1$, then the mapping
			$$(-\tfrac{1}{a},\tfrac{1}{a})\to\mB(\H),\qquad t\mapsto f(H+tV)$$
			is $n$ times differentiable in operator norm.
		\end{thm}
		The above result follows from Theorem \ref{thm:higher derivatives in norm}, which moreover provides alternative conditions to ensure operator differentiability, and establishes that
		%
		%
		\begin{align}\label{eq:MOI derivative}
			\frac{1}{n!}\frac{d^n}{dt^n}f(H+tV)\big|_{t=0}=\bT^{H,\ldots,H}_{f^{[n]}}(V,\ldots,V),
		\end{align}
		in which $\mathbf{T}^{H,\ldots,H}_{f^{[n]}}$ is an $n$-multilinear map from relatively bounded operators to bounded operators which extends the multiple operator integral of \cite{ACDS,Peller} and is compatible with \cite{HMvN}. The main reason multiple operator integrals are powerful tools in perturbation theory, noncommutative analysis, and noncommutative geometry, is because of their surprisingly clean analytical and algebraic properties. We extend several such properties to the relatively bounded setting. A further motivation for these developments are the questions raised in~\cite{widom}.
		
		For unbounded $V$, trying to obtain \eqref{eq:MOI derivative} directly by naively extending the multiple operator integral (directly inserting $V$ in the integral) fails because of domain problems. The intuition behind a formula like \eqref{eq:MOI derivative} in the relatively bounded setting is that the summability of a multiple operator integral $T_\phi^{H,\ldots,H}(V,\ldots,V)$ is typically related to the asymptotic behaviour of its symbol $\phi$, and for suitable $f$ the asymptotic behaviour of the divided difference $f^{[n]}(\lambda_0,\ldots,\lambda_n)$ can be controlled by $\lambda_0^{-1}\cdots\lambda_n^{-1}$. In the unbounded case however, distinct subtleties arise in the translation from symbols to operators, which we address with combinatorial techniques.
		
		The results mentioned above yield explicit expressions and bounds for the higher-order operator-norm Gateaux derivative $\frac{d^n}{dt^n} f(H+tV)\big|_{t=0}$ and the Taylor remainder. Particular upshots are algebraic rules for multi-variable Gateaux derivatives (see Remark \ref{rem:multi-Gateaux}), and the existence of a constant $c_f$ such that $$\norm{f(H+V)-f(H)}\leq c_f\|V(H-i)^{-1}\|.$$ 
		It would be interesting if the same techniques, or some modified version of them, could be used to obtain optimal constants.
		
		Using Theorem \ref{thm intro:higher derivatives in norm}, we establish the existence of an absolutely norm-converging Taylor series. The respective radius of convergence is related both to $\|V(H-i)^{-1}\|$ and an explicit norm of $f$, which uses the smoothness as well as the decay at infinity of $f$.
		\begin{thm}\label{thm intro:Taylor series}
			Let $H$ be self-adjoint in a separable Hilbert space $\H$, and let $V$ be symmetric and relatively $H$-bounded with $H$-bound $a\in[0,1)$. Let $f\in C^\infty(\R)$ be such that $\frac{d^n}{dx^n}(f(x)(x-i)^{n})$ are Fourier transforms of finite complex measures $\mu_n$ for all $n\in\N$ and let there exist constants $c_f,C_f>0$ such that $\|\mu_n\|\leq c_fC_f^{n}n!$ (see Lemma \ref{lem:examples Taylor series} for examples). If $\|V(H-i)^{-1}\|<\frac{1}{1+C_f}$, then we have
			$$f(H+V)=\sum_{n=0}^\infty\frac{1}{n!}\frac{d^n}{dt^n} f(H+tV)\big|_{t=0}\,,$$
			where the series converges absolutely in operator norm. If $\|\mu_n\|\leq c_f C_f^n (n!)^\gamma$ for $\gamma<1$ then the above identity holds, with absolutely norm-converging series, without any assumption on $\|V(H-i)^{-1}\|$.
		\end{thm}
		
		The `noncommutative Taylor series' may be understood as the combination of Theorem \ref{thm intro:Taylor series} and \eqref{eq:MOI derivative}, possibly supplemented with one of the various explicit (integral) expressions for \eqref{eq:MOI derivative}. As such, the noncommutative Taylor series has been studied in numerous contexts \cite{Feynman,HMvN,Lesch,Daletskii,Hansen,GM}. In quantum physics it often comes in guises such as the Dyson series, Volterra series and Born series. In noncommutative geometry, the noncommutative Taylor series comes up in the context of the spectral action \cite{S14,Sui11,vNvS21a} and heat kernel expansions \cite{Lesch,IM1,HMvN}, and the algebraic structure underlying its summands can be informative \cite{oneloop}. 
		We hope that, by unifying results scattered throughout the literature, Theorem \ref{thm intro:Taylor series} and \eqref{eq:MOI derivative} bring connections between these applications to the foreground.
		
		
		

		\paragraph{Spectral shift functions of all orders}

		The spectral shift function is a useful notion for the spectral analysis of quantum systems, and the question of its existence has sparked enormous progress in mathematical perturbation theory. 
		First defined by Krein \cite{Krein}, and generalized to higher order by Koplienko \cite{Koplienko}, the spectral shift function of order $k$ is the function $\eta_k=\eta_{H_0,V,k}$
		such that for all sufficiently regular $f:\R\to\C$ we have
		$$\Tr\left(f(H+V)-\sum_{m=0}^{k-1}\frac{d^m}{dt^m}f(H+tV)\big|_{t=0}\right)=\int_\R\eta_k(x)f^{(k)}(x)dx.$$
		In \cite{PSS} the spectral shift functions $\eta_k$ of orders $k\geq n$ were shown to exist whenever $V\in\mathcal S^n$. Because physical situations -- for example when $H$ is a differential operator and $V$ a multiplication operator -- require $V$ to have continuous spectrum, the paper \cite{vNS22} used the much weaker assumption $V(H-i)^{-1}\in\mathcal S^{n}$ ($n\in\N$) to obtain the spectral shift functions of order $k\geq n$. Under a similar assumption, \cite{vNS23} added existence of $\eta_{n-1}$ when $n$ is even. Throughout the literature one finds many such assumptions which generalise the case of $(H-i)^{-1}\in\S^n$ to a `locally compact' setting. Often, existence of the spectral shift function depends on $n$, which for a differential operator $H$ depends on its order and the dimension of the underlying space. The notable exception is that $\eta_1$ was shown to exist independently of $n$ in \cite{Yafaev}. However, for general $n\in\N$, existence of $\eta_2,\ldots,\eta_{n-1}$ remained open, even for bounded $V$ and under any of the above reasonable conditions. 
		
		The final result of this paper is the existence, uniqueness and regularity properties of $\eta_k$ for all orders $k\in\N$ under the assumptions that $V$ is symmetric and 
		\begin{align}\label{assumption}
			V(H-i)^{-p}\in\mathcal S^{n/p}\qquad(p=1,\ldots,n),
		\end{align}
		for an arbitrary $n\in\N$. 
		The assumption \eqref{assumption} unifies the `locally compact' assumption $V(H-i)^{-n}\in\mathcal S^1$ with the `relative Schatten' assumption $V(H-i)^{-1}\in\mathcal S^n$ and is used in \cite{Rennie,S21,SZ}. We exemplify its applicability in Section \ref{sct:Examples}.

		In practice, the summability $n$ should be correlated with the dimension of the underlying space and the order of the initial differential operator. The fact that in \cite{vNS23} this summability is correlated with the order of the spectral shift function, is thus revealed as an artefact from not starting with the `right' assumption. 
		
		Similar but distinct from the above crucial point, the boundedness assumption on $V$ appearing in previous works is revealed to be artificial by Theorem \ref{thm:nth order}. This shows the power of our extension of the multiple operator integral and the superscript difference identity, Theorem \ref{thm:superscript difference}.
		

		Compared to earlier results on higher-order spectral shift \cite{PSS,vNS22,vNS23}, the proof moreover gains an inductive structure: existence of the spectral shift function $\eta_k$ can be deduced from the existence of $\eta_{k-1}$. Our proof is thus split into induction basis -- existence of $\eta_1$ -- in Section \ref{sct:Krein}, and induction step in Section \ref{sct:Higher order}. Under different assumptions, existence of $\eta_1$ was already known, and the reader only interested in the cases already covered by \cite{Yafaev} needs only to read Section \ref{sct:Higher order}.
		
		While finishing this manuscript, the authors became aware of the preprint version of \cite{AP24}, which independently obtains first-order differentiability for relatively bounded perturbations and a sharpening of \cite{Yafaev} for relatively trace-class perturbations, for a different class of scalar functions. During the review process of the current paper, an independent preprint \cite{Fuerst} appeared connecting higher-order spectral shift functions with index theory for a concrete class of operators; its use of multiple operator integrals is closely related to our approach, though the goals differ.
		
		
		
		\paragraph{This article is structured as follows}
		Section \ref{sec:prel} contains preliminaries,  Section \ref{sec:moi} proves the multiple operator integration techniques needed in Sections \ref{sec:operator differ} and \ref{sec:ssf},  Section \ref{sec:operator differ} proves higher-order differentiability and existence of Taylor series, and Section \ref{sec:ssf} studies the existence of spectral shift functions of all orders. 

		\paragraph{Acknowledgement} The authors thank the anonymous referee for a careful reading of the manuscript and for providing numerous insightful suggestions, which have greatly improved the clarity and exposition of the article. The authors also thank Anna Skripka, Fedor Sukochev, and Eva-Maria Hekkelman for essential remarks and stimulating discussions. A. Chattopadhyay is supported by the Core Research Grant (CRG), File No: CRG/2023/004826, by SERB, DST, Govt. of India. C. Pradhan acknowledges support from NBHM post-doc fellowship, Govt. of India. T.D.H. van Nuland was supported in part by ARC (Australian Research Council) grant FL17010005, and in part by the NWO project ‘Noncommutative multi-linear harmonic analysis and higher order spectral shift’, OCENW.M.22.070.

		\section{Preliminaries}\label{sec:prel}
		\paragraph{Notations and conventions}
		
		We write $\N=\{1,2,...\}$. We let $\H$ be a separable Hilbert space, and $\mB(\H)$ the bounded linear operators on $\H$. For $p\in[1,\infty)$, we let $\S^p=\S^p(\H)$ denote the Banach space of Schatten $p$-class operators, with the Schatten norm $\|\cdot\|_p$. We use the convention that $\Scal^\infty=\Scal^\infty(\Hcal)$ denotes the set of all compact operators on $\Hcal$ with the usual operator norm $\|A\|_\infty=\|A\|$. By \enquote{$H$ is self-adjoint in $\H$} we mean $H$ is a self-adjoint possibly unbounded operator on a (dense) domain $\dom(H)\subseteq\H$.
		We use the convention that $(-\frac{1}{a},\frac{1}{a})=(-\infty,\infty)$ if $a=0$.
		We define $u:\R\to\C$ by
		$$u(x):=x-i.$$
		%
		%
		Let $X$ be an interval in $\mathbb{R}$ possibly coinciding with $\mathbb{R}$. Let $C(X)$ denote the space of all continuous functions on $X$, $C_0(\R)$ the space of continuous functions on $\R$ decaying to $0$ at infinity, $C^n(X)$ the space of $n$-times continuously differentiable functions on $X$. Let $C_b^n(X)$ denote the subset of $C^n(X)$ of such $f$ for which $f^{(n)}$ is bounded. Let $C_c^n(\mathbb{R})$ denote the subspace of $C^n(\mathbb{R})$ consisting of compactly supported functions.
		We also use the notation $C^0(\mathbb{R})=C(\mathbb{R})$. Let $C_c^n(X)$ denote the subspace of $C_c^n(\R)$ consisting of the functions whose closed support is contained in $X$.
		For $p\in[1,\infty]$, let $L^p(\R)$ denote the Lebesgue $L^p$-space with usual norm $\|f\|_p=\|f\|_{L^p(\R)}$.
		Let $L^1_{\text{loc}}(\R)$ denote the space of functions locally integrable on $\R$ equipped with the seminorms $f\mapsto \int_{-a}^a|f(x)|\,dx$, $a>0$. For any $f\in L^1(\R)$, $\hat{f}$ denotes the Fourier transform with convention $f(x)=\int_\R \hat{f}(y)e^{ixy}dx$. We canonically extend this Fourier transform to tempered distributions, though any distribution that we shall encounter in this paper will be canonically represented by a function or a measure. For a finite measure $\mu$ we write $\|\mu\|$ for its total variation, which coincides with the corresponding $L^1$-norm if $\mu$ is absolutely continuous.

		\subsection{New function spaces}

		Let $n\in\Z_{\geq0}$.
		We introduce the following function spaces:
		\begin{align*}
			W_0(\R):=&\{f\in C_b(\R):~f(x)=\int_\R e^{ixy}d\mu(y)\textnormal{ for a finite complex measure $\mu$ on $\R$}\},\\
			\mW_n(\R):=&\{f\in C^n(\R):~(fu^p)^{(p)}\in W_0(\R),~p=0,\ldots,n\}.
		\end{align*}
		
		For each $k\leq n$ we also introduce the auxiliary space:
		\begin{align*}
			\mW^{n}_{k}(\R)&=\{f\in C^{n}(\R):~(fu^{p})^{(n-k+p)}\in W_0(\R),p=0,\ldots,k\}\\
			&=\{f\in C^{n}(\R):~(fu^{k-n+p'})^{(p')}\in W_0(\R),p'=n-k,\ldots,n\}.
		\end{align*}
		
		The following lemma shows that test functions such as rational functions, Schwartz class functions, and $C_c^{n+1}(\R)$ are contained in the newly introduced function classes. 
		\begin{lem}
			Let $n\in\N$. Then, the following assertions hold.
			\begin{enumerate}[label=\textnormal{(\alph*)}]
				\item
				For every $\alpha>\frac12$,
				\begin{align*}
					\left\{f\in C^{n+1}(\R):~ f^{(p)}(x)=\O\left(|x|^{-p-\alpha}\right)
					\text{ as } |x|\to\infty,\;p=0,\ldots,n+1\right\}\subseteq \mW_n(\R).
				\end{align*}
				
				\item\label{inclusionsii}
				For each $k\leq n$ and for every $\alpha>\frac12$,
				\begin{align*}
					\left\{f\in C^{n+1}(\R):~ f^{(p)}(x)=\O\left(|x|^{-p+n-k-\alpha}\right)
					\text{ as } |x|\to\infty,\;p=n-k,\ldots,n+1\right\} \subseteq \mW^{n}_{k}(\R).
				\end{align*}
			\end{enumerate}
		\end{lem}
		\begin{proof}
			By \cite[Lemma~7]{PoSu_Crel09}, we have the implications
			\begin{align}\label{eq:FT in L1}
				f\in L^2(\R)\cap C^1(\R),~f'\in L^2(\R)\quad\Rightarrow\quad\hat{f}\in L^1(\R)\quad\Rightarrow\quad f\in W_0(\R),
			\end{align}
			from which the lemma follows after applying the repeated Leibniz rule.
		\end{proof}
		
		Let us collect some further properties of these function spaces.
		\begin{lem}\label{lem:function spaces inclusions}
			For $n,k\in \N$, $k\leq n$, $p\in\Z_{\geq0}$, the following holds.
			\begin{enumerate}[label=\textnormal{(\alph*)}]
				\item\label{item:2 function spaces inclusions} If $f\in\mW_k^n(\R)$ then $fu\in\mW_{k-1}^n(\R)$ and $f\in\mW_{k-1}^{n-1}(\R)$.
				\item\label{item:1 function spaces inclusions} If $(fu^p)^{(n)}\in W_0(\R)$ and $(fu^{p+1})^{(n+1)}\in W_0(\R)$ then $(fu^p)^{(n+1)}\in W_0(\R)$.
				\item\label{item:3 function spaces inclusions} We have $\mW_n(\R)=\{f\in C^n(\R):~(fu^p)^{(m)}\in W_0(\R)~\text{ for }~0\leq p\leq m\leq n\}.$
			\end{enumerate}
		\end{lem}
		\begin{proof}
			The first statement follows by definition of $\mW_k^n(\R)$. 
			
			We proceed to the second statement.
			As $u'=1$, $u''=0$, we have
			$$(fu^{p+1})^{(n+1)}=(fu^p)^{(n+1)}u+(n+1)(fu^p)^{(n)}.$$
			Therefore,
			\begin{align}\label{eq:derivatives cov}
				(fu^p)^{(n+1)}=(fu^{p+1})^{(n+1)}u^{-1}-(n+1)(fu^p)^{(n)}u^{-1}.
			\end{align}
			The Fourier transform of $u^{-1}$ is in $L^1(\R)$ because $u^{-1},(u^{-1})'\in C^\infty(\R)\cap L^2(\R)$ and we may apply \eqref{eq:FT in L1}. Hence, $u^{-1}\in W_0(\R)$, and because $W_0(\R)$ is an algebra (the convolution of complex measures is a complex measure), the second statement of the lemma follows from \eqref{eq:derivatives cov}. 
			
			The third statement is a straightforward consequence of the second.
		\end{proof}
		
		\subsection{Relative boundedness}
		
		
		\begin{defn}\label{defn:relatively bounded}
			Let $H$ be self-adjoint in $\H$. A linear operator $V$ with domain $\dom V\subseteq\H$ is called (relatively) $H$-bounded if $\dom H\subseteq \dom V$ and there exist $a,b\in[0,\infty)$ such that for all $\psi\in\dom H$ we have
			$$\norm{V\psi}\leq a\norm{H\psi}+b\norm{\psi}.$$
			The infimum over such numbers $a$ is called the $H$-bound of $V$.
		\end{defn}
		
		We collect the following properties of relatively bounded operators.
		\begin{lem}\label{lem:properties relatively bdd}
			Let $H$ be self-adjoint in $\H$.
			\begin{enumerate}[label=\textnormal{(\alph*)}]
				\item\label{item:properties rel bdd: norm} If $V$ is relatively $H$-bounded with $H$-bound $a$, then $V(H-i)^{-1}:\H\to\H$ is bounded, and $a\leq\|V(H-i)^{-1}\|\leq a+b$.
				\item\label{item:properties rel bdd: scaling} If $V$ is relatively $H$-bounded with $H$-bound $a$, then for all $z\in\C$ the operator $zV$ is relatively $H$-bounded with $H$-bound $|z|a$.
				\item\label{item:properties rel bdd: vector space} The space of relatively $H$-bounded operators is a $\C$-vector space and a left $\mB(\H)$-module for the canonical addition, scalar multiplication, and multiplication of operators in $\H$.
				\item\label{item:properties rel bdd: Kato-Rellich} If $V$ is symmetric and relatively $H$-bounded with $H$-bound $a$ then 
				\begin{align*}
					H+tV\quad\text{is self-adjoint on}\quad\dom(H+tV)=\dom(H)
				\end{align*}
				for all $t\in(-\tfrac1a,\tfrac1a)$.
				\item\label{item:properties rel bdd: shift} If $V$ is symmetric and relatively $H$-bounded with $H$-bound $a$, then $V$ is relatively $H+tV$-bounded for all $t\in(-\frac{1}{a},\frac{1}{a})$ with 
				$H+tV$-bound $\leq\frac{a}{1-|t|a}$. Moreover, $\|V(H+tV-i)^{-1}\|\leq\frac{a+b}{1-|t|a}$. 
			\end{enumerate}
		\end{lem}
		\begin{proof}
			Statements \ref{item:properties rel bdd: norm}, \ref{item:properties rel bdd: scaling}, and \ref{item:properties rel bdd: vector space} follow directly from Definition \ref{defn:relatively bounded}. Statement \ref{item:properties rel bdd: Kato-Rellich} follows from the Kato--Rellich theorem \cite[Theorem V-4.3]{Kato} (originally proved by Rellich) combined with \ref{item:properties rel bdd: norm}. From the triangle inequality applied to $\|V\psi\|\leq a\|(H+tV)\psi-tV\psi\|+b\|\psi\|$, it follows that
			$$\|V\psi\|\leq\frac{a}{1-a|t|}\|(H+tV)\psi\|+\frac{b}{1-a|t|}\|\psi\|.$$
			The final statement of \ref{item:properties rel bdd: shift} thus follows from \ref{item:properties rel bdd: norm} which by \ref{item:properties rel bdd: Kato-Rellich} applies to $H+tV$ instead of $H$.
		\end{proof}
		We shall make plenty use of the second resolvent identity, which conveniently also holds in the relatively bounded setting. See also the version for closed operators, cf. \cite[Lemma 2]{Clark}.
		\begin{lem}[second resolvent identity]\label{lem:second resolvent identity}
			Let $H$ be self-adjoint and let $V$ be symmetric and relatively $H$-bounded with $H$-bound $<1$. For each $z\in\C$ with $Im~z\neq 0$, we have
			$$(H+V-z)^{-1}-(H-z)^{-1}=-(H+V-z)^{-1}V(H-z)^{-1} .$$
		\end{lem}
		\begin{proof}
			After noting that $H+V$ is self-adjoint on $\dom(H+V)=\dom(H)$ (see Lemma \ref{lem:properties relatively bdd}\ref{item:properties rel bdd: Kato-Rellich}), the proof is the same as in the bounded case.
		\end{proof}
		
		\section{Multiple operator integrals with relatively bounded arguments}\label{sec:moi}
		We generalize the multiple operator integral to relatively bounded arguments, and analyze its algebraic and analytical properties.
		The construction is inspired by, and compatible with, \cite{HMvN} and \cite{AP}.

		We first recall the definition of the multiple operator integral for bounded arguments, as given in \cite{ACDS,Peller}. For an overview, see \cite{ST19}.\begin{defn}\label{def:MOI}
			Let $n\in\N$, let $H_0,\ldots,H_n$ be self-adjoint in $\mathcal{H}$, and let $V_1,\ldots,V_n\in\mB(\H)$. For a function $\phi:\R^{n+1}\to\C$, we write $\phi\in\BS(\R^{n+1})$ (or $\phi\in\BS$ for short) if there exist a finite\footnote{It is equivalent to use $\sigma$-finite measure spaces under a condition on the norms of $a_i(\cdot,\omega)$, see \cite[p.48]{ST19} (and \cite[Lemma 2.5]{HMvN}).} measure space $(\Omega,d\omega)$ and bounded measurable functions $a_0,\ldots,a_n:\R\times\Omega\to\C$ such that
			\begin{align}\label{divdef1b}
				\phi(\lambda_0,\ldots,\lambda_n)=\int_\Omega a_0(\lambda_0,\omega)\cdots a_n(\lambda_n,\omega)\,d\omega.
			\end{align}
			If $\phi$ is of the above form, then the multiple operator integral (MOI) $T_\phi^{H_0,\ldots,H_n}$ is defined by
			$$T_\phi^{H_0,\ldots,H_n}(V_1,\ldots,V_n)\psi:=\int_\Omega a_0(H_0,\omega) V_1 a_1(H_1,\omega)\cdots V_n a_n(H_n,\omega)\psi\,d\omega\qquad (\psi\in\H),$$
			in which the right-hand side is a Bochner integral.
		\end{defn}
		It follows that the operator $T_\phi^{H_0,\ldots,H_n}:\mB(\H)\times\cdots\times\mB(\H)\to\mB(\H)$ is well-defined and bounded.
		Importantly, the operator $T_\phi^{H_0,\ldots,H_n}$ depends only on $\phi$, not on the measure space $(\Omega,d\omega)$ or the functions $a_j$ \cite[Lemma 4.3]{ACDS}.
		The following extension of the MOI to unbounded arguments is simple but surprisingly powerful, and will be the main tool of this paper.
		
		\begin{defn}\label{def:MOI relatively bdd}
			Let $n\in\Z_{\geq0}$, let $H_0,\ldots,H_n$ be self-adjoint in $\H$, and let $V_j$ be relatively $H_j$-bounded for $j=1,\ldots,n$. Let $\alpha=(\alpha_1,\ldots,\alpha_n)\in\Z_{\geq0}^n$ be such that $V_j(H_j-i)^{-\alpha_j}\in\mB(\H)$, and define $u^\alpha:=u_1^{\alpha_1}\cdots u_n^{\alpha_n}$, where $u_j(\lambda_0,\ldots,\lambda_n):=\lambda_j-i$. For any measurable $\phi:\R^{n+1}\to\C$ such that $\phi u^\alpha\in\BS(\R^{n+1})$ we define the multiple operator integral with relatively bounded arguments as
			\begin{align}\label{eq:def MOI relatively bdd}
				\bT_\phi^{H_0,\ldots,H_n}(V_1,\ldots,V_n):=T^{H_0,\ldots,H_n}_{\phi u^\alpha}(V_1(H_1-i)^{-\alpha_1},\ldots,V_n(H_n-i)^{-\alpha_n}).
			\end{align}
		\end{defn}
		\begin{lem}\label{thm:well defined}
			Definition \ref{def:MOI relatively bdd} is well-defined, and for all $V_1,\ldots,V_n\in\mB(\H)$ we have
			$$\bT^{H_0,\ldots,H_n}_\phi(V_1,\ldots,V_n)=T^{H_0,\ldots,H_n}_\phi(V_1,\ldots,V_n),$$
			where $n\in\N$, $H_0,\ldots,H_n$ are self-adjoint, and $\phi\in\BS(\R^{n+1})$.
		\end{lem}
		\begin{proof}
			Let $\alpha,\beta\in\Z_{\geq0}^n$ be such that $V_j(H_j-i)^{-\alpha_j}$ is bounded for all $j$. Then, $V_j(H_j-i)^{-\alpha_j-\beta_j}$ is bounded as well. Moreover, it follows from Definition \ref{def:MOI} that
			\begin{align*}
				T_{\phi u^{\alpha+\beta}}^{H_0,\ldots,H_n}(V_1(H_1-i)^{-\alpha_1-\beta_1},\ldots,V_n(H_n-i)^{-\alpha_n-\beta_n})=T_{\phi u^{\alpha}}^{H_0,\ldots,H_n}(V_1(H_1-i)^{-\alpha_1},\ldots,V_n(H_n-i)^{-\alpha_n}).
			\end{align*}
			Hence, the right-hand side of \eqref{eq:def MOI relatively bdd} is independent from $\alpha$, yielding the first statement of the lemma. Taking $\alpha=(0,\ldots,0)$ yields the second.
		\end{proof}
		An important special case of Definition \ref{def:MOI relatively bdd} is obtained when the symbol $\phi$ is a divided difference $\phi=f^{[n]}$ of a function $f\in C^n(\R)$. The divided difference is defined recursively by
		\begin{align*}
			f^{[0]}(\lambda)&:=f(\lambda),\\
			f^{[n]}(\lambda_0,\ldots,\lambda_n)&:=\frac{f^{[n-1]}(\lambda_0,\ldots,\lambda_{n-1})-f^{[n-1]}(\lambda_1,\ldots,\lambda_n)}{\lambda_0-\lambda_n},
		\end{align*}
		where the fraction is continuously extended when $\lambda_0=\lambda_n$. 
		We shall moreover use the following alternate representation of $f^{[n]}$.
		\begin{lem}\label{lem:div_rep} Let $f\in C^n(\R)$ such that $f^{(n)}\in W_0(\R)$, with $f^{(n)}(\lambda)=\int e^{ix\lambda}d\mu(x)$. Then, for all $\lambda_0,\ldots,\lambda_n \in \R,$
			\begin{align*}
				f^{[n]}(\lambda_0,\ldots,\lambda_n)= \int_\R\int_{\Delta_n} e^{is_0x\lambda_0}e^{is_1x\lambda_1}\cdots e^{is_nx\lambda_n}\;  ds\, d\mu(x),
			\end{align*}
			where the simplex $\Delta_n=\{s=(s_0,\ldots,s_n)\in\R^{n+1}_{\geq0}~:~s_0+\cdots+s_n=1\}$ is endowed with the flat measure with total measure $1/n!$. Clearly, $\left(\R\times\Delta_n,  d\mu(x)\times ds\right)$ is a finite measure space, hence $f^{[n]}\in\BS$.
		\end{lem}
		\begin{proof}
			The proof follows directly from the arguments presented in \cite[Lemmas 5.1 and 5.2]{PSS}.
		\end{proof}
		\begin{cor}\label{cor:bd_moi_bdd}
			Let $n\in\N$, $H_0,\ldots,H_n$ be self-adjoint operators in $\mathcal H$ and let $f\in C^n(\R)$ be such that $f^{(n)}\in W_0(\R)$, $f^{(n)}(\lambda)=\int e^{ix\lambda}d\mu(x)$. 
			For all $V_1,\ldots,V_n\in\mathcal B(\mathcal H)$ and all $\psi\in\mathcal H$ we have
			$$T^{H_0,\ldots,H_n}_{f^{[n]}}(V_1,\ldots,V_n)\psi=\int_\R\int_{\Delta_n}e^{is_0xH_0}V_1e^{is_1xH_1}\cdots V_ne^{is_nxH_n}\psi\,ds\,d\mu(x).$$
			Consequently, for all $\alpha,\alpha_j\in[1,\infty]$ with $\tfrac{1}{\alpha_1}+\cdots+\tfrac{1}{\alpha_n}=\tfrac{1}{\alpha}$,
			\begin{align}
				\label{fourierbound}
				\big\|T_{f^{[n]}}^{H_0,\dots,H_n}\big\|_{\mathcal{S}^{\alpha_1}\times\cdots\times\mathcal{S}^{\alpha_n}\to\mathcal{S}^\alpha}
				\le \frac{1}{n!}\big\|\mu\big\|.
			\end{align}
		\end{cor}
		\begin{proof}
			Combining Definition~\ref{def:MOI} with Lemma~\ref{lem:div_rep}, and using $\phi = f^{[n]}$, the corollary follows.
		\end{proof}
		
		The following theorem explains the power of Definition \ref{def:MOI relatively bdd}.
		
		\begin{thm}\label{thm:MOI well-defined special case}
			For $n\in\N$, $H_0,\ldots,H_n$ self-adjoint in $\H$, and $f\in\mW_n(\R)$, we have $f^{[n]}u^{(1,\ldots,1)}\in\BS(\R^{n+1})$, and hence Definition \ref{def:MOI relatively bdd} defines a multilinear map
			$$\bT^{H_0,\ldots,H_n}_{f^{[n]}}:\mathbf X_{H_1}\times\cdots \times \mathbf X_{H_n}\to \mB(\H),$$
			where $\mathbf X_H$ is the space of $H$-bounded operators (which is a linear space by Lemma \ref{lem:properties relatively bdd}\ref{item:properties rel bdd: vector space}). More generally, if $\alpha=(\alpha_1,\ldots,\alpha_n)\in\{0,1\}^n$ is such that $V_j\in\mB(\H)$ if $\alpha_j=0$, and $f\in\mW^n_{|\alpha|}(\R)$, then $f^{[n]}u^\alpha\in\BS(\R^{n+1})$, and hence Definition \ref{def:MOI relatively bdd} defines a multilinear map
			$$\bT_{f^{[n]}}^{H_0,\ldots,H_n}:\mathbf X_{H_1}^{\alpha_1}\times\cdots\times\mathbf X_{H_n}^{\alpha_n}\to\mB(\H),$$
			where $\mathbf X^1_{H}=\mathbf X_{H}$ and $\mathbf X^0_{H}=\mB(\H)$.
		\end{thm}
		\begin{proof}
			The first statement of the theorem follows by combining Lemma \ref{lem:div_rep} with the formula
			\begin{align}\label{formul:div_diff}
				&f^{[n]}(\lambda_0,\ldots,\lambda_n)=\sum_{p=0}^n(-1)^{n-p}\sum_{0<j_1<\cdots<j_{p}\leq n}
				(fu^p)^{[p]}(\lambda_0,\lambda_{j_1},\ldots,\lambda_{j_p})
				u^{-1}(\lambda_1)\cdots u^{-1}(\lambda_n),
			\end{align}
			which was shown in \cite[Eq. (24)]{vNS22}. Let $\mJ=\{j\in\{1,\ldots,n\}~:~\alpha_j=1\}$, and $\mJ^\text{c}=\{1,\ldots,n\}\setminus\mJ$. The second statement of the theorem follows by combining Lemma \ref{lem:div_rep} with the formula
			\begin{align*}
				&f^{[n]}(\lambda_0,\ldots,\lambda_n)=\sum_{p=n-|\alpha|}^n(-1)^{n-p}\sum_{\substack{0<j_1<\cdots<j_{p}\leq n\\ \{j_1,\ldots,j_p\}\supseteq\mJ^{\text{c}}}}
				(fu^{-n+|\alpha|+p})^{[p]}(\lambda_0,\lambda_{j_1},\ldots,\lambda_{j_p})
				\prod_{j\in\mJ}u^{-1}(\lambda_j),
			\end{align*}
			which can be shown in a similar way (see \cite[eq. (21)]{vNS23}).
		\end{proof}
		
		\subsection{Perturbation formulas for the generalised multiple operator integral}
		\label{sct:subsct:properties MOI}
		
		It turns out that several useful identities of the multiple operator integral extend to the case of relatively bounded arguments. In order not to interrupt the flow of the paper, their proofs are in the appendix.
		
		The first such identity is a change of variables rule which adds resolvents of the superscript operators to the arguments. This identity forms a key step in deriving summability estimates for multiple operator integrals, as witnessed in the bounded case by \cite{CS18,vNS22,vNS23,vNvS21a}, and in the relatively bounded case in Sections \ref{sec:operator differ} and \ref{sct:Spectral shift functions}.
		\begin{thm}[change-of-variables]\label{thm:cov rel bdd}
			Let $H_0,\ldots,H_n$ be self-adjoint in $\H$. Let $\mJ\subseteq\{1,\ldots,n\}$ be a subset so that $V_k$ is bounded for each $k\in\{1,\ldots,n\}\setminus\mJ$ and $V_k$ is relatively $H_k$-bounded for each $k\in\mJ$. For each $f\in\mW_{|\mJ|}^n(\R)$ and each $j\in\{0, 1,\ldots,n\}$ (the boundary cases $j=0$ and $j=n$ being understood as in Theorem \ref{thm:cov}) we have
			\begin{align*}
				\bT^{H_0,\ldots,H_n}_{f^{[n]}}(V_1,\ldots,V_n)=&\bT^{H_0,\ldots,H_n}_{(fu)^{[n]}}(V_1,\ldots,V_{j-1},V_j(H_j-i)^{-1},V_{j+1},\ldots,V_n)\\
				&-\bT^{H_0,\ldots,H_{j-1},H_{j+1},\ldots,H_n}_{f^{[n-1]}}(V_1,\ldots,V_{j-1},V_j(H_j-i)^{-1}V_{j+1},V_{j+2},\ldots,V_n).
			\end{align*}
		\end{thm}
		\begin{proof}
			Follows from Theorem \ref{thm:cov rel bdd2} in the appendix.
		\end{proof}
		
		By repeating the proof of the above change of variables rule, one expands a generalised multiple operator integral into a finite sum of multiple operator integrals with bounded arguments. This in particular generates a useful norm bound of the (generalised) multiple operator integral.
		\begin{thm}\label{thm:bT in MOIs and bound}
			Let $n\in\N$, let $H_0,\ldots,H_n$ be self-adjoint in $\H$. For each $j\in\{1,\ldots,n\}$, let $V_j$ be a relatively $H_j$-bounded operator. Let $f\in \mW_n(\R)$ (i.e., $f\in C^n(\R)$ such that $(fu^p)^{(p)}\in W_0(\R)$ for all $p\in\{0,\ldots,n\}$). Writing $\tilde V_{j,l}:=V_{j+1}(H_{j+1}-i)^{-1}\cdots V_l(H_l-i)^{-1}\in\mB(\H)$, we have
			\begin{align*}
				\bT^{H_0,\ldots,H_n}_{f^{[n]}}(V_1,\ldots,V_n)=\sum_{p=0}^n(-1)^{n-p}\sum_{0<j_1<\cdots<j_p\leq n} T^{H_0,H_{j_1},\ldots,H_{j_p}}_{(fu^p)^{[p]}}(\tilde V_{0,j_1},\ldots,\tilde V_{j_{p-1},j_p})\tilde V_{j_p,n}\,.
			\end{align*} 
			Here, the case $p = 0$ on the right-hand side of the above identity is understood as $f(H_0)\tilde{V}_{0,n}$. Therefore, if $\widehat{(fu^p)^{(p)}}\in L^1(\R)$ we have
			\begin{align*}
				\norm{\bT^{H_0,\ldots,H_n}_{f^{[n]}}(V_1,\ldots,V_n)}\leq \sum_{p=0}^n\vect{n}{p}\frac{1}{p!}\absnorm{\widehat{(fu^p)^{(p)}}}\prod_{j=1}^n\norm{V_j(H_j-i)^{-1}},
			\end{align*}
			and if $\widehat{(fu^p)^{(p)}}\notin L^1$ then $\|\widehat{(fu^p)^{(p)}}\|_1$ may be replaced by $\|\mu_p\|$, where $(fu^p)^{(p)}(x)=\int e^{ixy}d\mu_p(y)$.
		\end{thm}
		\begin{proof}
			The first identity follows from repeated application of Theorem \ref{thm:cov rel bdd} and Lemma \ref{thm:well defined}. The bound follows from Corollary \ref{cor:bd_moi_bdd}.
		\end{proof}
		
		The following superscript difference identity is crucial for calculating derivatives and Taylor series. The proof is surprisingly subtle.
		\begin{thm}\label{thm:superscript difference}
			Let $n\in\N$, let $H_0,\ldots,H_n$ be self-adjoint in $\H$, and let $f\in\mW_{n+1}(\R)$. For each $j\in\{1,\ldots,n\}$, let $V_j$ be a $H_j$-bounded symmetric operator with $H_j$-bound $a$. For all $t\in(-\tfrac1a,\tfrac1a)$ and $j\in\{1,\ldots,n\}$ we have
			\begin{align*}
				& \bT^{H_0,\ldots,H_{j-1},H_j+tV_j,H_{j+1},\ldots,H_n}_{f^{[n]}}(V_1,\ldots,V_n)-\bT^{H_0,\ldots,H_n}_{f^{[n]}}(V_1,\ldots,V_n)\\
				&\hspace{90pt}=t\bT^{H_0,\ldots,H_{j-1},H_j+tV_j,H_j,\ldots,H_n}_{f^{[n+1]}}(V_1,\ldots,V_{j},V_j,\ldots,V_n).
			\end{align*}
		\end{thm}
		\begin{proof}
			This is a special case of Theorem \ref{thm:superscript difference bT} in the appendix.
		\end{proof}

		\section{Operator differentiation and Taylor series}\label{sec:operator differ}
		
		\subsection{First-order differentiation}
		\label{First-order differentiation}
		
		In this subsection, we shall prove the following first-order differentiability result.
		
		\begin{thm}\label{thm:derivative}
			Let $H$ be self-adjoint in (the separable Hilbert space) $\H$, let $V$ be symmetric and relatively $H$-bounded, and let $u(x):=x-i$. Let $f\in C^1(\R)$ be such that $f,(fu)'\in W_0(\R)$. Then $t\mapsto f(H+tV)$ is Gateaux differentiable at $t=0$ in strong operator topology, and its Gateaux derivative equals
			\begin{align}\label{1st_der_formula}
				\frac{d}{dt} f(H+tV)\big|_{t=0}=\bT_{f^{[1]}}^{H, H}(V)=T^{H,H}_{(fu)^{[1]}}(V(H-i)^{-1})-f(H)V(H-i)^{-1}.
			\end{align}
			If, in addition, $V(H-i)^{-1}\in\S^\infty$, then $t\mapsto f(H+tV)$ is Gateaux differentiable in $t=0$ in operator norm topology, and its Gateaux derivative is given by \eqref{1st_der_formula}.
		\end{thm}
		
		
		The first step in the proof of Theorem \ref{thm:derivative} is the following adaptation of the Duhamel formula.
		\begin{lem}[Weighted Duhamel formula]\label{lem:Duhamel}
			Let $H$ be self-adjoint in $\H$, and let $V$ be symmetric and relatively $H$-bounded with $H$-bound $<1$. For all $x\in\R$ and all $\psi\in\H$ we have
			\begin{align*}
				e^{ix(H+V)}(H-i)^{-1}\psi-e^{ixH}(H-i)^{-1}\psi=ix\int_0^1 e^{isx(H+V)}V(H-i)^{-1}e^{i(1-s)xH}\psi\,ds,
			\end{align*}
			where the integral on the right-hand side is a Bochner integral. Equivalently, for all $\varphi\in\dom H$ we have
			\begin{align*}
				e^{ix(H+V)}\varphi-e^{ixH}\varphi=ix\int_0^1 e^{isx(H+V)}Ve^{i(1-s)xH}\varphi\,ds.
			\end{align*}
			
		\end{lem}
		\begin{proof}
			The proof looks similar to the proof of \cite[Lemma 5.2]{ACDS}, but needs some careful adjustments. Let $\psi\in\H$ be arbitrary. The function $G:\R\to\H$ defined by
			$$G(s):=e^{isx(H+V)}e^{i(1-s)xH}(H-i)^{-1}\psi$$
			is continuous by Stone's theorem. Moreover, for $\varphi:=(H-i)^{-1}\psi$, we have $e^{i(1-s)xH}\varphi=(H-i)^{-1}e^{i(1-s)xH}\psi\in\ran(H-i)^{-1}=\dom H\subseteq\dom V$. Therefore, we obtain
			\begin{align}
				&\lim_{t\to0}\left(\frac{G(s+t)-G(s)}{t}\right)\nonumber\\
				=&\,\lim_{t\to0}\left(e^{i(s+t)x(H+V)}\frac{\left(e^{i(1-s-t)xH}-e^{i(1-s)xH}\right)}{t}\varphi+\frac{\left(e^{i(t+s)x(H+V)}-e^{isx(H+V)}\right)}{t}\left(e^{i(1-s)xH}\varphi\right)\right)\nonumber\\
				=&\,-ixe^{isx(H+V)}He^{i(1-s)xH}\varphi+ixe^{isx(H+V)}(H+V)e^{i(1-s)xH}\varphi\nonumber\\
				=&\,ixe^{isx(H+V)}V(H-i)^{-1}e^{i(1-s)xH}\psi,\label{eq:derivative of G}
			\end{align}
			by using Stone's theorem again and the fact that $e^{i(s+t)x(H+V)}$ is uniformly bounded by $1$. (Indeed, to define this exponential we have already tacitly used the fact that $H+V$ is self-adjoint by
			Lemma~\ref{lem:properties relatively bdd}(d).) By \eqref{eq:derivative of G}, the function $g:\R\to\H$ defined by
			\begin{align*}
				g(s):=ixe^{isx(H+V)}V(H-i)^{-1}e^{i(1-s)xH}\psi
			\end{align*}
			is the derivative of $G$. As $V(H-i)^{-1}$ is bounded, it follows that $g$ is continuous -- without the right multiplication of $(H-i)^{-1}$, this argument would fail. The $\H$-valued fundamental theorem of calculus states
			\begin{align*}
				G(1)-G(0)=\int_0^1 g(s)\,ds,
			\end{align*}
			where the right-hand side is a Bochner integral. By the definitions of $G$ and $g$, we obtain the lemma.
		\end{proof}
		
		\begin{lem}\label{lem:cor of Duhamel}
			Let $H$ be self-adjoint in $\H$, and let $V$ be symmetric and relatively $H$-bounded with $H$-bound $<1$. Let $f\in C^1(\R)$ with $f,f'\in W_0(\R)$, and let $\mu_{\hat{f'}}$ be the measure of which $f'$ is the inverse Fourier transform. For all $\psi\in\H$ we have
			$$f(H+V)(H-i)^{-1}\psi-f(H)(H-i)^{-1}\psi=\int_\R \int_0^1 e^{isx(H+V)}V(H-i)^{-1}e^{i(1-s)xH}\psi\,ds\,d\mu_{\hat{f'}}(x)$$
		\end{lem}
		\begin{proof}
			The statement follows from Lemma \ref{lem:Duhamel}, the dominated convergence theorem, and the fact that $ix\hat{f}(x)=\widehat{f'}(x)$ (when $\hat{f},\widehat{f'}$ are functions, and similarly when they are measures).
		\end{proof}
		
		We may write the above formula in the more convenient MOI notation as follows.
		\begin{lem}\label{lem:cor of cor of Duhamel}
			Let $H$ be self-adjoint in $\H$, and let $V$ be symmetric and relatively $H$-bounded with $H$-bound $<1$. For all $f\in C^1(\R)$ with $f,f'\in W_0(\R)$ we have
			\begin{align}\label{eq:cor of Duhamel}
				\left(f(H+V)-f(H)\right)(H-i)^{-1}=T^{H+V,H}_{f^{[1]}}(V(H-i)^{-1}).
			\end{align}
		\end{lem}
		\begin{proof}
			This is simply Lemma \ref{lem:cor of Duhamel} combined with the definition of $T^{H+V,H}_{f^{[1]}}(V(H-i)^{-1})\psi$, cf. Corollary \ref{cor:bd_moi_bdd}.
		\end{proof}

		\begin{thm}\label{thm:difference}
			Let $H$ be self-adjoint in $\H$, and let $V$ be symmetric and relatively $H$-bounded with $H$-bound $<1$. Let $f\in\mW_1(\R)$, that is, let $f\in C^1(\R)$ be such that $f,(fu)'\in W_0(\R)$. We have
			$$f(H+V)-f(H)=T^{H+V,H}_{(fu)^{[1]}}(V(H-i)^{-1})-f(H+V)V(H-i)^{-1}.$$
		\end{thm}
		\begin{proof}
			From $f,(fu)'\in W_0(\R)$ it follows that $f'\in W_0(\R)$, and so we may apply Lemma \ref{lem:cor of cor of Duhamel}.
			We now combine Lemma \ref{lem:cor of cor of Duhamel} with \cite[Theorem 3.10]{vNS22}, i.e., Theorem \ref{thm:cov} below for $j=n=1$. We obtain
			\begin{align*}
				(f(H+V)-f(H))(H-i)^{-1}&=T^{H+V,H}_{(fu)^{[1]}}(V(H-i)^{-2})-T^{H+V}_{f^{[0]}}()V(H-i)^{-2}\\
				&=\big(T^{H+V,H}_{(fu)^{[1]}}(V(H-i)^{-1})-f(H+V)V(H-i)^{-1}\big)(H-i)^{-1}.
			\end{align*}
			As $\ran (H-i)^{-1}=\dom H$, we obtain
			\begin{align}\label{eq:bdd ops acting on psi}
				(f(H+V)-f(H))\psi&=\big(T^{H+V,H}_{(fu)^{[1]}}(V(H-i)^{-1})-f(H+V)V(H-i)^{-1}\big)\psi
			\end{align}
			for all $\psi\in\dom H$. As $H$ is densely defined and the operators acting on $\psi$ on both sides of \eqref{eq:bdd ops acting on psi} are bounded, the theorem follows.
		\end{proof}
		

		We now prove our main theorem of this section (existence of first-order Gateaux derivative).
		\begin{proof}[Proof of Theorem \ref{thm:derivative}.]
			As $tV$ is relatively bounded with $H$-bound $<1$ for all $t\in (-\frac1a,\frac1a)$ (see Lemma \ref{lem:properties relatively bdd}) we may apply Theorem \ref{thm:difference} with $tV$ in place of $V$. We obtain
			\begin{align}\label{eq:difference quotient}
				\frac{f(H+tV)-f(H)}{t}=T^{H+tV,H}_{(fu)^{[1]}}(V(H-i)^{-1})-f(H+tV)V(H-i)^{-1}.
			\end{align}
			Firstly, we derive
			\begin{align*}
				\|f(H+tV)-f(H)\|&\leq |t|\|T^{H+tV,H}_{(fu)^{[1]}}(V(H-i)^{-1})\|+|t|\|f(H+tV)V(H-i)^{-1}\|\\
				&\leq |t|\left(\|\mu_1\|+\|f\|_\infty\right)\|V(H-i)^{-1}\|,
			\end{align*}
			which shows that
			\begin{align}\label{eq:norm convergence f(H+tV)}
				f(H+tV)\to f(H)
			\end{align}
			in norm as $t\to0$, where $(fu)'(x)=\int e^{ixy}d\mu_1(y)$. Moreover, we note for $\psi\in\H$ that
			$$e^{isx(H+tV)}V(H-i)^{-1}e^{i(1-s)xH}\psi\to e^{isxH}V(H-i)^{-1}e^{i(1-s)xH}\psi,$$
			for all $s\in[0,1]$ and all $x\in\R$. By the dominated convergence theorem, the above implies
			\begin{align}\label{eq:strong convergence DOI}
				T^{H+tV,H}_{(fu)^{[1]}}(V(H-i)^{-1})\psi\to T^{H,H}_{(fu)^{[1]}}(V(H-i)^{-1})\psi
			\end{align}
			for all $\psi\in\H$, i.e., convergence in the strong operator topology.
			If $V(H-i)^{-1}$ is compact, a similar argument shows that
			\begin{align}\label{eq:norm convergence DOI}
				T^{H+tV,H}_{(fu)^{[1]}}(V(H-i)^{-1})\to T^{H,H}_{(fu)^{[1]}}(V(H-i)^{-1})
			\end{align}
			in norm. Taking $t\to0$ in the formula \eqref{eq:difference quotient} and applying either \eqref{eq:norm convergence f(H+tV)} and \eqref{eq:strong convergence DOI} or \eqref{eq:norm convergence f(H+tV)} and \eqref{eq:norm convergence DOI} yields the theorem.
		\end{proof}

	}
	
	\subsection{Higher order differentiability}
	\begin{thm}\label{thm:higher derivatives in norm}
		Let $n\in\N$ and $f\in\mW_{n}(\R)$. The following holds.
		\begin{enumerate}[label=\textnormal{(\alph*)}]
			\item\label{item:n times diff at 0} For $H$ self-adjoint in $\H$ and $V$ symmetric and relatively $H$-bounded we have, in strong operator topology,
			\begin{align}\label{eq:nth deriv at 0}
				\frac{1}{n!}\frac{d^n}{dt^n} f(H+tV)\big|_{t=0}=\bT^{H,\ldots,H}_{f^{[n]}}(V,\ldots,V).
			\end{align}
			\item\label{item:n times diff at t0} For $H$ self-adjoint in $\H$ and $V$ symmetric and relatively $H$-bounded with $H$-bound $a\in[0,\infty)$, for all $t_0\in(-\tfrac1a,\tfrac1a)$ we have, in strong operator topology,
			\begin{align*}
				\frac{1}{n!}\frac{d^n}{dt^n} f(H+tV)\big|_{t=t_0}=\bT^{H+t_0V,\ldots,H+t_0V}_{f^{[n]}}(V,\ldots,V).
			\end{align*}
		\end{enumerate}
		Moreover, if we assume $V(H-i)^{-1}\in\S^\infty$, or we assume $f\in\mW_{n+1}(\R)$, then the above derivatives exist in operator norm topology.
	\end{thm}
	\begin{proof}
		From Lemma \ref{lem:properties relatively bdd}\ref{item:properties rel bdd: Kato-Rellich} and Lemma \ref{lem:properties relatively bdd}\ref{item:properties rel bdd: shift}, it follows that \ref{item:n times diff at 0} and \ref{item:n times diff at t0} are equivalent.
		The rest of the proof consists of showing that \ref{item:n times diff at 0} holds by induction to $n$. The base of the induction, that is, $n=1$, follows from Theorem \ref{thm:derivative}.
		
		Suppose, as an induction hypothesis, that  \ref{item:n times diff at 0} holds for a given $n\in\N$. It remains to show that \ref{item:n times diff at 0} holds when $n$ is replaced by $n+1$.
		Hence for each $f\in\mW_{n+1}(\R)$ we are to show the existence of, and compute, the strong operator (SOT) limit of the quotient
		\begin{align}
			&\frac{\frac{d^n}{ds^n} f(H+sV)\big|_{s=t}-\frac{d^n}{ds^n} f(H+sV)\big|_{s=0}}{t}\nonumber\\
			=&\,\frac{n!}{t}\left(\bT^{H+tV,\ldots,H+tV}_{f^{[n]}}(V,\ldots,V)-\bT^{H,\ldots,H}_{f^{[n]}}(V,\ldots,V)\right)\nonumber\\
			=&\,n!\sum_{k=0}^n\bT_{f^{[n+1]}}^{\overbrace{\scriptstyle H+tV,\ldots,H+tV}^{n-k+1\text{ times}},\overbrace{\scriptstyle H,\ldots,H}^{ k+1\text{ times}}}(V,\ldots,V),\label{eq:quotient}
		\end{align}
		where we have used \ref{item:n times diff at t0} in the first step, and used a telescoping sum together with Theorem \ref{thm:superscript difference} in the second step. 
		Next, we find the SOT-limit of the $0$th summand in \eqref{eq:quotient}, noting that the limit of the other summands can be found in a similar way. By Theorem \ref{thm:bT in MOIs and bound}, we have
		\begin{align}\label{represntation}
			\nonumber&\bT^{H+tV,\ldots,H+tV, H}_{f^{[n+1]}}(V,\ldots,V)\\
			\nonumber=&\sum_{p=0}^{n+1}(-1)^{n+1-p}\sum_{0<j_1<\cdots<j_p\leq n+1} T^{H_0(t),H_{j_1}(t),\ldots,H_{j_p}(t)}_{(fu^p)^{[p]}}(\tilde V_{0,j_1}(t),\ldots,\tilde V_{j_{p-1},j_p}(t))\tilde V_{j_p,n+1}(t)\\
			=&(-1)^{n+1}f(H+tV)\left(V(H+tV-i)^{-1}\right)^{n}\, \left(V(H-i)^{-1}\right)\\
			\nonumber&+\sum_{p=1}^{n+1}(-1)^{n+1-p}\sum_{0<j_1<\cdots<j_p\leq n+1} T^{H_0(t),H_{j_1}(t),\ldots,H_{j_p}(t)}_{(fu^p)^{[p]}}(\tilde V_{0,j_1}(t),\ldots,\tilde V_{j_{p-1},j_p}(t))\tilde V_{j_p,n+1}(t)
		\end{align}
		where $H_0(t)=\cdots=H_{n}(t)=H+tV$, $H_{n+1}(t)=H$, and $\tilde V_{j,l}(t):=V(H_{j+1}(t)-i)^{-1}\cdots V(H_l(t)-i)^{-1}\in\mB(\H)$. By Corollary \ref{cor:bd_moi_bdd}, we have
		\begin{align*}
			&T^{H_0(t),H_{j_1}(t),\ldots,H_{j_p}(t)}_{(fu^p)^{[p]}}(\tilde V_{0,j_1}(t),\ldots,\tilde V_{j_{p-1},j_p}(t))\\
			=& \int_\R\int_{\Delta_p} e^{ixs_0H_0(t)}\tilde V_{0,j_1}(t)\, e^{ixs_1H_{j_1}(t)}  \tilde V_{j_1, j_2}(t)\cdots\tilde V_{j_{p-1},j_p}(t)e^{ixs_p H_{J_p}(t)}\,ds\, d\mu_p(x),
		\end{align*}
		where the measure $\mu_p=\widehat{(fu^p)^{(p)}}$ is the distributional Fourier transform of $(fu^p)^{(p)}\in C_b(\R)\subseteq\S'(\R)$.
		Next we note the following facts:
		\begin{itemize}
			\item By Lemma \ref{lem:properties relatively bdd}\ref{item:properties rel bdd: shift}, $\|V(H+tV-i)^{-1}\|\leq 2(a+b)$ for all $t\in [-\frac{1}{2a},\frac{1}{2a}]$.
			\item By the second resolvent identity (Lemma \ref{lem:second resolvent identity}), for each $z\in\C $ with $Im~z\neq 0$, $t\mapsto (H+tV-z)^{-1}$, $t\mapsto V(H+tV-z)^{-1}$ are continuous on $[-\frac{1}{2a},\frac{1}{2a}]$ in operator norm, and \[\lim_{t\to 0}(H+tV-z)^{-1}=(H-z)^{-1}\quad\text{and }\quad \lim_{t\to 0}V(H+tV-z)^{-1}=V(H-z)^{-1}.\]
			\item For each fixed $x, s\in\R$, by \cite[Theorem VIII.20]{Simon}, $t\mapsto e^{ixs(H+tV)}$, $t\mapsto f(H+tV)$ are both continuous on $[-\frac{1}{2a},\frac{1}{2a}]$ in the strong operator topology, and 
			\[SOT-\lim_{t\to 0}e^{is(H+tV)}=e^{isH}\quad \text{ and }\quad SOT-\lim_{t\to 0}f(H+tV)=f(H).\]
			\item \begin{align}\label{sot_lim_1}
				\nonumber&SOT-\lim_{t\to 0} e^{ixs_0H_0(t)}\tilde V_{0,j_1}(t)\, e^{ixs_1H_{j_1}(t)} \tilde V_{j_1, j_2}(t)\cdots\tilde V_{j_{p-1},j_p}(t)e^{ixs_pH_{j_p}(t)}\\
				=&\, e^{ixs_0H}\tilde V_{0,j_1}(0)\, e^{ixs_1H} \tilde V_{j_1, j_2}(0)\cdots\tilde V_{j_{p-1},j_p}(0)e^{ixs_pH}
			\end{align}
			\item By Lemma \ref{lem:properties relatively bdd}\ref{item:properties rel bdd: shift}, $\|e^{ixs_0H_0(t)}\tilde V_{0,j_1}(t)\, e^{ixs_1H_{j_1}(t)} \tilde V_{j_1, j_2}(t)\cdots\tilde V_{j_{p-1},j_p}(t)e^{ixs_pH_{j_p}(t)}\|\leq \left(2(a+b)\right)^{j_1+\ldots+j_p}$ for all $t\in[-\frac{1}{2a},\frac{1}{2a}]$.
		\end{itemize}
		In conclusion, from the above, along with \cite[Corollary III.6.16]{DanSch}, we conclude that 
		\begin{align}\label{sot_lim_2}
			&SOT-\lim_{t\to 0}f(H+tV)\left(V(H+tV-i)^{-1}\right)^{n}\, \left(V(H-i)^{-1}\right)=f(H)\left(V(H-i)^{-1}\right)^{n+1}, \text{ and}\\[5pt]
			&SOT-\lim_{t\to 0}T^{H_0(t),H_{j_1}(t),\ldots,H_{j_p}(t)}_{(fu^p)^{[p]}}(\tilde V_{0,j_1}(t),\ldots,\tilde V_{j_{p-1},j_p}(t))=T^{H_0(0),H_{j_1}(0),\ldots,H_{j_p}(0)}_{(fu^p)^{[p]}}(\tilde V_{0,j_1}(0),\ldots,\tilde V_{j_{p-1},j_p}(0)).\label{sot_lim_3}
		\end{align}
		The above limits together with \eqref{represntation} imply that 
		\[SOT-\lim_{t\to 0}\bT^{H+tV,\ldots,H+tV, H}_{f^{[n+1]}}(V,\ldots,V)=\bT^{H,\ldots,H}_{f^{[n+1]}}(V,\ldots,V).\]
		By similar computations, we conclude that 
		\begin{align*}
			SOT-\lim_{t\to 0}\bT_{f^{[n+1]}}^{\overbrace{\scriptstyle H+tV,\ldots,H+tV}^{n-k+1\text{ times}},\overbrace{\scriptstyle H,\ldots,H}^{k+1\text{ times}}}(V,\ldots,V)=\bT^{H,\ldots,H}_{f^{[n+1]}}(V,\ldots,V)
		\end{align*}
		for every $0\leq k\leq n$. Therefore, from \eqref{eq:quotient}, we conclude that 
		\begin{align}\label{final_indu_step}
			\frac{1}{(n+1)!}\frac{d^{n+1}}{dt^{n+1}} f(H+tV)\big|_{t=0}=\bT^{H,\ldots,H}_{f^{[n+1]}}(V,\ldots,V)
		\end{align}
		in the strong operator topology. If we assume that $V(H-i)^{-1}$ is compact, then noting that the limits in \eqref{sot_lim_1}, \eqref{sot_lim_2}, and \eqref{sot_lim_3} exist in operator norm, we conclude that \eqref{final_indu_step} exists in operator norm.

		Next we show that, for $f \in \mW_{n+1}(\R)$, \eqref{eq:nth deriv at 0} exists in operator norm. By the same arguments as in the start of the proof, \ref{item:n times diff at 0} and \ref{item:n times diff at t0} are equivalent in this case as well. We shall prove \ref{item:n times diff at 0} by induction on $n$. The base case $n = 0$ is trivial. We shall show that \eqref{eq:nth deriv at 0} holds when $n$ is replaced by $n+1$. By using \eqref{eq:quotient}, introducing another telescoping sum, and applying Theorem \ref{thm:superscript difference} again, now using the fact that $f\in\mW_{n+2}(\R)$, we obtain
		\begin{align*}
			&\frac{\frac{d^n}{ds^n} f(H+sV)\big|_{s=t}-\frac{d^n}{ds^n} f(H+sV)\big|_{s=0}}{t}-(n+1)!\bT^{H,\ldots,H}_{f^{[n+1]}}(V,\ldots,V)\\
			=&\,n!\sum_{k=0}^n\sum_{l=0}^{n-k}\left(\bT^{\overbrace{\scriptstyle H+tV,\ldots,H+tV}^{l+1\text{ times}},\overbrace{\scriptstyle H,\ldots,H}^{n-l+1\text{ times}}}_{f^{[n+1]}}(V,\ldots,V)-\bT^{\overbrace{\scriptstyle H+tV,\ldots,H+tV}^{l\text{ times}},\overbrace{\scriptstyle H,\ldots,H}^{n-l+2\text{ times}}}_{f^{[n+1]}}(V,\ldots,V)\right)\\
			=&\,tn!\sum_{l=0}^n(n+1-l)\bT^{\overbrace{\scriptstyle H+tV,\ldots,H+tV}^{l+1\text{ times}},\overbrace{\scriptstyle H,\ldots,H}^{n-l+2\text{ times}}}_{f^{[n+2]}}(V,\ldots,V).
		\end{align*}
		From the operator-norm bound of $\bT^{H_0,\ldots,H_{n+1}}_{f^{[n+1]}}(V,\ldots,V)$ given by Theorem \ref{thm:bT in MOIs and bound}, we obtain the bound
		\begin{align*}
			&\norm{\frac{\frac{d^n}{ds^n} f(H+sV)\big|_{s=t}-\frac{d^n}{ds^n} f(H+sV)\big|_{s=0}}{t}-(n+1)!\bT^{H,\ldots,H}_{f^{[n+1]}}(V,\ldots,V)}\\
			\leq&\, |t|(n+2)!\sum_{p=0}^{n+2}\vect{n+2}{p}\frac{1}{p!}\|\mu_p\|\max\left(\norm{V(H-i)^{-1}},\norm{V(H+tV-i)^{-1}}\right)^{n+1}.
		\end{align*}
		From the above facts, we have $\norm{V(H+tV-i)^{-1}}\leq\frac{a+b}{1-a}$. Hence, as $t\to0$, the quotient on the left-hand side of \eqref{eq:quotient} converges in norm to $(n+1)!\bT^{H,\ldots,H}_{f^{[n+1]}}(V,\ldots,V)$, which shows by induction that \ref{item:n times diff at 0} holds for all $n$. This completes the proof of the theorem.
	\end{proof}
	\begin{rem}\label{rem:multi-Gateaux}
		Multi-variable Gateaux derivatives of operator functions can consequently be expressed in terms of (generalised) MOIs as
		$$D^n_Hf[V_1,\ldots,V_n]\equiv\frac{d}{dt_1}\cdots\frac{d}{dt_n}f(H+\sum_{i=1}^n t_iV_i)\Big|_{t=0}=\sum_{\substack{\text{permutations $\sigma$}\\ \text{of $1,\ldots,n$}}}\bT_{f^{[n]}}^{H,\ldots,H}(V_{\sigma(1)},\ldots,V_{\sigma(n)}),$$
		and analytic and algebraic properties of multi-variable Gateaux derivatives can be derived from those of MOIs, cf. Theorem \ref{thm:cov rel bdd2}, Theorem \ref{thm:superscript difference bT}, and \cite[Equation (4.2)]{vNvS21a}.
	\end{rem}
	
	\begin{thm}\label{thm:Taylor_series}
		(Taylor series) Let $H$ be self-adjoint (possibly unbounded) in $\H$ and let $V$ be symmetric and relatively $H$-bounded with $H$-bound $a\in[0,1)$. Let $f\in \cap_{n=1}^\infty\mW_n(\R)$ satisfy $\|\mu_n\|\leq c_fC_f^{n}n!$ for constants $c_f,C_f$ such that $(1+C_f)\|V(H-i)^{-1}\|<1$, where $(fu^n)^{(n)}(x)=\int e^{ixy}d\mu_n(y)$. We then have the Taylor expansion
		$$f(H+V)=\sum_{n=0}^\infty\frac{1}{n!}\frac{d^n}{dt^n} f(H+tV)\big|_{t=0},$$
		which converges absolutely in operator norm.
	\end{thm}
	\begin{proof}
		By Theorem \ref{thm:higher derivatives in norm}, the Taylor remainder can be written as
		\begin{align*}
			R_{n,H,f}(V):=&f(H+V)-\sum_{k=0}^{n-1}\frac{1}{k!}\frac{d^k}{dt^k} f(H+tV)\big|_{t=0}\\
			=&f(H+V)-\sum_{k=0}^{n-1}\bT^{H,\ldots,H}_{f^{[k]}}(V,\ldots,V).
		\end{align*}
		By induction and Theorem \ref{thm:superscript difference} it follows that 
		\begin{align*}
			R_{n,H,f}(V)=\bT^{H+V,H,\ldots,H}_{f^{[n]}}(V,\ldots,V).
		\end{align*}
		By applying Theorem \ref{thm:bT in MOIs and bound} (note that $H_0=H+V$ does not appear in the product $\prod_{j=1}^n\|V(H_j-i)^{-1}\|$), we obtain
		\begin{align*}
			\norm{R_{n,H,f}(V)}&\leq \sum_{p=0}^n\vect{n}{p}\frac{1}{p!}\|\mu_p\|\|V(H-i)^{-1}\|^{n},
		\end{align*}
		and $\|\bT^{H,\ldots,H}_{f^{[n]}}(V,\ldots,V)\|$ satisfies the exact same bound.
		By applying our assumption $\|\mu_p\|\leq c_fC_f^{p}p!$ and putting the binomial theorem in reverse, we obtain
		\begin{align*}
			\norm{R_{n,H,f}(V)}&\leq \sum_{p=0}^n\vect{n}{p}c_fC_f^{p}\|V(H-i)^{-1}\|^{n}=c_f(1+C_f)^n\|V(H-i)^{-1}\|^{n}.
		\end{align*}
		If $(1+C_f)\|V(H-i)^{-1}\|<1$ then $\|R_{n,H,f}(V)\|\to0$ as $n\to\infty$ and so 
		$$f(H+V)=\|\cdot\|\text{-}\lim_{n\to\infty}\sum_{k=0}^{n-1}\frac{1}{k!}\frac{d^k}{dt^k} f(H+tV)\big|_{t=0}.$$
		
		By the same argument, we have the absolute norm-convergence
		$$\sum_{n=0}^\infty\norm{\frac{1}{n!}\frac{d^n}{dt^n} f(H+tV)\big|_{t=0}}\leq \sum_{n=0}^\infty c_f(1+C_f)^n\|V(H-i)^{-1}\|^n=\frac{c_f}{1-(1+C_f)\|V(H-i)^{-1}\|}<\infty\,,$$
		concluding the proof.
	\end{proof}
	
	\begin{lem}\label{lem:examples Taylor series}
		For $n\in \N$ and $f\in \bigcap_{n=1}^\infty \mathscr{W}_n(\R)$, define the measures $\mu_n := \widehat{(f u^n)^{(n)}}$ and $\nu_n := \widehat{(f^{(n)} u^n)}$, the distributional Fourier transforms of $(f u^n)^{(n)}$ and $(f^{(n)} u^n)$, respectively. Denote $\WT:=\{f\in\cap_{n=1}^\infty\mathscr{W}_n(\R)~:~\exists c_f,C_f\geq 0~\forall n\in\N: \|\mu_n\|\leq c_fC_f^n n!\}$ for the class of functions to which Theorem \ref{thm:Taylor_series} (the Taylor series) applies. Denote
		\sloppy
		${\WCT}:=\{f\in\cap_{n=1}^\infty\mathscr{W}_n(\R)~:~\exists c_f,C_f\geq 0~\forall n\in\N: \|\nu_n\|\leq c_fC_f^n n!\}$.
		The following holds.
		\begin{enumerate}
			[label=\textnormal{(\alph*)}]
			\item\label{item:inclusion} $\WCT\subseteq\WT$.
			\item\label{item:prod} If $f,g\in\WCT$, then $fg\in\WCT$.
			\item\label{item:bdd rational} All bounded rational functions are in $\WCT$.
			\item\label{item:Gaussian} We have $f\in\WCT$ for $f(x):= e^{i\xi x}\,e^{-cx^2}$, $c>0$, $\xi\in\R$.
		\end{enumerate}
	\end{lem}
	\begin{proof}
		Throughout the proof, we use the notation $a(n)\prec b(n)$ if there exist $c,C\geq0$ such that $a(n)\leq c C^n b(n) $ for all $n\in\N$.
		
		Statements \ref{item:inclusion} and \ref{item:prod} are both a straightforward check that employs the Fourier convolution theorem.
		
		For \ref{item:bdd rational}, we let $u_w(x):=x-w$ for some $w \in \C$ with $Im~w\neq 0$. Then  
		$$(u_w^{-1})^{(n)}=(-1)^nn!u_w^{-n-1},$$
		and
		$$u_w^{-1}u=1+(w-i)u_w^{-1},$$ 
		and these formulas together with the binomial theorem imply that 
		\begin{align}\label{eq:u_w formula}
			(u_w^{-1})^{(n)}u^n= (-1)^n n! \sum_{k=0}^{n} \vect{n}{k}(w-i)^k u_w^{-k-1}.
		\end{align}
		By \cite[Lemma 7]{PoSu_Crel09} we have 
		\begin{align}\label{eq:bd u_w L1}
			\|\widehat{u_w^{-k-1}}\|_1\prec \|u_w^{-k-1}\|_2+\|(u_w^{-k-1})'\|_2.
		\end{align}
		Moreover, for all $m\in\{1,\ldots, n+2\}$ we have
		\begin{align}\label{eq:bd u_w L2}
			\|u_w^{-m}\|_2^2 \prec&~\int_\mathbb{R} \frac{1}{(1+x^2)^m}\,dx \leq \int_\mathbb{R} \frac{1}{1+x^2}\,dx\prec 1.
		\end{align}
		Combining \eqref{eq:u_w formula}, \eqref{eq:bd u_w L1}, and \eqref{eq:bd u_w L2}, we obtain
		\begin{align*}
			\|((u_w^{-1})^{(n)}u^n)\,\hat{~}\,\|_1 & \prec n!.
		\end{align*}
		Hence, $u_w^{-1}\in\WCT$. 
		Consequently, by \ref{item:prod}, 
		$$u_{w_1}^{-1} \cdots u_{w_n}^{-1}\in{\WCT}$$
		for all $w_i \in \mathbb{C}$ with $Im~w_i\neq 0$. Since every bounded rational function on $\R$ is a linear combination of such functions, \ref{item:bdd rational} is proven.

		We now prove \ref{item:Gaussian}. We set $c=1$ and $\xi=0$ for simplicity, noting that the general case is proved analogously.
		We shall use the fact that, for $m\in\{0,\ldots,2n\}$,
		\begin{align}\label{eq:bd x^m gaussian}
			\|\widehat{x^m e^{-x^2}}\|_1\prec \sqrt{m!}
		\end{align}
		(see, e.g., \cite[proof of Proposition 8(v)]{vNvS21a}).
		Write $f(x)=e^{-x^2}$. 
		We note that $u^n(x)=\sum_{l=0}^n\vect{n}{l}x^l(-i)^{n-l}$, and that from the well-known explicit expression for the Hermite polynomials it follows that $$f^{(n)}(x)=n!\sum_{m=0}^{\lfloor\frac{n}{2}\rfloor}\frac{(-1)^{m+n}}{m!(n-2m)!}(2x)^{n-2m}e^{-x^2}.$$
		So,
		\begin{align*}
			f^{(n)}(x)u^n(x)
			&= n!\sum_{l=0}^n\vect{n}{l}\sum_{m=0}^{\lfloor\frac{n}{2}\rfloor}\frac{(-i)^{n-l}(-1)^{m+n}}{m!(n-2m)!}\,x^l(2x)^{n-2m}e^{-x^2}.
		\end{align*}
		Combining the latter with \eqref{eq:bd x^m gaussian} yields
		\begin{align}\label{eq:fun bound}
			\|\widehat{f^{(n)}u^n}\|_1&\prec n!\sum_{l=0}^n\vect{n}{l}\sum_{m=0}^{\lfloor\frac{n}{2}\rfloor}\frac{1}{m!(n-2m)!}\sqrt{(n+l-2m)!}.
		\end{align}
		Stirling's approximation gives $k!\prec k^k$ and $k^k\prec k!$ for $k\in\{0,\ldots,2n\}$. This allows us to estimate, for all $m, l\leq n$,
		\begin{align}
			\frac{1}{m!(n-2m)!}\sqrt{(n+l-2m)!}&\leq \frac{\sqrt{(2n-2m)!}}{m!(n-2m)!}\nonumber\\
			&\prec \frac{(2n-2m)^{n-m}}{m^m(n-2m)!}\nonumber\\
			&\prec \frac{m!(n-m)!}{(2m)!(n-2m)!}\nonumber\\
			&=\frac{m!(n-m)!}{n!}\frac{n!}{(2m)!(n-2m)!}\nonumber\\
			&\leq \vect{n}{2m}.\label{eq:some factorials}
		\end{align}
		Using the fact that $\sum_{s=0}^p\vect{p}{s}=2^p\prec 1$ for all $p\leq n$, from combining \eqref{eq:fun bound} with \eqref{eq:some factorials} we get
		\begin{align*}
			\|\widehat{f^{(n)}u^n}\|_1&\prec n!\sum_{l=0}^n\vect{n}{l}\sum_{m=0}^{\lfloor\frac{n}{2}\rfloor}\vect{n}{2m}
			\prec n!,
		\end{align*}
		as desired. We conclude that $f\in\WCT$.
	\end{proof}
	
	Due to the relatively strong assumptions required for the convergence of the Taylor series for the function 
	$t\mapsto f(H+tV)$, it is often more appropriate to focus on the Taylor remainder instead. In the next section, we examine the spectral properties of the Taylor remainder, particularly in relation to trace formulas and spectral shift functions.

	\section{Spectral shift functions}\label{sec:ssf}
	\label{sct:Spectral shift functions}
	
	The next definition introduces the function class that features in our main theorem. Although it looks technical at first, the class neatly shows the influences of the summability parameter $n$, and the order $k$ of the spectral shift function, on the required decay and differentiability of the function $f$. Let $u(\lambda)=(\lambda-i), \lambda\in\R$. For $x,y \in \R$, let $x\vee y=\max\{x,y\}$.
	\begin{defn}\label{def:function class Q_n}
		Let $n,k\in \mathbb{N}$. Let $\mathfrak{Q}_n^k(\R)$ denote the space of all functions $f\in C^k(\R)$ such that
		\begin{enumerate}[label=\textnormal{(\alph*)}]
			\item\label{item:def25i} $fu^{2n}\in C_b(\R)$, 
			\item\label{item:def25ii} $f^{(l)}u^{n + l + 1}\in C_0(\R)$, $1\leq l\leq k$,~\text{and}
			\item\label{item:def25iii} $\widehat{f^{(k)}u^{k\vee n}}\in L^1(\R)$.
		\end{enumerate}
	\end{defn}
	We note the following properties of $f\in\mathfrak{Q}_n^k(\R)$.
	\begin{prop}\label{Fourierprop}
		Let $n, k\in\N$. Then, for each $f\in\mathfrak{Q}_n^k(\R)$,
		\begin{enumerate}[label=\textnormal{(\alph*)}]
			\item\label{item:func-1} $\widehat{f^{(l)}u^{s}}\in L^1(\R)$ for $s,l\in\mathbb Z_{\geq0}$ such that $s\leq n\vee l$ and $l\leq k$;
			\item\label{item:func-2} $\left((fu^{s})^{(l)}\right)\hat{~}\in L^1(\R)$ for $s,l\in\mathbb Z_{\geq0}$ such that $s\leq n\vee l$ and $l\leq k$;
			\item\label{item:func-3} $\mathfrak{Q}_n^k(\R)\subset \mW_k(\R)$.
		\end{enumerate}
	\end{prop}
	\begin{proof}
		Item \ref{item:func-1} follows from Definition \ref{def:function class Q_n}, the convolution theorem of the Fourier transform, and the well-known fact that $\hat{g}\in L^1(\R)$ for all $g\in C^1(\R)$ with $g,g'\in L^2(\R)$. Item \ref{item:func-2} follows from \ref{item:func-1} and the Leibniz rule. Item \ref{item:func-3} is immediate.
	\end{proof}
	Let $\mathfrak{R}_n:=\{(\cdot-z)^{-l}: Im(z)\neq 0, l\in\N, l\geq 2n+1\}$, and let $\mathcal{S}(\R)$ denote the class of Schwartz functions on $\R$. Then it follows from the definition of $\mathfrak{Q}_n^k(\R)$ that 
	\[C_c^{k+1}(\R)\subset \mathfrak{Q}_n^k(\R)\subset C_0(\R),~~ \mathcal{S}(\R)\subset \mathfrak{Q}_n^k(\R), \text{ and } \mathfrak{R}_n\subset\mathfrak{Q}_n^k(\R).\]
	As one straightforwardly checks that $\mathfrak{Q}_n^k(\R)$ is an algebra, arbitrary products of functions in $C_c^{n+1}(\R)\cup\mathcal S(\R)\cup\mathfrak{R}_n$ are in $\mathfrak{Q}_n^k(\R)$ as well.
	
	\begin{rem}
		For all $n, k_1, k_2\in \Nat$, Proposition \ref{Fourierprop}\ref{item:func-1} implies that
		\[\mathfrak{Q}_n^{k_1}(\R)\subseteq \mathfrak{Q}_n^{k_2}(\R) \text{ whenever } k_2\leq k_1.\]
	\end{rem}
	
	\noindent For proof of the following Lemma, we refer to \cite[Lemma 2.4]{vNS23}.
	\begin{lem}\label{lem:partial integration}
		Let $n,m\in\Nat\cup\{0\}$, and let $\mu$ be a (finite) complex Radon measure on $\R$. For every $\epsilon\in(0,1]$ and every $k\in\N\cup\{0\}$ there exists a complex Radon measure $\tilde\mu_{k,\epsilon}$ with $$\norm{\tilde\mu_{k,\epsilon}}\leq\left\|u^{-1-\epsilon}\right\|_1
		\norm{\mu}$$ and
		\begin{align}
			\label{lsint}
			\int_\R g^{(n)}u^m \,d\mu=\int_\R g^{(n+k)}u^{m+k+\epsilon}d\tilde\mu_{k,\epsilon}
		\end{align}
		for all $g\in C^{n+k}(\R)$ satisfying $g^{(n+l)}u^{m+l}\in C_0(\R)$, $l=0,\ldots,k-1$ and $g^{(n+k)}u^{m+k-1}\in L^1(\R)$.
	\end{lem}
	\smallskip

	\smallskip
	
	Throughout this section, we shall refer to the following hypothesis.
	\begin{hypo}\label{mainhypo}
		Let $n\in \N$, and let $n\geq2$. Let $H$ be self-adjoint in $\Hcal$, and let $V$ be symmetric and relatively $H$-bounded with $H$-bound $<1$, that is,
		$$\norm{V\psi}\leq a\norm{H\psi}+b\norm{\psi} \text{ for all } \psi\in\dom H,$$
		for some $a\in[0,1)$ and $b\in[0,\infty)$, such that 
		\begin{align}\label{eq:cond}
			V(H-i)^{-p}\in\Snp,\qquad p\in\{1,\ldots,n\}.
		\end{align}
	\end{hypo}
	We firstly remark that a symmetric operator $V$ satisfying \eqref{eq:cond} is automatically relatively $H$-bounded, but the specific numbers $a$ and $b$ play a central role in several results below. We secondly remark that Hypothesis \ref{mainhypo} for $(n,H,V)$ implies Hypothesis \ref{mainhypo} for $(n+1,H,V)$ \textit{et cetera}. For improved function classes for the case $n=1$, see \cite[Theorem 4.1]{vNS22}.
	
	\subsection{Krein spectral shift functions}\label{sct:Krein}
	In this subsection, we will obtain the Krein trace formula (first-order spectral shift formula) and the associated spectral shift function under Hypothesis \ref{mainhypo}. The following two lemmas are essential to prove our main results in this subsection.
	\begin{lem}\label{aprlem1}
		Assume Hypothesis \ref{mainhypo}. There exists a sequence $\{V_k\}_{k\in\N}$ of finite rank self-adjoint operators such that 
		\begin{enumerate}[label=\textnormal{(\alph*)}]
			\item\label{item:aprlem1ii} $\|V_k(H-i)^{-1}\|, \|V_k(H+V_k-i)^{-1}\|\leq \frac{a+b}{1-a}$;\\
			\item\label{item:aprlem1iii} $\|\cdot\|_{n/p}-\lim\limits_{k\to\infty}V_k(H-i)^{-p}=V(H-i)^{-p}$ for $1\leq p\leq n$;\\
			\item\label{item:aprlem1iv} $\|V_k(H-i)^{-p}\|_{n/p}\,\leq \|V(H-i)^{-p}\|_{n/p}$ for $1\leq p\leq n$.
		\end{enumerate}
	\end{lem}
	\begin{proof}
		Let $E_{H}(\cdot)$ be the spectral measure of $H$. Let $k\in\N$, and let $P_k=E_{H}((-k,k))$. Then the following statements are true.
		\begin{itemize}
			\item for each $\psi \in \H$, $P_k\psi\in \dom H$,
			\item $SOT-\lim\limits_{k\to\infty}P_k= I$.
		\end{itemize}
		
		Note that 
		\begin{align}\label{belongL1}
			P_kVP_k((H-i)^{-n}P_k+P_k^\perp)\,=\, P_kVP_k(H-i)^{-n}P_k\,=\, P_kV(H-i)^{-n}P_k\in\Scal^1.
		\end{align}
		By functional calculus it follows that $((H-i)^{-n}P_k+P_k^\perp)$  is boundedly invertible. Hence, from \eqref{belongL1}, we have 
		\begin{align}\label{belongL2}
			P_kVP_k\,=\, P_kV(H-i)^{-n}P_k ((H-i)^{-n}P_k+P_k^\perp)^{-1}\in\Scal^1.
		\end{align}
		For each fixed $k\in\N$, by the spectral theorem there exists a sequence $\{E^l_{P_kVP_k}\}_{l\in\N}$ of finite rank projections such that 
		\[E^l_{P_kVP_k}P_kVP_k=P_kVP_kE^l_{P_kVP_k}, \text{ and } \|\cdot\|_1-\lim\limits_{l\to\infty}E^l_{P_kVP_k}P_kVP_k=P_kVP_k.\]
		Therefore, there exists a strictly increasing sequence of natural numbers $\{l_k\}_{k\in\N}$ such that 
		\[\left\|E^{l_k}_{P_kVP_k}P_kVP_k-P_kVP_k\right\|_1<\tfrac{1}{k}.\]
		Let $V_k=E^{l_k}_{P_kVP_k}P_kVP_k$. Then 
		\begin{enumerate}[label=\textnormal{(\alph*)}]
			\item For each $\psi\in\H$, $\|V_k\psi\|=\left\|E^{l_k}_{P_kVP_k}P_kVP_k\psi\right\|\leq \left\|VP_k\psi\right\|\leq a\|H\psi\|+ b\|\psi\|$. Therefore, by Lemma \ref{lem:properties relatively bdd}\ref{item:properties rel bdd: norm} and Lemma \ref{lem:properties relatively bdd}\ref{item:properties rel bdd: shift}, $ \|V_k(H-i)^{-1}\|, \|V_k(H+V_k-i)^{-1}\|\leq \frac{a+b}{1-a}$.
			\item We have
			\begin{align*}
				&\|V_k(H-i)^{-p}-V(H-i)^{-p}\|_{n/p}\\ \leq~&\|V_k(H-i)^{-p}-P_kVP_k(H-i)^{-p}\|_{n/p}+\|P_kV(H-i)^{-p}P_k-V(H-i)^{-p}\|_{n/p}\\
				\leq~ & \tfrac{1}{k}+ \|P_kV(H-i)^{-p}P_k-V(H-i)^{-p}\|_{n/p}\longrightarrow 0 \text{ as } k\to \infty.
			\end{align*}
			\item We have $\|V_k(H-i)^{-p}\|_{n/p}\,=\|E^{l_k}_{P_kVP_k}P_kV(H-i)^{-p}P_k\|_{n/p}\,\leq \|V(H-i)^{-p}\|_{n/p}$.
		\end{enumerate}
		This completes the proof.
	\end{proof}
	\textbf{Convention:} Let $H$ be a self-adjoint operator in $\H$, and $V$ be a symmetric and relatively $H$-bounded with $H$ bound $a$ such that $V(H-i)^{-l}\in \S^k$ for some $l,k$. Then it is easy to show that $(H-i)^{-l}V$ equals $\left(V(H+i)^{-l}\right)^*$ on $\dom V$. As $\dom V$ is dense in $\H$, the operator $(H-i)^{-l}V$ has a unique bounded extension on the whole of $\H$, and we still denote this extension by $(H-i)^{-l}V$. Also, $V(H-i)^{-l}\in \S^k$ implies $ (H-i)^{-l}V \in\S^k$ with $\|(H-i)^{-l}V\|_k=\|V(H+i)^{-l}\|_k$.
	
	\begin{lem}\label{aprlem2}
		Let $n, H, V$ satisfy Hypothesis \ref{mainhypo}. Let $(R,W)\in\{(0,V), (0,V_k), (V_k,V-V_k), (V_k,V_l-V_k)\}$, where $\{V_k\}_{k\in\N}$ is given by Lemma \ref{aprlem1}. For $t\in[0,1]$, we denote
		\[H^t:=H+R+tW.
		\]
		For $1\leq p\leq n$, define 
		\begin{align}\label{kk3}
			\alpha_{n,V,H}:=\max_{1\leq p\leq n}\|V(H-i)^{-p}\|_{n/p}.
		\end{align}
		Then 
		\begin{enumerate}[label=\textnormal{(\alph*)}]
			\item\label{esta} $(H^t-i)^{-p}W, W(H^t-i)^{-p}\in \Snp$, and
			\item\label{estb} $\|(H^t-i)^{-p}W\|_{n/p}=\|W(H^t-i)^{-p}\|_{n/p}\leq  2^p\,\frac{1+b}{1-a}\,\alpha_{n,W,H}\,\max\limits_{\ell=0,\, p-1}\{\alpha_{n,V,H}^{\ell}\}$.
		\end{enumerate}
		
	\end{lem}
	
	\begin{proof}
		As $V_k\in\mB(\H)$, and the $H$-bound of $V$ is $<1$, it follows that $R, W$ are relatively $H$-bounded and $H^t$-bounded. The equality in \ref{estb} follows directly from $(H^t+i)^{-l}W=(W(H^t-i)^{-l})^*$ (see the above convention), as this operator equals $(H^t-i)^{-l}W$ up to a unitary.
		
		We prove the remaining results using mathematical induction on $p$. Let $p=1$. Let $V_t:=H^t-H$. By the second resolvent identity Lemma \ref{lem:second resolvent identity}, we have
		\begin{align*}
			W[(H^t-i)^{-1}-(H-i)^{-1}]=-W(H-i)^{-1}V_t(H^t-i)^{-1}.
		\end{align*}
		Note that $V_t=R+tW$ is $H$-bounded with $H$-bound $\leq a<1$. Hence, by Lemma \ref{lem:properties relatively bdd}\ref{item:properties rel bdd: shift}, we have $\|V_t(H^t-i)^{-1}\|\leq \frac{a+b}{1-a}$. Therefore, the above identity implies that
		\begin{align}\label{eq:est 2nd res id}
			\|W(H^t-i)^{-1}\|_{n}&\leq (1+ \frac{a+b}{1-a}) \|W(H-i)^{-1}\|_{n}\nonumber\\
			&=\frac{1+b}{1-a} \|W(H-i)^{-1}\|_{n}.
		\end{align} Thus \ref{esta} and \ref{estb} are true for $p=1$. 
		
		Suppose \ref{esta} and \ref{estb} are true for $p=1,2,\ldots,m$ for some $m\in\{1,\ldots, n-1\}$. We will prove that \ref{esta} and \ref{estb} also hold for $p=m+1$. Indeed, a telescopic sum together with repeated applications of Lemma~\ref{lem:second resolvent identity} yields
		
		\begin{align*}
			&W[(H^t-i)^{-(m+1)}-(H-i)^{-(m+1)}]\\
			&= W\sum_{l=0}^{m}\big((H^t-i)^{-1}\big)^{l} \big((H^t-i)^{-1}-(H-i)^{-1} \big)\big((H-i)^{-1}\big)^{m-l}\\
			&=-\sum_{l=0}^{m}W(H^t-i)^{-(l+1)}V_t(H-i)^{-(m+1-l)}.
		\end{align*}
		Therefore, using H\"older's inequality from the above identity we conclude
		\begin{align}\label{estc}
			\nonumber&\|W(H^t-i)^{-(m+1)}\|_{n/(m+1)}\\
			\nonumber\leq~ & \|W(H-i)^{-(m+1)}\|_{n/(m+1)}+\|W(H^t-i)^{-1}\|_{n}\,\|V(H-i)^{-m}\|_{n/m}\\
			&\hspace*{1in}+\sum_{l=1}^{m}\|W(H^t-i)^{-l}\|_{n/l}\,\|V(H-i)^{-(m+1-l)}\|_{n/(m+1-l)}.
		\end{align}
		Let us write $C_p= \frac{1+b}{1-a}\,\alpha_{n,W,H}\,\max\limits_{\ell=0,\, p-1}\{\alpha_{n,V,H}^{l}\}$, and consider the three terms on the right-hand side of \eqref{estc}. The first term is bounded by $\alpha_{n,W,H}\leq C_{m+1}$, and thanks to \eqref{eq:est 2nd res id} the second term is also bounded by $C_{m+1}$. By induction hypothesis, the third term is bounded by $\sum_{l=1}^m2^lC_l\|V(H-i)^{-(m+1-l)}\|_{n/(m+1-l)}\leq\sum_{l=1}^m2^lC_{m+1}.$ 
		In total, we obtain
		\begin{align*}
			\|W(H^t-i)^{-(m+1)}\|_{n/(m+1)}\leq 2^{m+1}C_{m+1},
		\end{align*}
		proving the induction step. The lemma follows.
	\end{proof}
	
	Next, we recall known norm estimates of the classical MOI. The following estimate is a consequence of \cite[Theorem 4.4.7  and Remark 4.4.2]{ST19}, first proven for $\tilde H=H$ in \cite[Theorem 5.3 and Remark 5.4]{PSS}. 
	\begin{thm}\label{inv-thm}
		Let $k\in\mathbb{N}$ and let $\alpha,\alpha_1,\ldots,\alpha_k\in(1,\infty)$ satisfy $\tfrac{1}{\alpha_1}+\cdots+\tfrac{1}{\alpha_k}=\tfrac{1}{\alpha}$. Let $H,\tilde H$ be self-adjoint operators in $\Hcal$. Assume that $V_m\in\mathcal{S}^{\alpha_m},\, 1\leq m\leq k$. Then there exists $c_{\alpha,k}>0$ such that
		\begin{align}\label{est}
			\|T^{\tilde{H},H,\ldots,H}_{f^{[k]}}(V_1,V_2,\ldots,V_k)\|_{\alpha}\leq c_{\alpha,k}\, \|f^{(k)}\|_\infty \prod\limits_{1\leq m \leq k}\|V_m\|_{\alpha_m}
		\end{align}
		for every $f\in C_b^k(\R)$ such that $\widehat{f^{(k)}}\in L^1(\R)$.
	\end{thm}

	The following theorem gives a key estimate to establish Krein's trace formula under Hypothesis \ref{mainhypo}.
	\begin{thm}\label{Krthm0}
		Let $n, H, V$ satisfy Hypothesis \ref{mainhypo}. Let $(R,W)\in\{(0,V),(V_l,V-V_l),(V_l,V_k-V_l)\}$, where $\{V_k\}_{k\in\N}$ is given by Lemma~\ref{aprlem1}. Let $t\in[0,1]$, and denote $H^t=H+R+tW$. Then for each $f\in \mathfrak{Q}_n^1(\R)$, we have
		\begin{align}\label{eq:two parts}
			\bT^{H^t,H^t}_{f^{[1]}}(W)= T^{H^t,H^t}_{(fu^n)^{[1]}}(W(H^t-i)^{-n})-\sum_{k=0}^{n-1}(fu^k)(H^t)W(H^t-i)^{-(k+1)}\, \in\mathcal{S}^1,
		\end{align}
		and each of the above $n+1$ summands is $\S^1$-continuous in $t$. 
		Moreover,
		\begin{align}\label{Newtraceest}
			\left|\Tr\left(T^{H^t,H^t}_{(fu^n)^{[1]}}(W(H^t-i)^{-n})\right)\right|&\leq \, \|(fu^n)^{(1)}\|_\infty\,\|W(H^t-i)^{-n}\|_1,\\
			\label{New11}\left|\Tr\left(\sum_{k=0}^{n-1}(fu^k)(H^t)W(H^t-i)^{-(k+1)}\right)\right|&\leq n\|fu^{n-1}\|_\infty\,\|W(H^t-i)^{-n}\|_1.
		\end{align}
	\end{thm}
	\begin{proof}
		Let $f\in \mathfrak{Q}_n^1(\R)$. Then by repeated applications of Proposition~\ref{Fourierprop}, Theorem \ref{thm:cov rel bdd} and Lemma \ref{thm:well defined}, we get
		\begin{align}\label{Kr1}
			\nonumber \bT^{H^t,H^t}_{f^{[1]}}(W)=&T^{H^t,H^t}_{(fu)^{[1]}}(W(H^t-i)^{-1})-f(H^t)W(H^t-i)^{-1}\\
			\nonumber =&T^{H^t,H^t}_{(fu^2)^{[1]}}(W(H^t-i)^{-2})-(fu)(H^t)W(H^t-i)^{-2}-f(H^t)W(H^t-i)^{-1}\\
			=&T^{H^t,H^t}_{(fu^n)^{[1]}}(W(H^t-i)^{-n})-\sum_{k=0}^{n-1}(fu^k)(H^t)W(H^t-i)^{-(k+1)}.
		\end{align}
		As Proposition \ref{Fourierprop} implies $\widehat{(fu^n)^{(1)}}\in L^1(\R)$, Corollary \ref{cor:bd_moi_bdd} (namely, \eqref{fourierbound}), in particular means that $T_{(fu^n)^{(1)}}^{H^t,H^t}$ maps $\S^1$ to $\S^1$. Applying Lemma \ref{aprlem2} therefore yields
		\begin{align}\label{MOI:trclass}
			T^{H^t,H^t}_{(fu^n)^{[1]}}(W(H^t-i)^{-n})\in \mathcal{S}^1.
		\end{align}
		Since $t\mapsto (H^t-i)^{-1}$ is strongly continuous, it follows (see \cite[Theorem VIII.20]{ReedSimonI}) that $t\mapsto e^{isH^t}$ is strongly continuous. Hence, Corollary \ref{cor:bd_moi_bdd} and Lemma \ref{aprlem2} also yield $\S^1$-continuity of the above operator in $t$.
		
		Since $fu^{2n}\in C_b(\R)$ by Definition \ref{def:function class Q_n}\ref{item:def25i}, for each $0\leq k\leq n-1$, by Lemma \ref{aprlem2} we have
		\begin{align}\label{trcl}
			(fu^k)(H^t)W (H^t-i)^{-(k+1)}=(fu^{k+n})(H^t)(H^t-i)^{-n}W (H^t-i)^{-(k+1)}\in\Scal^1.
		\end{align}	
		Thus, the identity \eqref{Kr1} and the properties \eqref{MOI:trclass} and \eqref{trcl} together establish \eqref{eq:two parts}. The $\S^1$-continuity of \eqref{trcl} in $t$ follows from Lemma \ref{aprlem2} and strong continuity of $t\mapsto (fu^{k+n})(H^t)$, where the latter fact follows (see \cite[Theorem VIII.20]{ReedSimonI}) from the inclusion $fu^{k+n}\in C_b(\R)$ and the strong continuity of $t\mapsto (H^t-i)^{-1}$.
		
		By a standard DOI property of the general form ``$\Tr(T^{D,D}_{g^{[1]}}(A))=\Tr(g'(D)A)$'', which in this case can be derived from Corollary \ref{cor:bd_moi_bdd}, we have $$\Tr\left( T^{H^t,H^t}_{(fu^n)^{[1]}}(W(H^t-i)^{-n}) \right)=\Tr\left( (fu^n)^{(1)}(H^t)\,(W(H^t-i)^{-n}) \right).$$ Hence, by utilizing H\"older's inequality, we derive the estimate \eqref{Newtraceest}.

		By using \eqref{trcl} and the cyclicity of the trace we conclude
		\begin{align}\label{Kr2}
			\nonumber\Tr\big((fu^k)(H^t)W(H^t-i)^{-(k+1)}\big)=&\Tr\big((H^t-i)^{-(k+1)}(fu^k)(H^t)W\big)\\
			=&\Tr\big((fu^{n-1})(H^t)(H^t-i)^{-n}W\big).
		\end{align}
		Applying H\"older's inequality to the right-hand side of \eqref{Kr2}, 	combined with the invariance of $\|\cdot\|_1$ under adjoints, results in the estimate \eqref{New11}. This concludes the proof.
	\end{proof}
	
	We briefly recall the classical result that the first-order spectral shift function exists for perturbations of trace class. This result is due to M. G. Krein \cite{Krein}, see also \cite[Theorem 8.3.3]{Yafaev92}.
	\begin{thm}[Krein]
		\label{thm:Krein}
		Let $H$ be a self-adjoint operator in $\H$, and let $V=V^*\in\S^1$. Then there exists an integrable function $\xi$ such that for all $f\in C^1(\R)$ satisfying $\widehat{f'}\in L^1(\R)$ (slightly more generally, $\widehat{f'}$ is only assumed to be a finite complex measure) we have
		$$\Tr(f(H+V)-f(H))=\int_{-\infty}^\infty\xi(\lambda)f'(\lambda)\,d\lambda.$$
	\end{thm}
	
	The following is the main result of this subsection.
	\begin{thm}\label{krmainthm}
		Let $n, H, V$ satisfy Hypothesis \ref{mainhypo}. Then there exists a locally integrable function $\eta_{1,H,V}\in L^1\left(\R, \frac{d\lambda}{(1+|\lambda|)^{n+\epsilon}}\right)$  such that
		\begin{align}\label{Kr3}
			\int_\R \frac{|\eta_{1,H,V}(\lambda)|}{(1+|\lambda|)^{n+\epsilon}}d\lambda\leq \,n\,2^n\,\big(2n+3+2(n+1)\epsilon^{-1}\big)\, \frac{1+b}{1-a}\,\max\limits_{\ell=1, \,n}\{\alpha_{n,V,H}^{\ell}\}
		\end{align}
		for all $\epsilon\in(0,1]$, where $\alpha_{n,V,H}$ is given by \eqref{kk3},
		and, moreover,
		\begin{align}\label{Kr4}
			\Tr\left( f(H+V)-f(H) \right)=\int_\R f'(\lambda)\eta_{1,H,V}(\lambda)d\lambda
		\end{align}
		for every $f\in \mathfrak{Q}_n^1(\R)$. The locally integrable function $\eta_{1,H,V}$ is determined by \eqref{Kr4} uniquely up to an additive constant.
	\end{thm}
	\begin{proof} Let $f\in \mathfrak{Q}_n^1(\R)$. Let $H_t=H+tV, t\in[0,1]$. By the fundamental theorem of calculus and using Theorem \ref{thm:higher derivatives in norm} we get the (a priori pointwise) integral
		\begin{align}\label{Kr5}
			f(H+V)-f(H) =\int_{0}^{1} \bT^{H_t,H_t}_{f^{[1]}}(V)\,dt.
		\end{align}
		
		We now apply Theorem \ref{Krthm0} to the above. By \eqref{Newtraceest}, \eqref{New11}, Lemma \ref{aprlem2}, and $\S^1$-continuity of the trace we may take the trace of the integral of \eqref{eq:two parts} and find
		\begin{align}
			&\Tr(f(H+V)-f(H))\\
			=&\Tr\left(\int_0^1T^{H_t,H_t}_{(fu^n)^{[1]}}(V(H_t-i)^{-n})\,dt\right)-\Tr\left(\int_0^1\sum_{k=0}^{n-1}(fu^k)(H_t)V(H_t-i)^{-(k+1)}dt\right)\\
			=&\alpha_1\big((fu^n)^{(1)}\big)+\alpha_2\big(fu^{n-1}\big)
		\end{align}
		for certain linear functionals $\alpha_1,\alpha_2$ defined on suitable subspaces of $C_0(\R)$, and all $f\in\Q_n^1(\R)$. The above integrands are $\S^1$-continuous by Theorem \ref{Krthm0}, which shows that the above expressions are well defined and we may interchange trace with integral everywhere.

		From \eqref{Newtraceest} and \eqref{New11} it follows that $\alpha_1$ and $\alpha_2$ are continuous in the supremum norm, and hence we may extend $\alpha_1$ and $\alpha_2$ to continuous functionals on $C_0(\R)$. By the Riesz--Markov representation theorem, there exist two complex Borel measures $\mu_1,\mu_2$ such that

		\begin{align}\label{Kr7}
			\nonumber\Tr \left(f(H+V)-f(H)\right)=&\int_{\R} (fu^n)^{(1)}(\lambda)d\mu_1(\lambda)+\int_{\R} (fu^{n-1})(\lambda)d\mu_2(\lambda)\\
			=&\int_{\R} \Big(f^{(1)}(\lambda)u^n(\lambda)+nf(\lambda)u^{n-1}(\lambda)\Big)d\mu_1(\lambda)+\int_{\R} (fu^{n-1})(\lambda)d\mu_2(\lambda),
		\end{align}
		and with total variation bounded by \eqref{Newtraceest}, \eqref{New11} and Lemma \ref{aprlem2}\ref{estb} as
		\begin{align}\label{measurebound}
			\nonumber\|\mu_1\|,\|\mu_2\|&\leq n\|V(H_t-i)^{-n}\|_1\\
			&\leq n\,2^n\,\frac{1+b}{1-a}\,\max\limits_{\ell=1, \,n}\{\alpha_{n,V,H}^{\ell}\}.
		\end{align}
		For $\epsilon\in(0,1]$, applying Lemma \ref{lem:partial integration} to \eqref{Kr7} ensures the existence of a measure $\nu_{\epsilon}$ such that
		\begin{align}\label{Kr8}
			\Tr \left(f(H+V)-f(H)\right)=\int_{\R}f^{(1)}(\lambda)\,d\nu_\epsilon(\lambda)\qquad(f\in\Q_n^1(\R))
		\end{align}
		with 
		\begin{align*}
			\int_\R |u^{-n-\epsilon}(\lambda)|d|\nu_\epsilon|(\lambda)\leq\|\mu_1\|+\|u^{-1-\epsilon}\|_1(n\|\mu_1\|+\|\mu_2\|).
		\end{align*}
		As $\frac{1}{1+|\lambda|}\leq |u(\lambda)^{-1}|$, this further implies that
		\begin{align}\label{Kr9}
			\nonumber\int_\R \frac{d|\nu_\epsilon|(\lambda)}{\left(1+|\lambda|\right)^{n+\epsilon}}&\leq\|\mu_1\|+\|u^{-1-\epsilon}\|_1(n\|\mu_1\|+\|\mu_2\|)\\
			&\leq n\,2^n \,\big(2n+3+2(n+1)\epsilon^{-1}\big)\, \frac{1+b}{1-a}\,\max\limits_{\ell=1, \,n}\{\alpha_{n,V,H}^{\ell}\},
		\end{align}
		where we have estimated $\|u^{-1-\epsilon}\|_1=2\int_0^1 |u^{-1-\epsilon}|+2\int_1^\infty|u^{-1-\epsilon}|\leq2+2\epsilon^{-1}$.
		
		Next, we will demonstrate that the measure $\nu_{\epsilon}$ in \eqref{Kr8} is absolutely continuous with respect to the Lebesgue measure. Let $\{V_k\}_{k\in \N}$ be the sequence given in Lemma \ref{aprlem1}. Therefore, by Theorem \ref{thm:Krein}, there exists a sequence of integrable functions $\{\eta_{k,1}\}_{k\in \N}$ on $\R$ such that, for all $f\in C_c^\infty(\R)$,
		\begin{align}\label{eq:Krein V_k}
			\Tr \left(f(H+V_k)-f(H)\right)=\int_{\R}f^{(1)}(\lambda)\,\eta_{k,1}(\lambda)\,d\lambda.
		\end{align}
		If $V_{k,l}:=V_k-V_l$, then from \eqref{eq:Krein V_k}, subsequently applying the fundamental theorem of calculus for MOIs (see \cite[sentence below (42)]{vNS22}), and Theorem \ref{Krthm0}, we gather
		\begin{align}\label{Kr10}
			&\nonumber\left|\int_\R f^{(1)}(\lambda)(\eta_{k,1}-\eta_{l,1})(\lambda)\,d\lambda\right|\\
			\nonumber=&|\Tr\left(f(H+V_k)-f(H+V_l)\right)|\\
			\nonumber=&\left|\int_{0}^{1} \Tr\left(\bT^{H+V_l+tV_{k,l},H+V_l+tV_{k,l}}_{f^{[1]}}(V_{k,l})\right)\,dt\right|\\
			\leq &\left(\|(fu^n)^{(1)}\|_\infty+ n\|fu^{n-1}\|_\infty\right)\,\int_{0}^{1}~dt~\|V_{k,l}(H+V_l+tV_{k,l}-i)^{-n}\|_1.
		\end{align}
		Thanks to Lemma \ref{aprlem2}, from \eqref{Kr10} we obtain
		\begin{align}\label{Kr11}
			&\nonumber\left|\int_\R f^{(1)}(\lambda)(\eta_{k,1}-\eta_{l,1})(\lambda)\,d\lambda\right|\\
			\leq &\left(\|(fu^n)^{(1)}\|_\infty+ n\|fu^{n-1}\|_\infty\right)\,
			2^n\,\frac{1+b}{1-a}\,\alpha_{n,V_{k,l},H}\,\max\limits_{p=0,\, n-1}\{\alpha_{n,V,H}^{p}\}.
		\end{align}
		Let $f\in C_c^\infty((-d,d))$ for some fixed $d>0$. Then, by using the standard estimates $\|gu^k\|_\infty\leq\|g\|_\infty(1+d)^k$ (for $g\in C_c^\infty((-d,d))$ and $\|f\|_\infty\leq 2d\|f'\|_\infty$, from \eqref{Kr11} we get
		\begin{align}\label{Kr12}
			&\left|\int_\R f^{(1)}(\lambda)(\eta_{k,1}-\eta_{l,1})(\lambda)\,d\lambda\right|\\
			\leq \nonumber&\left( 1+d+4nd\right)(1+d)^{n-1}\,\|f^{(1)}\|_\infty\,
			2^n\,\frac{1+b}{1-a}\,\alpha_{n,V_{k,l},H}\,\max\limits_{p=0,\, n-1}\{\alpha_{n,V,H}^{p}\}.
		\end{align}
		Therefore, subsequently using \cite[Lemma 4.5]{PrSk26}, \eqref{Kr12}, and Lemma \ref{aprlem1}\ref{item:aprlem1iii} implies that
		\begin{align*}
			\nonumber	\|[\eta_{k,1}]-[\eta_{l,1}]\|_{L^1((-d,d))/\C}=&\sup_{\|f^{(1)}\|_\infty\leq 1}\left|\int_\R f^{(1)}(\lambda)(\eta_{k,1}-\eta_{l,1})(\lambda)\,d\lambda\right|\\
			\nonumber	\leq &\left(1+d+4nd\right)(1+d)^{n-1}\,
			2^n\, \frac{1+b}{1-a}\,\alpha_{n,V_{k,l},H}\,\max\limits_{p=0,\,n-1}\{\alpha_{n,V,H}^{p}\}\\
			&\longrightarrow 0 \text{ as } k,l \to\infty.
		\end{align*}
		Thus $\{[\eta_{k,1}]\}_{k\in\N}$ is a Cauchy sequence in $L^1_{loc}(\R)/\C$. Consequently, by \cite[Remark 4.8]{PrSk26}, there exists a Cauchy sequence $\{\hat \eta_{k,1}\}_{k\in\N}$ in $L^1_{loc}(\R)$ such that $\hat\eta_{k,1}$ belongs to $[\eta_{k,1}]$ and satisfies \eqref{eq:Krein V_k}; we let $\eta_1$ be its $L^1_{loc}(\R)$-limit. By a calculation completely analogous to \eqref{Kr10}-\eqref{Kr11}, we conclude for $f\in C_c^{\infty}(\R)$ that
		\begin{align}\label{Kr13}
			\nonumber\Tr \left(f(H+V)-f(H)\right)=&\lim\limits_{k\to\infty}\Tr\left(f(H+V_k)-f(H)\right)\\
			\nonumber=&\lim\limits_{k\to\infty} \int_{\R}f^{(1)}(\lambda)\,\hat\eta_{k,1}(\lambda)\,d\lambda\\
			=&\int_{\R}f^{(1)}(\lambda)\,\eta_{1}(\lambda)\,d\lambda.
		\end{align} 
		Now combining  \eqref{Kr8} and \eqref{Kr13} we have the following equality
		\begin{align}\label{Kr14}
			\int_{\R}f^{(1)}(\lambda)\,d\nu_{\epsilon}(\lambda)=\int_{\R}f^{(1)}(\lambda)\,\eta_{1}(\lambda)\,d\lambda\quad \text{for all} \quad f\in C_c^{\infty}(\R).
		\end{align}
		Let $d\mu(\lambda)=\eta_{1}(\lambda)-d\nu_{\epsilon}(\lambda)$. Then, it follows from \eqref{Kr14} that
		\begin{align}\label{Kr15}
			\int_{a}^{b}f^{(1)}(\lambda)d\mu(\lambda)=0\quad \text{for all }f\in C_c^\infty(\R).
		\end{align}
		Now consider the distribution $T_{\mu}$ defined by 	
		$$T_{\mu}(\phi):=\int_{a}^{b} \phi\, d\mu(\lambda)$$ for every $\phi\in C^\infty_c(\R)$.
		By \eqref{Kr15} and the definition of the derivative of a distribution, $T_{\mu}^{(1)}=0$. Hence by \cite[Theorem 3.10 and Example 2.21]{gwaiz}, $d\mu(\lambda)=c\,d\lambda$ for some constant $c$. Consequently, $\eta_{1}\in L^1_{loc}(\R)$ satisfying \eqref{Kr13} is unique up to an additive constant. In particular, we can assume $d\nu_{\epsilon}(\lambda)=\eta_{1}(\lambda)d\lambda$, and obtain
		\begin{align*}
			\int_{\R}\frac{|\eta_{1}(\lambda)|}{(1+|\lambda|)^{n+\epsilon}}d\lambda= \int_{\R}\frac{d|\nu_{\epsilon}|(\lambda)}{(1+|\lambda|)^{n+\epsilon}}.
		\end{align*}
		By setting   $\eta_{1,H,V}=\eta_{1}$, the above equality together with \eqref{Kr8}-\eqref{Kr9} completes the proof.
	\end{proof}
	
	\subsection{Higher order spectral shift functions}\label{sct:Higher order}
	The aim of this section is to provide all higher order spectral shift functions corresponding to a pair of self adjoint operators $(H,V)$ satisfying the Hypothesis~\ref{mainhypo}. 
	
	The following lemma is essential to reach our goal.
	\begin{lem}\label{divexpthm}
		Let $n,H,V$ satisfy Hypothesis \ref{mainhypo}. Write $\tilde{V}=V(H-i)^{-1}$. Then for $1\leq k\leq n-1$, and $f\in \mathfrak{Q}_n^k(\R)$, 
		\begin{align}\label{Hg4}
			&\bT_{f^{[k]}}^{H,\ldots,H}(V,\ldots,V)\\
			\nonumber=&\sum_{l=0}^k\sum_{\substack{j_1+\cdots+j_{l+1}=k\\ j_1,\ldots,j_l\geq1,j_{l+1}\geq0}}\sum_{r=0}^{\min(n-k,l)}(-1)^{k-l+r}
			\sum_{\substack{p_0\geq0,p_1,\ldots,p_r\geq1\\ p_0+\cdots+p_r=n-k}} T_{(fu^{n-k+l-r})^{[l-r]}}^{H,\ldots,H}(\tilde V^{j_1},\ldots,\tilde V^{j_{l-r}})\\
			&\nonumber\times(H-i)^{-p_0}\tilde V^{j_{l-r+1}}\cdots (H-i)^{-p_r}\tilde V^{j_{l+1}}.
		\end{align}
		We recall that by standard convention, a product $B_1\cdots B_m$ for $m<1$ is $1$, a list $(B_1,\ldots,B_{m})$ for $m<1$ is the empty list $()$, and, by definition, $T^{H}_{g^{[0]}}()=g(H)$.
	\end{lem}
	\begin{proof}
		By Theorem \ref{thm:bT in MOIs and bound}, it follows that for any $k\geq 1$, and $f\in\Q_n^k(\R)$,
		\begin{align}\label{Hg5}
			\bT_{f^{[k]}}^{H,\ldots,H}(V,\ldots,V)
			=\sum_{l=0}^{k}\sum_{\substack{j_1+\cdots+j_{l+1}=k\\j_1,\ldots,j_{l}\geq 1, j_{l+1}\geq 0}} (-1)^{k-l} T_{(fu^l)^{[l]}}^{H,\ldots,H}\big(\tilde{V}^{j_1},\ldots,\tilde{V}^{j_{l}}\big)\tilde{V}^{j_{l+1}}.
		\end{align}
		We may now apply \cite[Lemma 9]{vNvS21a} to each of the terms above, for $s=n-k$. Indeed, for all $l\in\{0,\ldots,k\}$, from Proposition \ref{Fourierprop} it follows that $fu^l$ is in the desired function class $\mathcal{W}^{n-k,l}$, in the notation of \cite[Section 3]{vNvS21a}.  
		The lemma follows immediately.
	\end{proof}

	The following theorem provides the existence of measures which are necessary to obtain the spectral shift functions of any order in the main theorem of this section.
	\begin{thm}\label{measure-existance}
		Let $n,H,V$ satisfy Hypothesis \ref{mainhypo}. Let $\epsilon\in (0,1]$. Then for each $k\in\N$,
		\begin{align*}
			\bT_{f^{[k]}}^{H,\ldots,H}(V,\ldots,V)\in\S^1,
		\end{align*}
		and
		there exists a complex Borel measure $\mu_{k,\epsilon}$ such that 
		\begin{align}\label{existance: higher order measures1}
			\Tr\left(\bT_{f^{[k]}}^{H,\ldots,H}(V,\ldots,V)\right)=\int_{\R}f^{(k)}(\lambda)u^{k\vee n+\epsilon}(\lambda)d\mu_{k,\epsilon}(\lambda)
		\end{align}	
		for every $f\in \mathfrak{Q}_n^k(\R)$, with 
		\begin{align}\label{existance: higher order measures2}
			\|\mu_{k,\epsilon}\|\leq c_{n,k,\epsilon}  \,\alpha_{n,V,H}^{k},
		\end{align}
		where $c_{n,k,\epsilon}$ is some positive constant, and $	\alpha_{n,V,H}$ is given by \eqref{kk3}.
	\end{thm}
	\begin{proof}
		
		For $1\leq k<n$: By Lemma \ref{divexpthm}, we have the identity \eqref{Hg4}. Next we show that each term in the right hand side of \eqref{Hg4} is a trace class operator and hence we can estimate the trace of those terms. Assuming the notations involved in \eqref{Hg4}, we note the following. 
		\begin{enumerate}[label=\textnormal{(\roman*)}]
			\item\label{Obs1} Recalling that $T_{g^{[0]}}^{H}()=g(H)$, we notice that the summands corresponding to $r=l$ are of the form of (a minus sign times)
			$$(fu^{n-k})(H)(H-i)^{-p_0}\tilde{V}^{j_{1}}\cdots(H-i)^{-p_l}\tilde{V}^{j_{l+1}}.$$
			Employing a similar argument as given for \eqref{trcl}, $fu^{2n}\in C_b(\R)$ implies that 
			$$ (fu^{n-k})(H)\tilde{V}^{j_{1}}\cdots(H-i)^{-p_l}\tilde{V}^{j_{l+1}}\in\S^1,$$
			and using the cyclicity property of the trace along with H\"older's inequality for Schatten norms, this yields 
			\begin{align}\label{Hg7}
				\nonumber&\left|\Tr\left(\big(fu^{n-k}\big)(H)(H-i)^{-p_0}\tilde{V}^{j_{1}}(H-i)^{-p_1}\cdots \tilde{V}^{j_{l}}(H-i)^{-p_{l}}\tilde{V}^{j_{l+1}}\right)\right|\\
				\nonumber=&\left|\Tr\left(\big(fu^{n-k}\big)(H)\tilde{V}^{j_{1}}(H-i)^{-p_1}\cdots \tilde{V}^{j_{l}}(H-i)^{-p_{l}}\tilde{V}^{j_{l+1}}(H-i)^{-p_{0}}\right)\right|\\
				\leq& 
				\,\|fu^{n-k}\|_\infty\,\alpha_{n,V,H}^k,
			\end{align}
			where $\alpha_{n,V,H}$ is given by \eqref{kk3}.
			
			\item\label{Obs2} For the rest of the terms in the summation of right hand side of \eqref{Hg4}: Now Proposition \ref{Fourierprop}\ref{item:func-2}, Corollary \ref{cor:bd_moi_bdd}, and H\"older's inequality for Schatten norms together imply
			\begin{align*}
				T_{\left(fu^{n-k+l-r}\right)^{[l-r]}}^{H,\ldots,H}\big(\tilde{V}^{j_1},\ldots,\tilde{V}^{j_{l-r}}\big)(H-i)^{-p_1}\tilde{V}^{j_{l-r+1}}\cdots (H-i)^{-p_{r+1}}\tilde{V}^{j_{l+1}}\in\S^1.
			\end{align*}
			If the product on the right of the multiple operator integral is nontrivial (either $r>0$ or $j_{l+1}>0$), then by applying H\"older's inequality for Schatten norms and using Theorem~\ref{inv-thm} we obtain		
			\begin{align}\label{Hg8}
				\nonumber&\left|\Tr\left(T_{\left(fu^{n-k+l-r}\right)^{[l-r]}}^{H,\ldots,H}\big(\tilde{V}^{j_1},\ldots,\tilde{V}^{j_{l-r}}\big)(H-i)^{-p_1}\tilde{V}^{j_{l-r+1}}\cdots (H-i)^{-p_{r+1}}\tilde{V}^{j_{l+1}}\right)\right|\\
				\leq &~d_{n,k} \, \|(fu^{n-k+l-r})^{(l-r)}\|_\infty\,\alpha_{n,V,H}^{k},
			\end{align}
			where $ d_{n,k}$ is a constant depending only on $n$ and $k$.
			
			The case where $r=0$ and $j_{l+1}=0$ takes a bit more care, because in the corresponding term $\Tr(T^{H,\ldots,H}_{(fu^{n-k+l})^{[l]}}(\tilde V^{j_1},\ldots,\tilde V^{j_l})(H-i)^{-(n-k)})$ the resolvent is not necessarily summable, and Theorem \ref{inv-thm} is not applicable for $\alpha=1$. Luckily, this case is similarly covered by \cite[Lemma 2.9(i)]{vNS23} after commuting the resolvent inside the multiple operator integral.
		\end{enumerate}
		
		\smallskip
		
		\noindent Thus, \eqref{Hg4} together with \ref{Obs1} and \ref{Obs2} implies 
		\begin{align*}
			\bT_{f^{[k]}}^{H,\ldots,H}(V,\ldots,V)\in\S^1.
		\end{align*}
		Having noted the remarks above, we apply the estimates \eqref{Hg7} and \eqref{Hg8}, the Hahn-Banach theorem, and the Riesz--Markov representation theorem, and conclude that there exist complex Borel measures $\mu_{n,k,l,r}$, $0\leq l\leq k$, $0\leq r \leq \min\{n-k,l\}$ such that 
		\begin{align}\label{eq:bound mu_{n,k,l,r}}
			\|\mu_{n,k,l,r}\|\leq  \tilde d_{n,k}\, \alpha_{n,V,H}^{k},
		\end{align}
		where $\tilde d_{n,k}$ is a constant depending only on $n$ and $k$, and
		\begin{align}\label{Hgg}
			\nonumber&\Tr\left(T_{f^{[k]}}^{H,\ldots,H}(V,\ldots,V)\right)\\
			\nonumber=& \sum_{l=0}^{k}\sum_{r=0}^{\min\{n-k,l\}}\int_\R \left(fu^{n-k+l-r}\right)^{(l-r)}(\lambda)\,d\mu_{n,k,l,r}(\lambda)\\
			\nonumber=& \sum_{l=0}^{k}\sum_{r=0}^{\min\{n-k,l\}}\sum_{s=0}^{l-r}
			\vect{l-r}{s}\int_\R f^{(s)}(\lambda)\,\left(u^{n-k+l-r}\right)^{(l-r-s)}(\lambda)\,d\mu_{n,k,l,r}(\lambda)\\
			=&\sum_{\substack{0\leq l\leq k\\0\leq r\leq \min\{n-k,l\}\\ 0\leq s\leq l-r}}
			\vect{l-r}{s}\frac{(n-k+l-r)!}{(n-k+s)!}\int_\R f^{(s)}(\lambda)\,u^{n-k+s}(\lambda)\,d\mu_{n,k,l,r}(\lambda).
		\end{align}
		
		We note that $s\leq l\leq k$.
		\noindent Thanks to Lemma \ref{lem:partial integration} and Equations \eqref{eq:bound mu_{n,k,l,r}} and \eqref{Hgg}, there exists a measure $\mu_{k,\epsilon}$ that satisfies \eqref{existance: higher order measures1} and \eqref{existance: higher order measures2}, which completes the proof for $k<n$.
		
		\smallskip
		
		\noindent For $k\geq n$: Thanks to Equation \eqref{Hg5}, Theorem \ref{inv-thm}, and Hölder's inequality, we can argue similarly to the case $k<n$ and conclude the existence of a measure $\mu_{k,\epsilon}$ satisfying \eqref{existance: higher order measures1} and \eqref{existance: higher order measures2}.
		This completes the proof.
	\end{proof}
	We proceed to present and demonstrate the central theorem of this section.  In this theorem and its proof, the Taylor remainder is denoted
	$$\mathcal R_{k,f,H}(V):=f(H+V)-\sum_{m=0}^{k-1}\frac{d^m}{dt^m}f(H+tV)\big|_{t=0}.$$

	\begin{thm}
		Let $n, H, V$ satisfy Hypothesis \ref{mainhypo}. Then for each fixed $k\in \N$, there exists a locally integrable real-valued function $\eta_k:=\eta_{k,H,V}$ such that 
		\begin{align}\label{Hg9}
			\int_\R \frac{|\eta_{k}(\lambda)|}{(1+|\lambda|)^{ n+k+\epsilon}}d\lambda\leq C_{n,k,\epsilon}\,\frac{1+b}{1-a}\,\max\limits_{\ell=1, \,k\vee n}\{\alpha_{n,V,H}^{\ell}\} 
		\end{align}	
		for all $\epsilon\in(0,1]$, where $C_{n,k,\epsilon}$ are some positive constants depending only on $n,k, \epsilon$, and $\alpha_{n,V,H}$ is given by \eqref{kk3}, and, moreover,
		\begin{align}\label{Hg10}
			\Tr\left(\mathcal{R}_{k,f,H}(V)\right)=\int_{\R}f^{(k)}(\lambda)\eta_{k}(\lambda)d\lambda
		\end{align}
		for each $f\in \mathfrak{Q}_n^k(\R) $. The locally integrable function $\eta_{k}$ is determined by \eqref{Hg10} uniquely up to a polynomial summand of degree at most $k-1$.
	\end{thm}
	\begin{proof}
		We prove \eqref{Hg9} and \eqref{Hg10} using the principle of mathematical induction on $k$. The base case, $k=1$, of the induction follows from Theorem~\ref{krmainthm}. Now we assume \eqref{Hg9} and \eqref{Hg10} are true for $k\in\{1,2,\ldots,m\}$ for some $m\in \Nat$. Let $\epsilon\in (0,1]$. Let $f\in\Q_n^{m+1}(\R)$.  Thanks to Theorem \ref{thm:higher derivatives in norm} and Theorem \ref{measure-existance}, there exists a finite complex Borel measure $\mu_{m,\epsilon}$ such that
		\begin{align}\label{Hg11}
			\nonumber\Tr\left(\mathcal{R}_{m+1,f,H}(V)\right)=&\Tr\left(\mathcal{R}_{m,f,H}(V)\right)- \Tr\left(\bT^{H,\ldots,H}_{f^{[m]}}(V,\ldots,V)\right)\\
			=&\int_{\R}f^{(m)}(\lambda)\eta_{m}(\lambda)d\lambda-\int_{\R}f^{(m)}(\lambda)u^{m\vee n+\epsilon}(\lambda)d\mu_{m,\epsilon}(\lambda).
		\end{align}
		Next, by performing integration by parts on the right hand side of \eqref{Hg11}, we obtain
		\begin{align}\label{Hg111}
			\nonumber &\Tr\left(\mathcal{R}_{m+1,f,H}(V)\right) =\int_{\R} f^{(m)}(\lambda)\,\eta_m(\lambda)d\lambda-\int_\R f^{(m)}(\lambda)\,u^{m\vee n+\epsilon}(\lambda)d\mu_{m,\epsilon}(\lambda)\\
			\nonumber=&\int_{\R}f^{(m+1)}(\lambda)\left(-\int_{0}^{\lambda}\eta_{m}(t)dt\right)\,d\lambda+\int_{\R}f^{(m+1)}(\lambda)\,u^{m\vee n+\epsilon}(\lambda)\,\mu_{m,\epsilon}((-\infty,\lambda])\,d\lambda\\
			\nonumber& -(m\vee n+\epsilon)\int_\R f^{(m+1)}(\lambda)\left(\int_{0}^{\lambda}u^{m\vee n+\epsilon-1}(t)\mu_{m,\epsilon}\left((-\infty,t]\right)dt\right)d\lambda\\
			=&\int_{\R}f^{(m+1)}(\lambda)\eta_{m+1}(\lambda)d\lambda,
		\end{align}
		where
		\begin{align}\label{Hg12}
			\nonumber&\eta_{m+1}(\lambda)\\
			=&u^{m\vee n+\epsilon}(\lambda)\mu_{m,\epsilon}((-\infty,\lambda])-\int_{0}^{\lambda}\eta_{m}(t)dt-(m\vee n+\epsilon)\int_{0}^{\lambda}u^{m\vee n+\epsilon-1}(t)\mu_{m,\epsilon}\left((-\infty,t]\right)dt.
		\end{align}
		Now, let us confirm that it follows from \eqref{existance: higher order measures2}, \eqref{Hg12}, Theorem \ref{krmainthm} and the induction hypothesis, that 
		\begin{align}\label{new_Hg}
			\int_\R \frac{|\eta_{m+1}(\lambda)|}{(1+|\lambda|)^{n+m+1+\epsilon}}\,d\lambda\leq \tilde{C}_{n,m+1,\epsilon}\, \frac{1+b}{1-a}\,\max\limits_{\ell=1, \,m\vee n}\{\alpha_{n,V,H}^{\ell}\},
		\end{align}
		for every $\epsilon\in(0,1]$, where $\tilde{C}_{n,m+1,\epsilon}$ is some positive constant. Indeed, the base case, $k=1$, of the induction follows from Theorem~\ref{krmainthm}. Now, assume that the estimate \eqref{Hg9} holds for $k = 1, 2, \ldots, m$. It follows from \eqref{Hg12} that
		
		\begin{align*}
			&\int_\R \frac{|\eta_{m+1}(\lambda)|}{(1+|\lambda|)^{n+m+1+\epsilon}}\, d\lambda\\
			\leq &\int_\R\frac{|u^{m\vee n+\epsilon}(\lambda)| \|\mu_{m,\epsilon}\|}{(1+|\lambda|)^{n+m+1+\epsilon}}\, d\lambda + \int_\R\int_\R \frac{1_{[0,\lambda]}(t)\,|\eta_{m}(t)|}{(1+|\lambda|)^{n+m+1+\epsilon}}\,dt\,d\lambda\\
			&+(m\vee n+\epsilon)\int_R \frac{\int_{0}^{\lambda} |u^{m\vee n+\epsilon-1}(t)|\,\|\mu_{m,\epsilon}\|\, dt}{(1+|\lambda|)^{n+m+1+\epsilon}}\, d\lambda\\
			\leq &\int_\R\frac{\|\mu_{m,\epsilon}\|}{(1+|\lambda|)^{n+m+1-m\vee n }}\,d\lambda + \int_\R \,|\eta_{m}(t)| \left(\int_\R \frac{1_{[0,\lambda]}(t)}{(1+|\lambda|)^{n+m+1+\epsilon}}\,d\lambda\right)\,dt\\
			&+(m\vee n+\epsilon)\int_R \frac{\|\mu_{m,\epsilon}\|}{(1+|\lambda|)^{n+m+1-m\vee n}}\, d\lambda.
		\end{align*}
		Now, by \eqref{existance: higher order measures2} and the induction hypothesis, the estimate \eqref{new_Hg} follows from the above. This completes the induction on $k$. 
		
		Since the left-hand side of \eqref{Hg10} is real-valued whenever $f$ is real-valued, we obtain that $Re~\eta_{k}$ (instead of $\eta_k$) satisfies \eqref{Hg10} for real-valued $f\in \mathfrak{Q}_n^k(\R) $ and, consequently, for all $f\in \mathfrak{Q}_n^k(\R) $. So we may assume $\eta_k$ is real valued.
		
		The uniqueness of $\eta_k$ can be established in a similar manner as we did in Theorem \ref{krmainthm}. Indeed, suppose $\eta_k$ and $\xi_k$ satisfy \eqref{Hg10}, and let $\gamma_k=\eta_k-\xi_k$. Then from \eqref{Hg10}, we conclude
		\begin{align}\label{Hg13}
			\int_\R f^{(k)}(\lambda)\gamma_k(\lambda)d\lambda=0\quad \text{for all }f\in C_c^\infty(\R).
		\end{align}
		Now consider the distribution $T_{\gamma_k}$ defined by 	
		$$T_{\gamma_k}(\phi)=\int_\R \phi(\lambda)\, \gamma_k(\lambda)d\lambda$$ for every $\phi\in C^\infty_c(\R)$.
		By \eqref{Hg13} and the definition of the derivative of a distribution, $T_{\gamma_k}^{(k)}=0$. Hence by \cite[Theorem 3.10 and Example 2.21]{gwaiz}, $\gamma_k(\lambda)d\lambda=f_k(\lambda)\,d\lambda$ for some polynomial $f_k$ of degree at most $k-1$. Consequently, $\eta_{k}\in L^1_{loc}(\R)$ satisfying \eqref{Hg10} is unique up to a polynomial summand of degree at most $k-1$.  Hence, we can assert that the estimates \eqref{Hg10} hold true for all $\epsilon \in (0,1]$.
		This completes the proof.
	\end{proof}

	%

	Next, we provide the existence of $n$-th order spectral shift functions under a relaxed assumption compared to the above, which extends \cite[Theorem 4.1]{vNS22} to the relatively bounded context.
	
	\begin{thm}\label{thm:nth order}
		Let $n\in \N$. Let $H$ be a self-adjoint (unbounded) operator in $\H$, and let $V$ be a symmetric and relatively $H$-bounded operator with $H$-bound $a\in[0,1)$, that is,
		$$\norm{V\psi}\leq a\norm{H\psi}+b\norm{\psi} \text{ for all } \psi\in\dom H,\text{ for some $b\in[0,\infty)$,}$$  such that $V(H-i)^{-1}\in\S^n$. Then, there exists $c_n>0$ and a real-valued function $\eta_n$
		such that
		\begin{align}\label{eta estimate}
			\int_\R |\eta_n(x)|\,\frac{dx}{(1+|x|)^{n+\epsilon}}\leq c_n\,(1+\epsilon^{-1})\, \frac{1+b}{1-a}\,\nrm{V(H-i)^{-1}}{n}^n\quad\text{for all } \epsilon >0
		\end{align}	
		and
		\begin{align}
			\label{tff}
			\Tr(\mathcal R_{n,f,H}(V))=\int_\R f^{(n)}(x)\eta_n(x)\,dx\,
		\end{align}
		for every $f\in\mathfrak{W}_n$ (see \cite[Definition 3.1]{vNS22}). The locally integrable function $\eta_n$ is determined by \eqref{tff} uniquely up to a polynomial summand of degree at most $n-1$.
	\end{thm}
	\begin{proof}
		The proof can be established along the same lines as the proof of \cite[Theorem 4.1]{vNS22}, along with \cite[Remark 4.8]{PrSk26}, by establishing analogous results to those of \cite[Lemma 4.8, Lemma 4.9]{vNS22} in the relatively bounded setting. The latter can be achieved in a similar way as we did in Lemma \ref{aprlem1} and Lemma \ref{aprlem2}, and Theorem \ref{krmainthm}. The replacement of the factor $(1+\|V\|)$ with $\frac{1+b}{1-a}$ is notable; these factors coincide if $V$ is bounded.
	\end{proof}
	
	\section{Applications}
	\subsection{Applications: bounded perturbations}\label{sct:Examples}
	
	In this section we collect some known classes of examples of self-adjoint operators $H$ and $V$ satisfying our assumptions, namely $V\in\mathcal B(\mathcal H)$ and
	\begin{align}\label{eq:assumption}
		V(H-i)^{-p}\in\mathcal S^{n/p}\qquad(p=1,\ldots,n).
	\end{align}
	We merely scratch the surface of the collection of examples; more can be found in \cite{vNS22,Rennie,SZ,Simon}, and references therein.
	
	A result about Schatten class membership typically comes along with a bound on the Schatten norm, and by our main result these bounds imply estimates on the weighted integral norm of the spectral shift functions, for each specific situation. For brevity we omit these bounds, as they can be found in the respective cited references. Such bounds are called Cwikel estimates and Birman--Solomyak estimates, and form an active area of research \cite{LSZ,MP}.
	
	\subsubsection{Dirac and Laplace operators on Euclidean spaces}
	If, for suitable functions $f,g$ on $\R^d$, the perturbation $V= M_f$ is the operator of multiplication by $f$, and moreover $(H-i)^{-p}=g(-i\nabla)$, where $\nabla$ is the vector whose components are the operators $\{\partial_i\}_{i=1}^d$, then conditions for \eqref{eq:assumption} to hold can be found in \cite{Cwikel} and \cite[Chapter 4]{Simon}.
	As shown in the latter reference, such results typically require different proofs for $p\leq n/2$ and $p>n/2$. Moreover, the spaces
	$$\ell^p(L^q)(\R^d):=\left\{f:\R^d\to\C\text{ measurable}~:~\sum_{k\in\mathbb{Z}^d}(\|f\upharpoonright_{(0,1)^d+k}\|_{L^q(((0,1)^d+k))})^p<\infty\right\},$$
	for $p\in[1,\infty)$ and $q\in[1,\infty]$, are frequently useful.
	
	Two especially important examples of $H$ are the Laplacian $H=\Delta=\sum_{k=1}^d\frac{\partial^2}{\partial x_k^2}$ in $L^2(\R^d)$ and the Dirac operator $H=D$, defined as
	$$D:=-i\sum_{k=1}^d e_k\otimes\frac{\partial}{\partial x_k}$$
	in $\C^N\otimes L^2(\R^d)$, where $e_1,\ldots,e_d\in M_N(\C)$ are self-adjoint matrices satisfying $e_k e_l+e_l e_k=2\delta_{k,l}$ for all $k,l\in\{1,\ldots,d\}$. Here $N=2^{d/2}$ if $d$ is even and $N=2^{(d+1)/2}$ if $d$ is odd. 
	The following result is proven in \cite[Theorem III.4]{S21}.
	
	\begin{thm}[\cite{S21}]
		
		\label{thm:summability Dirac and Laplace flat space}
		Let $m,d\in\N$ and $f\in\ell^1(L^2)(\R^d)\cap L^\infty(\R^d)$ be real-valued, with associated multiplication operator $M_f$.
		\begin{enumerate}[label=\textnormal{(\alph*)}]
			\item If $n>d$, then
			$$(I\otimes M_f)(D-i)^{-p}\in\S^{n/p},\qquad p=1,\ldots,n.$$
			\item If $n>\frac{d}{2}$, then 
			$$M_f(\Delta-i)^{-p}\in\S^{n/p},\qquad p=1,\ldots,n.$$
		\end{enumerate}
	\end{thm}
	
	\subsubsection{The Hodge--Dirac operator on a Riemannian space}
	Part of the above result generalizes to suitable Riemannian manifolds, in the sense of the following theorem. For its proof we refer to \cite[Equation (3.11)]{SZ}, which in that paper is proven in order to obtain \cite[Theorem 3.4.1]{SZ} by specializing to $p=d$. 
	\begin{thm}[\cite{SZ}]
		Let $(X,g)$ be a second countable $d$-dimensional complete smooth Riemannian manifold. Let $\Omega^k_c(X)$ be the space of smooth compactly supported differential k-forms, and $\H_k$ its completion with respect to the natural Hilbert space structure defined by $g$. Defined in the Hilbert space $\mathcal{H}:=\oplus_{k=0}^d \H_k$, let $D_g$ be the associated Hodge-Dirac operator. For $f\in C_c^\infty(X)$, let $M_f$ denote the associated multiplication operator on $\H$. Then, for all $p\geq1$ we have
		$$M_f(D_g-i)^{-p}\in\mathcal{S}^{d/p,\infty}.$$
	\end{thm}
	Our main result thus becomes applicable by setting $n=d+1$, and, independently of dimension, implies existence of spectral shift functions of all orders for Dirac operators on manifolds as in the above theorem.
	
	In light of the above, it is reasonable to expect a generalization of Theorem \ref{thm:summability Dirac and Laplace flat space} to exist for suitable Riemannian manifolds, but this is well beyond the scope of this text. A similar remark applies to matrix-valued differential operators and pseudodifferential operators.
	
	\subsubsection{Noncommutative Geometry}
	Noncommutative geometry is a prime motivator for this paper, as it too is involved with proving results about geometrical objects -- such as differential operators -- without reference to their underlying space. The typical benefit for the noncommutative geometer is that the thus obtained results hold for more general objects which may not be geometrical, but may still be useful to describe physics.

	A first application of the spectral shift function is to the theory of spectral flow \cite{ACDS}, as used in index theory. Typically, one takes
	$$H=D,\qquad V=u[D,u^*],\qquad\text{so that}\qquad H+V=uDu^*,$$
	for a unitary $u$. More generally, perturbations in gauge theory take the form $V=\sum_{j=1}^n a_j[D,b_j]$, when $a_j$ and $b_j$ are elements of a `smooth' algebra $\mathcal A\subseteq\mB(\H)$.
	Higher order spectral shift functions also provide analytical information on the Taylor remainders of the spectral action, which may be useful for the description of classical and quantum field theories \cite{ILV,oneloop}.
	
	
	
	The following result shows that our Hypothesis \ref{mainhypo} holds for a wide class of nonunital spectral triples.
	It is a special case of \cite[Proposition 10]{Rennie}. 

	\begin{thm}
		[\cite{Rennie}]
		Let $(\A,\H,\mathcal D)$ be a smooth local $(d,\infty)$-summable spectral triple in the sense of \cite[Definition 7]{Rennie}.  Then for all $a\in\mB(\H)$ satisfying $a\phi=\phi a=a$ for some $\phi\in\mathcal A_c$ we have 
		$$a(\mathcal D-i)^{-p}\in\S^{(d+1)/p},\qquad p=1,\ldots,d+1.$$
	\end{thm}
	
	In particular cases one may expect strenghtenings of the above result. This is indeed the case for the Moyal plane $\R^d_\theta$, which is a prominent motivating example of a nonunital spectral triple.
	The result below follows from \cite{LSZ}, as is written out explicitly below, for convenience of the reader. 
	\begin{thm}[\cite{LSZ}]
		Let $d\geq 2$ be even, $V\in W^{d,1}(\R^d_\theta)$, $h\in\N$, and set $n:=\lfloor \frac{d}{h}\rfloor+1$. For all $p\in\{1,\ldots,n\}$, if we assume that $g\in \ell^{d/(hp),\infty}(L^\infty)(\R^d)$, then we have that $Vg(-i\nabla_\theta)\in\S^{n/p}$. In particular,
		$$V(D_\theta-i)^{-p}\in\S^{n/p}\qquad(p\in\{1,\ldots,n\},~n=d+1),$$
		and
		$$V(\Delta_\theta-i)^{-p}\in\S^{n/p}\qquad(p\in\{1,\ldots,n\},~n=d/2+1),$$	
		where $D_\theta$ and $\Delta_\theta$ are the noncommutative Dirac and Laplace operators (see \cite[p. 29]{LSZ}).
	\end{thm}
	\begin{proof}
		
		A combination of \cite[Proposition 6.15(iii)]{LSZ}, \cite[Definition 6.14]{LSZ}, and \cite[Definition 6.7]{LSZ} yields $V\in W^{d,q}(\R^d_\theta)$ and thus $V\in L^q(\R^d_\theta)$ for all $q\geq1$. 
		
		We define $n:=\lfloor\frac{d}{h}\rfloor+1$ and let $p\in\{1,\ldots,n\}$, $h\in\N$, and $g\in \ell^{d/(hp),\infty}(L^\infty)(\R^d)\subseteq L^{d/(hp),\infty}(\R^d)$. 	
		
		First suppose that $p\leq n/2$. Setting $q:=d/(hp)$, a case distinction shows that $q\geq1$. We find $g\in L^{q,\infty}(\R^d)$ and $V\in L^q(\R^d_\theta)$, hence \cite[Lemma 2.3]{LSZ} implies that $V\otimes g\in\mathcal L_{q,\infty}(\mathcal{N})$, the noncommutative $L^{q,\infty}$-space of the von Neumann algebra $\mathcal{N}:=L^\infty(\R^d_\theta)\otimes L^\infty(\R^d)$. In particular, since $n/p> d/(hp)$, we have $V\otimes g\in (\mathcal{L}_{n/p}\cap \mathcal{L}_\infty)(\mathcal{N})$. We note that $\frac{n}{p}\geq2$ by assumption, so $L^{n/p}\cap L^\infty$ is an interpolation space for $(L^2,L^\infty)$. Therefore, \cite[Theorem 7.2]{LSZ} gives
		$$Vg(-i\nabla_\theta)\in \mathcal L_{n/p}(\mathcal B(L^2(\R^d)))=\mathcal S^{n/p}.$$
		
		Now suppose that $p>n/2$. Set $q:=\frac{n}{p}\in[1,2)$ and note that $g\in \ell^q(L^\infty)(\R^d)$. We have $V\in W^{d,q}(\R^d_\theta)$ by Sobolev embedding, hence \cite[Theorem 7.7]{LSZ} implies that
		$$Vg(-i\nabla_\theta)\in\mathcal L_q(\mathcal B(L^2(\R^d)))=\mathcal S^{n/p},$$
		completing the proof of the first statement.
		
		To prove the second statement, one employs the functions $g$ on $\R^d$ defined by
		$$g(x):=(\gamma \cdot x-i)^{-p},$$
		$$g(x):=(\|x\|^2-i)^{-p},$$
		where $\gamma$ is the vector of Clifford matrices (see \cite[Definition 6.17]{LSZ}).
		We indeed have $g\in \ell^{d/(hp),\infty}(L^\infty)(\R^d)$ for $h=1$, $h=2$, respectively. In the former case $g$ is matrix-valued and this inclusion can be interpreted component-wise, the finite component $\C^{2^{\lfloor d/2\rfloor}}$ of the Hilbert space factors through, and one indeed obtains $V(D_\theta-i)^{-p}\in\S^{n/p}(\C^{2^{\lfloor d/2\rfloor}}\otimes L^2(\R^d))$ as well as $V(\Delta_\theta-i)^{-p}\in\S^{n/p}(L^2(\R^d))$.
	\end{proof}
	\subsection{Applications: unbounded perturbations}
	Atomic Hamiltonians provide a rich class of examples in which the perturbation $V$ is relatively bounded.
	The most basic example is the Hamiltonian of the hydrogen atom,
	i.e., the unbounded operator
	$$H+V=-\frac{\hbar^2}{2m}\sum_i\frac{\partial^2}{\partial x_i^2}+V,$$
	acting in $L^2(\R^d)$, where $V$ is the Coulomb potential, namely, the multiplication operator
	$$V\psi(x)=-\frac{e^2}{4\pi\epsilon_0}\frac{1}{\norm{x}}\psi(x),$$
	for physical constants $\hbar,m,e,\epsilon_0\in\R$. A celebrated result of Kato (\cite{Kato}) tells us that, indeed, $V$ is relatively $H$-bounded. This result does not change if $V$ is replaced by any other $L^2(\R^d)+L^\infty(\R^d)$-function, or if a larger number of possibly interacting particles are considered, either in $H$ or $V$ or both. In these atomic or molecular quantum systems, $a=0$, and therefore \eqref{eq:intro:operator function} exists on $(-\infty,\infty)$, whereas $a$ becomes nonzero in the `relativistic case' where $H$ is replaced with a Dirac operator $D$, and \eqref{eq:intro:operator function} exists on a bounded interval $(-\frac{1}{a},\frac{1}{a})$. Clearly, the relatively bounded condition allows much more than these two classes of examples.
	
	%
	%
	%
	
	\subsubsection{Finite-extent Coulomb potentials}
	In order to apply our results on the existence and regularity properties of the spectral shift functions (at all order), one needs to establish the $n/p$-summability of $V(H-i)^{-p}$.
	
	It is known that the long-range behaviour of the Coulombic potential on $L^2(\R^3)$ obstructs Cwikel or Birman--Solomyak estimates. 
	The following is a consequence of \cite[Proposition 4.7]{Simon}:
	\begin{lem}
		If $M_f(i+\Delta)^{-n}\in \S^1$ for some $n\in\N$, then $f\in \ell^1(L^2)(\R^d)$.
	\end{lem}
	As the next best thing (which suffices for many situations) one may consider spatially truncated (finite-extent) Coulomb potentials. By this we mean a compactly supported (or sufficiently decaying) function $f$ which satisfies $f(x)=\frac{1}{\|x\|}$ for $x$ in a compact region in $\R^d$.
	For such $f$, we have $f\in \ell^1(L^p)(\R^d)$ precisely when $p<d$.
	%
	In particular, a finite-extent Coulomb potential on $\R^3$ is an example of an element in $\ell^1(L^2)(\R^3)$. In this case indeed, spectral shift functions of all orders exist.
	
	\begin{thm}\label{lem:Coulomb Laplacian}
		For any $f\in \ell^1(L^2)(\R^3)$ and $\Delta$ the Laplacian on $\R^3$ we have
		$$M_f(i+\Delta)^{-p}\in\S^{4/p}\qquad(p=1,2,3,4).$$
	\end{thm}
	\begin{proof}
		Let $m\in\{2,3,4\}$.
		Define $g_m:\R^3\to\C$ by $$g_m(x):=(i-\|x\|^2)^{-m}.$$
		Take $p_S:=\frac{4}{m}\in[1,2]$. Then $f\in\ell^1(L^2)\subseteq\ell^{p_S}(L^2)$ and $g_m\in \ell^{p_S}(L^2)$. By \cite[Theorem 4.5]{Simon}, therefore
		$$M_f(i+\Delta)^{-m}=M_fg_m(-i\nabla)\in\S^{4/m}.$$
		
		We moreover note that $g_1\in L^2(\R^3)$ and $f\in\ell^1(L^2)(\R^3)\subseteq L^2(\R^3)$. So by \cite[Theorem 4.1]{Simon}, we obtain
		$$M_f(i+\Delta)^{-1}=M_fg_1(-i\nabla)\in\S^2\subseteq\S^4,$$
		finishing the proof.
	\end{proof}
	
	It is clear from the above proof that there are various ways to improve Theorem \ref{lem:Coulomb Laplacian}; this is however beyond the scope of this text. For further results we refer to \cite{Yafaev,Yafaev92,KatoBook,ReedSimonI} and references therein.
	
		%

	\appendix
	
	\section{Proofs of MOI identities}
	We give proofs of the identities claimed in \textsection\ref{sct:subsct:properties MOI}. Throughout, whenever $\mJ\subseteq\{1,\ldots,n\}$ is such that $V_j\in\mB(\H)$ for $j\notin\mJ$, we use the notation
	\begin{align}\label{eq:notation bT}
		\bT_{\phi,\mJ}(V_1,\ldots,V_n):=\bT_{\phi}(V_1,\ldots,V_n)
	\end{align}
	to emphasise that $V_j$ is possibly unbounded for $j\in\mJ$. In particular, $\bT_{\phi,\emptyset}=T_{\phi}$.
	
	\subsection{Change of variables}
	
	In the case of bounded perturbations $V_1,\ldots,V_n$, we recall the following result proved and used by Skripka and coathors in \cite{CS18,vNS22,vNS23,vNvS21a}.
	\begin{thm}[change of variables]\label{thm:cov}
		Let $H_0,\ldots,H_n$ be self-adjoint and let $V_1,\ldots,V_n\in\mB(\H)$ be bounded. For all $f\in C^n(\R)$ with $f^{(n)},f^{(n-1)},(fu)^{(n)}\in W_0(\R)$, we have for all $j\in\{1,\ldots,{n-1}\}$,
		\begin{align*}
			T^{H_0,\ldots,H_n}_{f^{[n]}}(V_1,\ldots,V_n)=&T^{H_0,\ldots,H_n}_{(fu)^{[n]}}(V_1,\ldots,V_{j-1},V_j(H_j-i)^{-1},V_{j+1},\ldots,V_n)\\
			&-T^{H_0,\ldots,H_{j-1},H_{j+1},\ldots,H_n}_{f^{[n-1]}}(V_1,\ldots,V_{j-1},V_j(H_j-i)^{-1}V_{j+1},V_{j+2},\ldots,V_n).
		\end{align*}
		For the boundary case $j=n$ we similarly have
		\begin{align*}
			T^{H_0,\ldots,H_n}_{f^{[n]}}(V_1,\ldots,V_n)=&T^{H_0,\ldots,H_n}_{(fu)^{[n]}}(V_1,\ldots,V_{n-1},V_n(H_n-i)^{-1})\\
			&-T^{H_0,\ldots,H_{n-1}}_{f^{[n-1]}}(V_1,\ldots,V_{n-1})V_n(H_n-i)^{-1},
		\end{align*}
		and for the boundary case $j=0$ we have
		\begin{align*}
			T^{H_0,\ldots,H_n}_{f^{[n]}}(V_1,\ldots,V_n)=&T^{H_0,\ldots,H_n}_{(fu)^{[n]}}((H_0-i)^{-1}V_1,V_2\ldots,V_n)\\
			&-(H_0-i)^{-1}V_1T^{H_1,\ldots,H_n}_{f^{[n-1]}}(V_2,\ldots,V_n).
		\end{align*}
	\end{thm}
	We extend this theorem to relatively bounded arguments as follows. The following theorem is a rephrasing of Theorem \ref{thm:cov rel bdd} in the explicit notation \eqref{eq:notation bT}.
	\begin{thm}[change of variables, relatively bounded]\label{thm:cov rel bdd2}
		Let $H_0,\ldots,H_n$ be self-adjoint. Let $\mJ\subseteq\{1,\ldots,n\}$ be a subset so that $V_k$ is bounded for each $k\in\{1,\ldots,n\}\setminus\mJ$ and $V_k$ is relatively $H_k$-bounded for each $k\in\mJ$. For each $f\in\mW_{|\mJ|}^n(\R)$ and each $j\in\{0,1,\ldots,n\}$ (the boundary cases $j=0$ and $j=n$ being understood as in Theorem \ref{thm:cov}) we have
		\begin{align*}
			\bT^{H_0,\ldots,H_n}_{f^{[n]},\mJ}(V_1,\ldots,V_n)=&\bT^{H_0,\ldots,H_n}_{(fu)^{[n]},\mJ\setminus\{j\}}(V_1,\ldots,V_{j-1},V_j(H_j-i)^{-1},V_{j+1},\ldots,V_n)\\
			&-\bT^{H_0,\ldots,H_{j-1},H_{j+1},\ldots,H_n}_{f^{[n-1]},p_j^{-1}(\mJ)}(V_1,\ldots,V_{j-1},V_j(H_j-i)^{-1}V_{j+1},V_{j+2},\ldots,V_n),
		\end{align*}
		where $p_j:\{1,\ldots,n-1\}\to\{1,\ldots,n\}$ is the order-preserving map whose range excludes $j$.
	\end{thm}
	\begin{proof}
		If $j\notin\mJ$, then the theorem follows directly from Definition \ref{def:MOI relatively bdd} and Theorem \ref{thm:cov}, so we assume that $j\in\mJ$. For notational simplicity, we moreover assume that $j=1$. We write $\alpha=(\alpha_1,\ldots,\alpha_n)\in\{0,1\}^n$ where $\alpha_l=1$ if and only if $l\in\mJ$. We also abbreviate $R_l:=(H_l-i)^{-1}$. By definition,
		\begin{align}\label{eq:MOI VR}
			\bT^{H_0,\ldots,H_n}_{f^{[n]},\mJ}(V_1,\ldots,V_n)&=T^{H_0,\ldots,H_n}_{f^{[n]}u^\alpha}(V_1R_1^{\alpha_1},\ldots,V_nR_n^{\alpha_n}).
		\end{align}
		We note that
		$$f^{[n]}(\lambda_0,\ldots,\lambda_n)=(fu)^{[n]}(\lambda_0,\ldots,\lambda_n)u^{-1}(\lambda_1)-f^{[n-1]}(\lambda_0,\lambda_2,\ldots,\lambda_n)u^{-1}(\lambda_1),$$
		so also, for $\lambda=(\lambda_0,\ldots,\lambda_n)$,
		$$(f^{[n]}u^\alpha)(\lambda)=((fu)^{[n]}u_1^{-1}u^{\alpha})(\lambda)-f^{[n-1]}(\lambda_0,\lambda_2,\ldots,\lambda_n)u^{-1}(\lambda_1)u^\alpha(\lambda),$$
		and $u^{-1}(\lambda_1)u^\alpha(\lambda)=u^{\alpha_2}(\lambda_2)\cdots u^{\alpha_n}(\lambda_n)$.
		Therefore, denoting $\alpha'=(\alpha_2,\ldots,\alpha_n)$, we have
		\begin{align}\label{eq:divdif rule with u k}
			(f^{[n]}u^\alpha)(\lambda)=((fu)^{[n]}u^{\alpha-\delta_1})(\lambda)-(f^{[n-1]}u^{\alpha'})(\lambda_0,\lambda_2,\ldots,\lambda_n),
		\end{align}
		where $\delta_1=(1,0,\ldots,0)$.
		Note that $|\mJ|=|\alpha|$. As $f\in\mW_{|\alpha|}^n(\R)$, we have $fu\in\mW_{|\alpha-\delta_1|}^n(\R)$ and $f\in\mW_{|\alpha'|}^{n-1}(\R)$ by Lemma \ref{lem:function spaces inclusions}\ref{item:2 function spaces inclusions}. Hence, by \eqref{eq:MOI VR}, \eqref{eq:divdif rule with u k}, and Theorem \ref{thm:MOI well-defined special case} we have
		\begin{align*}
			\bT^{H_0,\ldots,H_n}_{f^{[n]},\mJ}(V_1,\ldots,V_n)=&T^{H_0,\ldots,H_n}_{(fu)^{[n]}u^{\alpha-\delta_1}}(V_1R_1^{\alpha_1},V_2R_2^{\alpha_2},\ldots,V_nR_n^{\alpha_n})\\
			&-T^{H_0,H_2,\ldots,H_n}_{f^{[n-1]}u^{\alpha'}}(V_1R_1^{\alpha_1}V_2R_2^{\alpha_2},V_3R_3^{\alpha_3},\ldots,V_nR_n^{\alpha_n})\\
			=&\bT^{H_0,\ldots,H_n}_{(fu)^{[n]},\mJ\setminus\{1\}}(V_1(H_1-i)^{-1},V_2,\ldots,V_n)\\
			&-\bT^{H_0,H_2,\ldots,H_n}_{f^{[n-1]},p_1^{-1}(\mJ)}(V_1(H_1-i)^{-1}V_2,V_3,\ldots,V_n),
		\end{align*}
		which completes the proof.
	\end{proof}

	\subsection{Superscript difference}
	

	The following proposition is an extension of Theorem \ref{thm:difference} to higher order, and an extension of \cite[Theorem 4.3.14]{ST19} to relatively bounded differences $A-B$, and is a step towards its generalization involving general multilinear symbols $\mathbf T^{H_0,\ldots,H_n}_{f^{[n]},\mJ}$.
	
	\begin{prop}\label{prop:superscript cov}
		Let $n\in\N_{\geq1}$ and $j\in\{1,\ldots,n\}$. Let $H_0,\ldots,H_{n}$ be self-adjoint, let $V_1,\ldots,V_{n-1}\in\mB(\H)$, and let $V_n$ be a relatively $H_n$-bounded operator. Let $V$ be symmetric and relatively $H_j$-bounded with $H_j$-bound $<1$ and suppose that $V_j$ is relatively $H_j+V$-bounded (for instance, if $V$ is a scalar multiple of $V_j$). Write $A=H_j+V$, $B=H_j$. For all $f\in\mW^{n+1}_2(\R)$ (i.e. $f\in C^{n+1}(\R)$ such that $f^{(n-1)},(fu)^{(n)},(fu^2)^{(n+1)}\in W_0(\R)$) we have
		\begin{align*}
			&\bT^{H_0,\ldots,H_{j-1},A,H_{j+1},\ldots,H_n}_{f^{[n]},\{n\}}(V_1,\ldots,V_n)-\bT^{H_0,\ldots,H_{j-1},B,H_{j+1},\ldots,H_n}_{f^{[n]},\{n\}}(V_1,\ldots,V_n)\\
			&\qquad\qquad\qquad\qquad=\bT_{f^{[n+1]},\{j+1,p_{j+1}(n)\}}^{H_0,\ldots,H_{j-1},A,B,H_{j+1},\ldots,H_n}(V_1,\ldots,V_j,A-B,V_{j+1},\ldots,V_n),
		\end{align*}
		where $p_{j+1}:\{1,\ldots,n\}\to\{1,\ldots,n+1\}$ is the order-preserving map whose range excludes $j+1$.
	\end{prop}
	\begin{proof}
		We prove the statement for $j=1$ for notational simplicity; the proof for other values of $j$ is completely analogous. Similarly we assume $n\geq2$.
		We shall denote $\tilde V_n=V_n(H_n-i)^{-1}$. It follows from Definition \ref{def:MOI relatively bdd} together with the fact that $f^{[n]}u^{(0,\ldots,0,1)}\in \BS(\R^{n+1})$ that
		\begin{align}\label{eq:new-1}
			\nonumber\bT^{H_0,A,H_{2},\ldots,H_n}_{f^{[n]},\{n\}}(V_1,\ldots,V_n)(H_n-i)^{-1} =& T^{H_0,A,H_{2},\ldots,H_n}_{f^{[n]}u^{(0,\ldots, 0, 1)}}(V_1,\ldots,\tilde V_n)(H_n-i)^{-1}\\
			\nonumber=&  T^{H_0,A,H_{2},\ldots,H_n}_{f^{[n]}u^{(0,\ldots, 0, 1)}}(V_1,\ldots,\tilde V_n(H_n-i)^{-1})\\
			\nonumber=&\bT^{H_0,A,H_{2},\ldots,H_n}_{f^{[n]}, \{n\}}(V_1,\ldots,\tilde V_n)\\
			=&T^{H_0,A,H_{2},\ldots,H_n}_{f^{[n]}}(V_1,\ldots,\tilde V_n),
		\end{align}
		where the last equality follow from Lemma \ref{thm:well defined}. Since $(fu)^{(n)}\in W_0(\R)$, by Corollary \ref{cor:bd_moi_bdd}, we have
		\begin{align}\label{eq:new-2}
			\nonumber&T^{H_0,A,H_2,\ldots,H_n}_{(fu)^{[n]}}(V_1(A-i)^{-1},V_2,\ldots,V_n)\\
			\nonumber=&\int_\R\int_{\Delta_n}e^{is_0xH_0}V_1(A-i)^{-1}e^{is_1xA}V_2e^{is_2xH_2}\cdots V_ne^{is_nxH_n}\,ds\,d\mu(x)\\
			\nonumber=&\int_\R\int_{\Delta_n}e^{is_0xH_0}V_1e^{is_1xA}(A-i)^{-1}V_2e^{is_2xH_2}\cdots V_ne^{is_nxH_n}\,ds\,d\mu(x)\\
			=&T^{H_0,A,H_2,\ldots,H_n}_{(fu)^{[n]}}(V_1,(A-i)^{-1}V_2,\ldots,V_n),
		\end{align}
		where $(fu)^{(n)}(x)=\int e^{ixy} d\mu(y)$.
		
		Now by first using the identity \eqref{eq:new-1}, then using change-of-variables (Theorem \ref{thm:cov}), then using the identity \eqref{eq:new-2}, then using the second resolvent identity (Lemma \ref{lem:second resolvent identity}), we obtain
		\begin{align}\label{eq:rewriting with Vn tilde}
			\nonumber&\Big(\bT^{H_0,A,H_{2},\ldots,H_n}_{f^{[n]},\{n\}}(V_1,\ldots,V_n)-\bT^{H_0,B,H_{2},\ldots,H_n}_{f^{[n]},\{n\}}(V_1,\ldots,V_n)\Big)(H_n-i)^{-1}\\
			\nonumber=&\,T^{H_0,A,H_{2},\ldots,H_n}_{f^{[n]}}(V_1,\ldots,\tilde V_n)-T^{H_0,B,H_{2},\ldots,H_n}_{f^{[n]}}(V_1,\ldots,\tilde V_n)\\
			\nonumber=&\,T^{H_0,A,H_{2},\ldots,H_n}_{(fu)^{[n]}}(V_1,(A-i)^{-1}V_{2},V_{3},\ldots,\tilde V_n)-T^{H_0,H_{2},\ldots,H_n}_{f^{[n-1]}}(V_1(A-i)^{-1}V_{2},V_{3},\ldots,\tilde V_n)\\
			\nonumber&\quad-T^{H_0,B,H_{2},\ldots,H_n}_{(fu)^{[n]}}(V_1,(B-i)^{-1}V_{2},V_{3},\ldots,\tilde V_n)+T^{H_0,H_{2},\ldots,H_n}_{f^{[n-1]}}(V_1(B-i)^{-1}V_{2},V_{3},\ldots,\tilde V_n)\\
			\nonumber=&\,-T^{H_0,A,H_{2},\ldots,H_n}_{(fu)^{[n]}}(V_1,(A-i)^{-1}(A-B)(B-i)^{-1}V_{2},V_{3},\ldots,\tilde V_n)\\
			\nonumber&\quad+T^{H_0,A,H_{2},\ldots,H_n}_{(fu)^{[n]}}(V_1,(B-i)^{-1}V_{2},V_{3},\ldots,\tilde V_n)-T^{H_0,B,H_{2},\ldots,H_n}_{(fu)^{[n]}}(V_1,(B-i)^{-1}V_{2},V_{3},\ldots,\tilde V_n)\\
			&\quad+T^{H_0,H_{2},\ldots,H_n}_{f^{[n-1]}}(V_1(A-i)^{-1}(A-B)(B-i)^{-1}V_{2},V_{3},\ldots,\tilde V_n).
		\end{align}
		The second and third term on the right-hand side can be combined as follows. Both multiple operator integrals have purely bounded arguments, so when applied to a vector $\psi\in\H$ they can be written as $\H$-valued integrals over a finite measure obtained from the Fourier transform of $(fu)^{(n)}\in W_0(\R)$. Note that $(fu)^{(n+1)}\in W_0(\R)$. Arguing as in the proof of \cite[Eq.\,(5.6)]{ACDS}, namely by first invoking Corollary \ref{cor:bd_moi_bdd}, then applying the weighted Duhamel formula (Lemma \ref{lem:Duhamel}), and finally using Fubini's theorem (\cite[Lemma 3.8]{ACDS}), we obtain for $\psi\in\H$ that
		\begin{align*}
			&T^{H_0,A,H_{2},\ldots,H_n}_{(fu)^{[n]}}(V_1,(B-i)^{-1}V_{2},V_{3},\ldots,\tilde V_n)\psi-T^{H_0,B,H_{2},\ldots,H_n}_{(fu)^{[n]}}(V_1,(B-i)^{-1}V_{2},V_{3},\ldots,\tilde V_n)\psi\\
			=&\int_\R\int_{\Delta_n}e^{is_0xH_0}V_1\big(e^{is_1xA}-e^{is_1xB}\big)(B-i)^{-1}V_2e^{is_2xH_2}\cdots \tilde V_ne^{is_nxH_n}\,\psi\,ds\,d\mu(x)\\
			=& \,i\int_\R\int_{\Delta_n}e^{is_0xH_0}V_1\Big(\int_{0}^{s_1}e^{itxA}(A-B)e^{i(s_1-t)xB}\,dt\Big)(B-i)^{-1}V_2e^{is_2xH_2}\cdots \tilde V_ne^{is_nxH_n}\,\psi\,ds\,d\mu(x)\\
			=&\, i\int_\R\int_{\Delta_n}\int_{0}^{s_1}e^{is_0xH_0}V_1e^{itxA}(A-B)(B-i)^{-1}e^{i(s_1-t)xB}V_2e^{is_2xH_2}\cdots \tilde V_ne^{is_nxH_n}\,\psi\,dt\,\,ds\,d\mu(x)\\
			=&\,T^{H_0,A,B,H_{2},\ldots,H_n}_{(fu)^{[n+1]}}(V_1,(A-B)(B-i)^{-1},V_{2},V_{3},\ldots,\tilde V_n)\psi,
		\end{align*}
		where in the third equality, we use that $(B-i)^{-1}e^{i(s_1-t)xB}=e^{i(s_1-t)xB}(B-i)^{-1}$. Therefore,
		\begin{align*}
			&T^{H_0,A,H_{2},\ldots,H_n}_{(fu)^{[n]}}(V_1,(B-i)^{-1}V_{2},V_{3},\ldots,\tilde V_n)\psi-T^{H_0,B,H_{2},\ldots,H_n}_{(fu)^{[n]}}(V_1,(B-i)^{-1}V_{2},V_{3},\ldots,\tilde V_n)\\
			=&T^{H_0,A,B,H_{2},\ldots,H_n}_{(fu)^{[n+1]}}(V_1,(A-B)(B-i)^{-1},V_{2},V_{3},\ldots,\tilde V_n).
		\end{align*}
		Substituting this in \eqref{eq:rewriting with Vn tilde}, we find
		\begin{align*}
			X(H_n-i)^{-1}:=&\Big(\bT^{H_0,A,H_{2},\ldots,H_n}_{f^{[n]},\{n\}}(V_1,\ldots,V_n)-\bT^{H_0,B,H_{2},\ldots,H_n}_{f^{[n]},\{n\}}(V_1,\ldots,V_n)\Big)(H_n-i)^{-1}\\
			=&\,-T^{H_0,A,H_{2},\ldots,H_n}_{(fu)^{[n]}}(V_1,(A-i)^{-1}(A-B)(B-i)^{-1}V_{2},V_{3},\ldots,\tilde V_n)\\
			&\quad+T^{H_0,A,B,H_{2},\ldots,H_n}_{(fu)^{[n+1]}}(V_1,(A-B)(B-i)^{-1},V_{2},V_{3},\ldots,\tilde V_n)\\
			&\quad+T^{H_0,H_{2},\ldots,H_n}_{f^{[n-1]}}(V_1(A-i)^{-1}(A-B)(B-i)^{-1}V_{2},V_{3},\ldots,\tilde V_n)\\
			=&\,-T^{H_0,A,H_{2},\ldots,H_n}_{f^{[n]}}(V_1,(A-B)(B-i)^{-1}V_{2},V_{3},\ldots,\tilde V_n)\\
			&\quad+T^{H_0,A,B,H_{2},\ldots,H_n}_{(fu)^{[n+1]}}(V_1,(A-B)(B-i)^{-1},V_{2},V_{3},\ldots,\tilde V_n)\\
			=&\,\bT^{H_0,A,B, H_{2},\ldots,H_n}_{f^{[n+1]},\{2\}}(V_1,A-B,V_{2},V_{3},\ldots,\tilde V_n)\\
			=&\,\bT^{H_0,A,B, H_{2},\ldots,H_n}_{f^{[n+1]},\{2,n+1\}}(V_1,A-B,V_{2},V_{3},\ldots,\tilde V_n)\\
			=&\,\bT^{H_0,A,B, H_{2},\ldots,H_n}_{f^{[n+1]},\{2,n+1\}}(V_1,A-B,V_{2},V_{3},\ldots,V_n)(H_n-i)^{-1}=:Y(H_n-i)^{-1},
		\end{align*}
		where the second equality is obtained from the first by applying the change of variables (Theorem \ref{thm:cov}) and using \eqref{eq:new-2}; the passage from the second to the third equality follows directly from Definition \ref{def:MOI relatively bdd} together with the identity
		\[
		f^{[n+1]}(\lambda_0,\lambda_1,\ldots,\lambda_{n+1})u(\lambda_2)
		=(fu)^{[n+1]}(\lambda_0,\ldots,\lambda_{n+1})
		-(f)^{[n]}(\lambda_0,\lambda_1,\lambda_3,\ldots,\lambda_{n+1}),
		\]
		while the passage from the third to the fourth equality, as well as from the fourth to the last, proceeds analogously to the argument used for \eqref{eq:new-1}, namely by invoking Definition \ref{def:MOI relatively bdd} and Lemma \ref{thm:well defined}. As $\ran (H_n-i)^{-1}=\dom H_n$ lies dense in $\H$, and since the operators $X$ and  $Y$ in front of $(H_n-i)^{-1}$ on the left-hand side and right-hand side are bounded, we obtain the proposition.
	\end{proof}
	
	The above proposition serves as the induction basis in the inductive proof of the following result, which is an important component in the proof of our main theorems.
	
	\begin{thm}\label{thm:superscript difference bT}
		Let $n\in\N$ and $j\in\{0,\ldots,n\}$. Let $H_0,\ldots,H_n$ be self-adjoint, and let $V_1,\ldots,V_n$ be operators. Let $\mJ\subseteq\{1,\ldots,n\}$ be a subset with $n\in\mJ$ so that $V_k$ is bounded for each $k\in\{1,\ldots,n\}\setminus\mJ$ and $V_k$ is relatively $H_k$-bounded for each $k\in\mJ$. Let $W$ be symmetric and relatively $H_j$-bounded with $H_j$-bound $<1$. If $j>0$, suppose that $V_j$ is relatively $H_j+W$-bounded (e.g., $W=tV_j$ for $t\in(-\frac1a,\frac1a)$).
		For all $f\in\mW^{n+1}_{|\mJ|+1}(\R)$ we have
		\begin{align}\label{eq:superscript difference}
			\nonumber&\mathbf T^{H_0,\ldots,H_{j-1},H_j+W,H_{j+1},\ldots,H_n}_{f^{[n]},\mathcal J}(V_1,\ldots,V_n)-\mathbf T^{H_0,\ldots,H_n}_{f^{[n]},\mathcal J}(V_1,\ldots,V_n)\\
			&\hspace{90pt}=\mathbf T^{H_0,\ldots,H_{j-1},H_j+W,H_j,\ldots,H_n}_{f^{[n+1]},p_{j+1}(\mathcal J)\cup\{j+1\}}(V_1,\ldots,V_{j},W,V_{j+1},\ldots,V_n),
		\end{align}
		where $p_{j+1}:\{1,\ldots,n\}\to\{1,\ldots,n+1\}$ is defined by $k\mapsto k$ for $k\leq j$ and $k\mapsto k+1$ for $k> j$.
	\end{thm}
	\begin{proof}
		We prove this by induction to the number of elements in $\mJ$. The induction basis, $\mJ=\{n\}$, is Proposition \ref{prop:superscript cov}.
		
		For the induction step, we choose $k\in\mJ\setminus\{n\}$. Suppose first that $k\neq j$. In fact, let us assume that $k=1$ for notational simplicity, because the general proof is precisely the same.
		
		By Lemma \ref{lem:function spaces inclusions}, from $f\in\mW^{n+1}_{|\mJ|+1}(\R)$ it follows that $fu\in\mW^{n+1}_{|\mJ|}(\R)$ and $f\in\mW^{n}_{|\mJ|}(\R)$. Theorem \ref{thm:cov rel bdd2} therefore implies
		\begin{align}
			\nonumber&\bT^{H_0,\ldots,H_{j-1},H_j+W,H_{j+1},\ldots,H_n}_{f^{[n]},\mJ}(V_1,\ldots,V_n)-\bT^{H_0,\ldots,H_n}_{f^{[n]},\mJ}(V_1,\ldots,V_n)\\
			\nonumber=&\,\bT^{H_0,\ldots,H_{j-1},H_j+W,H_{j+1},\ldots,H_n}_{(fu)^{[n]},\mJ\setminus\{1\}}(V_1(H_1-i)^{-1},V_2,\ldots,V_n)\\
			\nonumber&\quad-\bT^{H_0,H_2,\ldots,H_{j-1},H_j+W,H_{j+1},\ldots,H_n}_{f^{[n-1]},p_1^{-1}(\mJ)}(V_1(H_1-i)^{-1}V_2,V_3,\ldots,V_n)\\
			\nonumber&\quad-\bT^{H_0,\ldots,H_n}_{(fu)^{[n]},\mJ\setminus\{1\}}(V_1(H_1-i)^{-1},V_2,\ldots,V_n)+\bT^{H_0,H_2,\ldots,H_n}_{f^{[n-1]},p_1^{-1}(\mJ)}(V_1(H_1-i)^{-1}V_2,V_3,\ldots,V_n)\\
			\nonumber=& \,\Big(\bT^{H_0,\ldots,H_{j-1},H_j+W,H_{j+1},\ldots,H_n}_{(fu)^{[n]},\mJ\setminus\{1\}}(V_1(H_1-i)^{-1},V_2,\ldots,V_n)\\
			\nonumber& \quad\qquad-\bT^{H_0,\ldots,H_n}_{(fu)^{[n]},\mJ\setminus\{1\}}(V_1(H_1-i)^{-1},V_2,\ldots,V_n)\Big)\\
			\nonumber&-\Big(\bT^{H_0,H_2,\ldots,H_{j-1},H_j+W,H_{j+1},\ldots,H_n}_{f^{[n-1]},p_1^{-1}(\mJ)}(V_1(H_1-i)^{-1}V_2,V_3,\ldots,V_n)\\
			&\qquad\qquad\qquad-\bT^{H_0,H_2,\ldots,H_n}_{f^{[n-1]},p_1^{-1}(\mJ)}(V_1(H_1-i)^{-1}V_2,V_3,\ldots,V_n)\Big)\label{eq:apply cov}
		\end{align}
		We apply the induction hypothesis (two times) to the right-hand side of \eqref{eq:apply cov}, and obtain
		\begin{align*}
			&\bT^{H_0,\ldots,H_{j-1},H_j+W,H_{j+1},\ldots,H_n}_{f^{[n]},\mJ}(V_1,\ldots,V_n)-\bT^{H_0,\ldots,H_n}_{f^{[n]},\mJ}(V_1,\ldots,V_n)\\
			=&\,\bT^{H_0,\ldots,H_{j-1},H_j+W,H_{j},\ldots,H_n}_{(fu)^{[n+1]},p_{j+1}(\mJ\setminus\{1\})\cup\{j+1\}}(V_1(H_1-i)^{-1},V_2,\ldots,V_{j},W,V_{j+1},\ldots,V_n)\\
			&\quad-\bT^{H_0,H_2,\ldots,H_{j-1},H_j+W,H_{j},\ldots,H_n}_{f^{[n-1]},p_{j}(p_1^{-1}(\mJ))\cup\{j\}}(V_1(H_1-i)^{-1}V_2,V_3,\ldots,V_j,W,V_{j+1},\ldots,V_n)\\
			=&\,\mathbf T^{H_0,\ldots,H_{j-1},H_j+W,H_j,\ldots,H_n}_{f^{[n+1]},p_{j+1}(\mathcal J)\cup\{j+1\}}(V_1,\ldots,V_{j},W,V_{j+1},\ldots,V_n),
		\end{align*}
		where we have applied Theorem \ref{thm:cov rel bdd2} again in the last step, this time using $f\in\mW^{n+1}_{|\mJ|+1}(\R)$, as $|p_{j+1}(\mJ)\cup\{j+1\}|=|\mJ|+1$.
		
		We now handle the case $k=j$. We may assume that $\mJ=\{j,n\}$ without loss of generality; this can be achieved if in the induction process we remove elements $k\neq j$ from $\mJ\setminus\{n\}$ until we end up with $\mJ=\{j,n\}$.  Again, for purely notational simplicity, we assume that $k=j=1$.
		
		
		In the following, the first identity is obtained by applying Theorem~\ref{thm:cov rel bdd2} (twice, using $f\in\mW^n_{|\mJ|}(\R)$). The passage from the first to the second identity, and from the second to the third, consists of straightforward algebraic rearrangements of the initial expression. The transition from the third to the fourth identity follows from Lemma~\ref{lem:second resolvent identity} (the second resolvent identity). The fourth-to-fifth identity is derived using Proposition~\ref{prop:superscript cov} (using $fu\in\mW^{n+1}_{|\mJ|}(\R)$,) together with Theorem~\ref{thm:cov rel bdd2} (using $f\in\mW^n_{|\mJ|}(\R)$). Finally, the last identity follows from the fifth by another application of Theorem~\ref{thm:cov rel bdd2}(using $f\in\mW^{n+1}_{|\mJ|+1}(\R)$).
		
		\begin{align*}
			&\bT^{H_0,H_1+W,H_{2},\ldots,H_n}_{f^{[n]},\{1,n\}}(V_1,\ldots,V_n)-\bT^{H_0,\ldots,H_n}_{f^{[n]},\{1,n\}}(V_1,\ldots,V_n)\\
			=&\,\bT^{H_0,H_1+W,H_{2},\ldots,H_n}_{(fu)^{[n]},\{n\}}(V_1(H_1+W-i)^{-1},V_{2},\ldots,V_n)\\
			&\qquad\qquad-\bT^{H_0,H_{2},\ldots,H_n}_{f^{[n-1]},\{n-1\}}(V_1(H_1+W-i)^{-1}V_{2},V_{3},\ldots,V_n)\\
			&\quad-\bT^{H_0,\ldots,H_n}_{(fu)^{[n]},\{n\}}(V_1(H_1-i)^{-1},V_{2},\ldots,V_n)+\bT^{H_0,H_{2},\ldots,H_n}_{f^{[n-1]},\{n-1\}}(V_1(H_1-i)^{-1}V_{2},V_{3},\ldots,V_n)\\
			=&\,\bT^{H_0,H_1+W,H_{2},\ldots,H_n}_{(fu)^{[n]},\{n\}}(V_1(H_1+W-i)^{-1},V_{2},\ldots,V_n)-\bT^{H_0,\ldots,H_n}_{(fu)^{[n]},\{n\}}(V_1(H_1-i)^{-1},V_{2},\ldots,V_n)\\
			&\quad-\bT^{H_0,H_{2},\ldots,H_n}_{f^{[n-1]},\{n-1\}}(V_1(H_1+W-i)^{-1}V_{2},V_{3},\ldots,V_n)+\bT^{H_0,H_{2},\ldots,H_n}_{f^{[n-1]},\{n-1\}}(V_1(H_1-i)^{-1}V_{2},V_{3},\ldots,V_n)\\
			=&\,\bT^{H_0,H_1+W,H_{2},\ldots,H_n}_{(fu)^{[n]},\{n\}}(V_1(H_1+W-i)^{-1},V_{2},\ldots,V_n)-\bT^{H_0,\ldots,H_n}_{(fu)^{[n]},\{n\}}(V_1(H_1+W-i)^{-1},V_{2},\ldots,V_n) \\
			&+\bT^{H_0,\ldots,H_n}_{(fu)^{[n]},\{n\}}(V_1(H_1+W-i)^{-1}-V_1(H_1-i)^{-1},V_{2},\ldots,V_n)\\
			&- \bT^{H_0,H_{2},\ldots,H_n}_{f^{[n-1]},\{n-1\}}(V_1(H_1+W-i)^{-1}V_{2}-V_1(H_1-i)^{-1}V_{2},V_{3},\ldots,V_n)
		\end{align*}	
		\begin{align*}
			=&\,\bT^{H_0,H_1+W,H_{2},\ldots,H_n}_{(fu)^{[n]},\{n\}}(V_1(H_1+W-i)^{-1},V_{2},\ldots,V_n)-\bT^{H_0,\ldots,H_n}_{(fu)^{[n]},\{n\}}(V_1(H_1+W-i)^{-1},V_{2},\ldots,V_n)\\
			&\quad-\bT^{H_0,\ldots,H_n}_{(fu)^{[n]},\{n\}}(V_1(H_1+W-i)^{-1}W(H_1-i)^{-1},V_{2},\ldots,V_n)\\
			&\quad+\bT^{H_0,H_{2},\ldots,H_n}_{f^{[n-1]},\{n-1\}}(V_1(H_1+W-i)^{-1}W(H_1-i)^{-1}V_{2},V_{3},\ldots,V_n)\\
			=&\,\mathbf T^{H_0,H_1+W,H_1,\ldots,H_n}_{(fu)^{[n+1]},\{2,n+1\}}(V_1(H_1+W-i)^{-1},W,V_{2},\ldots,V_n)-\mathbf T^{H_0,\ldots,H_n}_{f^{[n]},\{1,n\}}(V_1(H_1+W-i)^{-1}W,V_{2},\ldots,V_n)\\
			=&\,\mathbf T^{H_0,H_1+W,H_1,\ldots,H_n}_{f^{[n+1]},\{1,2,n+1\}}(V_1,W,V_{2},\ldots,V_n),
		\end{align*}
		concluding the proof.
	\end{proof}
	
	\section*{Conflict of interest statement}
	On behalf of all authors, the corresponding author states that there is no conflict of interest.
	\section*{Data availability statement}
	No datasets were generated or analyzed during the current study.


\begin{thebibliography}{99}
		
		
		\bibitem{AP}
		\href{https://link.springer.com/article/10.1134/S1064562422700041}
		{
			A. B. Aleksandrov, \& V. V. Peller, 
			{\it Functions of Pairs of Unbounded Noncommuting Self-Adjoint Operators under Perturbation} {Dokl. Math.} {106} (2022) 407--411.
		}
		
		
		\bibitem{AP24}
		\href{https://onlinelibrary.wiley.com/doi/10.1002/mana.70000}
		{
			A. B. Aleksandrov, \& V. V. Peller,
			{\it Functions of self-adjoint operators under relatively bounded and relatively trace class perturbations}, Math. Nachr. 298 (9) (2025) 3027--3048. arXiv:2411.01901 [math.FA].
		}
		
		
		\bibitem{gwaiz}
		\href{https://doi.org/10.1201/9780849306693}
		{M. A. Al-Gwaiz, {\em Theory of distributions}. Monographs and Textbooks in Pure and Applied Mathematics, 159. Marcel Dekker, Inc., New York, (1992) XII+257 pp.}
		
		
		\bibitem{ACDS}
		\href{https://www.cambridge.org/core/journals/canadian-journal-of-mathematics/article/operator-integrals-spectral-shift-and-spectral-flow/46A8C54A8303F7CC4FA38302212D2F18}
		{
			N.~A. Azamov,  A.~L. Carey, P.~G. Dodds, \& F.~A. Sukochev 
			{\it Operator integrals, spectral shift, and spectral flow}, {Canad. J. Math.} {61} (2) (2009) 241--263.
		}
		
		\bibitem{BS1}
		M. Sh.  Birman, \& M. Z. Solomyak,  {\it Double Stieltjes operator integrals.} Problems of mathematical physics, No. I: Spectral theory and wave processes (Russian). Izdat. Leningrad. Univ., Leningrad, (1966) pp. 33–67.
		
		
		
		\bibitem{Clark}
		\href{https://msp.org/pjm/1968/25-1/p03.xhtml}
		{
			C. Clark,  {\it On relatively bounded perturbations of ordinary differential operators}, {Pac. J. Math.} {25} (1) (1968) 59--70.
		}
		
		\bibitem{CoLemSu21} 
		\href{https://aif.centre-mersenne.org/articles/10.5802/aif.3422/}
		{C. Coine, C. Le Merdy and F. Sukochev, {\it When do triple operator integrals take value in the trace class?}, Ann. Inst. Fourier (Grenoble) {71} (4) (2021) 1393--1448.}
		
		
		\bibitem{Cwikel}
		\href{https://doi.org/10.2307/1971160}
		{M. Cwikel, {\it Weak type estimates for singular values and the number of bound states of Schr\"odinger operators}, Ann. Math. {106} (1977) 93--100.}
		
		
		\bibitem{CS18}
		\href{https://www.sciencedirect.com/science/article/pii/S0022123618300879}
		{
			A. Chattopadhyay, \& A. Skripka, . 
			{\it Trace formulas for relative Schatten class perturbations},
			{J. Funct. Anal.} {274} (2018) 3377--3410.
		}
		
		\bibitem{Daletskii}
		\href{https://doi.org/10.1007/BF01077924}
		{
			Y. L. Daletskii, {\it A noncommutative taylor formula and functions of triangle operators}, {Funct. Anal. Its Appl.} {24} (1990) 64--66.
		}
		
		\bibitem{DaletskiiKrein1956}
		Y. L. Daletski\u{\i},  \& S. G. Kre\u{\i}n, {\it Integration and differentiation of functions of Hermitian operators and applications to the theory of perturbations}, { Vorone\u{z}. Gos. Univ. Trudy Sem. Funkcional. Anal.} {1956} (1) (1956) pp. 81–105.
		
		\bibitem{DanSch}
		{N. Dunford, \& Jacob T. Schwartz, {\em Linear {O}perators. {I}. {G}eneral {T}heory}. John Wiley \& Sons, Inc., Hoboken, New Jersey (1988).}
		
		
		\bibitem{Feynman}
		\href{https://doi.org/10.1103/PhysRev.84.108}{R. P. Feynman,  {\it An operator calculus having applications in quantum electrodynamics}, Physical Review 84 (1) (1951) 108.}
		
		\bibitem{Fuerst}
		\href{https://arxiv.org/abs/2506.01647}
		{O. F\"urst, {\it Higher order spectral shift of Euclidean Callias operators}, Preprint, arXiv:2506.01647 [math.SP] (2025).
		}
		
		
		\bibitem{GM}
		\href{https://arxiv.org/abs/2410.14034}
		{B. Batu G\"{u}neysu, \& J. Miehe, {\it Fermionic Dyson expansions and stochastic Duistermaat-Heckman localization on loop spaces}, Preprint, arXiv:2410.14034v1 [math.DG] (2024).}
		
		\bibitem{Hansen}
		\href{https://doi.org/10.1063/1.2186925}
		{
			F. Hansen, {\it Trace functions as Laplace transforms}, {J. Math. Phys.} {47} (4) (2006) 043504.
		}
		
		\bibitem{HMvN}
		\href{https://arxiv.org/pdf/2404.16338}
		{
			E. Hekkelman, E. McDonald, \& T. D. H. van Nuland, 
			{\it Multiple operator integrals, pseudodifferential calculus, and asymptotic expansions}, Preprint, arXiv:2404.16338 [math.FA] (2024).
		}
		
		\bibitem{ILV}
		\href{https://doi.org/10.1007/s00220-012-1587-8}
		{B. Iochum, C. Levy, \& D. Vassilevich, {\it Spectral action beyond the weak-field approximation}, {Commun. Math. Phys.} {316} (2012) 595--613.}	
		
		\bibitem{IM1} 
		\href{https://doi.org/10.1016/j.geomphys.2018.02.014}
		{B. Iochum, \& T. Masson, {\it Heat asymptotics for nonminimal Laplace type operators and application to noncommutative tori}, {J. Geom. Phys.} {129} (2018) 1--24.}
		
		
		
		\bibitem{Kato}
		\href{https://www.ams.org/journals/tran/1951-070-02/S0002-9947-1951-0041010-X/S0002-9947-1951-0041010-X.pdf}
		{
			T. Kato, {\it Fundamental properties of Hamiltonian operators of Schrödinger type}, 
			{Trans. Am. Math. Soc.} {70} (2) (1951) 195--211.
		}
		
		\bibitem{KatoBook}
		\href{https://doi.org/10.1007/978-3-642-66282-9}
		{T. Kato, {\em Perturbation Theory for Linear Operators} (Reprint of the 1980 edition). Springer-Verlag, Berlin (1995).}
		
		\bibitem{Koplienko}
		\href{https://doi.org/10.1007/BF00968686}
		{L. S. Koplienko, {\it The trace formula for perturbations of nonnuclear type.} {Sibirsk. Mat. Zh.} {25} (5) (1984) 62--71; English translation, {Siberian Math. J.} {25} 735–743.}
		
		\bibitem{Krein}
		\href{https://doi.org/10.1002/mana.19881370118}
		{M. G. Krein, {\it On the trace formula in perturbation theory}, {\it Mat. Sbornik N.S.} {33} (75) (1953) 597--626.}
		
		
		
		\bibitem{Lesch} \href{https://doi.org/10.4171/jncg/11-1-6}
		{M. Lesch, {\it Divided differences in noncommutative geometry: rearrangement lemma, functional calculus and expansional formula}, {J. Noncommut. Geom} {11} (1) (2017) 193--223.}
		
		\bibitem{LSZ}
		\href{https://doi.org/10.1112/plms.12301}
		{G. Levitina, F. Sukochev, \& D. Zanin, {\it Cwikel estimates revisited}, {P. Lond. Math. Soc. (3)} {120} (2) (2020) 265--304.}
		
		\bibitem{MP}
		\href{https://doi.org/10.1063/5.0056289}
		{E. McDonald, \& R. Ponge, {\it Cwikel estimates and negative eigenvalues of Schr\"odinger operators on noncommutative tori}, {J. Math. Phys.} {63} (4) (2022).}
		
		
		
		\bibitem{vNS22}
		\href{https://doi.org/10.4171/jst/425}
		{
			T. D. H. van Nuland, \& A. Skripka,
			\newblock {\it Spectral shift for relative Schatten class perturbations}, {J. Spectr. Theor.} {12} (2022) 1347--1382.
		}
		
		\bibitem{vNS23}
		\href{https://arxiv.org/pdf/2211.03330.pdf}
		{
			T. D. H. van Nuland,  \& A. Skripka,  {\it Higher-order spectral shift function for resolvent comparable perturbations}, { J. Operat. Theor.} 93 (1) (2025) 3--36.
		}
		
		\bibitem{vNvS21a}
		\href{https://doi.org/10.4171/jncg/500}
		{
			T. D. H. van Nuland,  \& W. D. van Suijlekom,  
			{\it Cyclic cocycles in the spectral action},
			{J. Noncommut. Geom.} {16} (2021) 1103--1135.
		}
		
		\bibitem{oneloop}
		\href{https://doi.org/10.1007/JHEP05(2022)078}
		{T. D. H. van Nuland, \&  W. D. van Suijlekom, {\it One-loop corrections to the spectral action}, {J. High Energy Phys.} (2022) Paper No. 078, 14 pp.}
		
		
		
		\bibitem{Pavlov1971}
		\href{https://doi.org/10.1007/978-1-4757-0013-8_4}
		{B. Pavlov, {\it On multidimensional integral operators}, Linear Operators and Operator Equations (1971) pp. 81–97.}
		
		\bibitem{Peller}
		\href{https://www.sciencedirect.com/science/article/pii/S0022123605003307}
		{
			V. V. Peller, {\it Multiple operator integrals and higher operator derivatives}, \textit{J. Funct. Anal.} {233} (2006) 515--544.
		}
		
		
		\bibitem{PoSu_Crel09}
		\href{https://doi.org/10.1515/CRELLE.2009.006}
		{D. Potapov, \& F. Sukochev, {\it Unbounded Fredholm modules and double operator integrals}, {J. Reine Angew. Math.} {626} (2009) 159--185.}
		
		\bibitem{PSS}
		\href{https://doi.org/10.1007/s00222-012-0431-2}
		{D. Potapov, A. Skripka, \& F. Sukochev, {\it Spectral shift function of higher order}, {Invent. Math.} {193} (3) (2013) 501--538.}
		
		\bibitem{PoSkSu14} 
		\href{https://londmathsoc.onlinelibrary.wiley.com/doi/10.1112/plms/pdt024}
		{ D. Potapov, A. Skripka and F. Sukochev, {\it Higher-order spectral shift for contractions}, Proc. Lond. Math. Soc. (3) {108} (2) (2014) 327--349.}
		
		\bibitem{PoSkSu16}
		\href{https://doi.org/10.1016/j.jfa.2016.01.001}
		{D. Potapov, A. Skripka and F. Sukochev, {\it Functions of unitary operators: derivatives and trace formulas}, J. Funct. Anal., {\bf 270} (2016), no. 6, 2048--2072.}
		
		\bibitem{PrSk26}
		\href{https://arxiv.org/pdf/2611.05879v1}{C. Pradhan and A. Skripka, {\it BMV conjecture and positivity of spectral shift for relatively Schatten perturbations}, preprint, arXiv:2611.05879v1, (2026).}
		
		
		\bibitem{ReedSimonI}
		\href{https://doi.org/10.1016/B978-0-12-585001-8.X5001-6}
		{M. Reed, \& B. Simon, {\em Methods of modern mathematical physics. I. Functional analysis},  Academic Press, New York-London (1972).}
		
		
		\bibitem{Rennie}
		\href{http://dx.doi.org/10.1023/B:KTHE.0000021311.27770.e8}
		{A. C. Rennie, {\it Summability for nonunital spectral triples}, {K-Theory} {31} (1) (2004) 71--100.}
		
		
		
		\bibitem{Simon}
		\href{https://www.ams.org/books/surv/120/surv120-endmatter.pdf}
		{
			B. Simon,
			{\em Trace ideals and their applications}.
			Mathematical surveys and monographs, volume {120},
			American mathematical society, Providence, RI (2005).
		}
		
		
		
		\bibitem{S14}
		\href{https://www.sciencedirect.com/science/article/pii/S0022123613004886}
		{
			A. Skripka,  
			{\it Asymptotic expansions for trace functionals}, {J. Funct. Anal.} {266} (2014) 2845--2866.
		}
		
		\bibitem{Sk17Adv} 
		\href{https://doi.org/10.1016/j.aim.2017.02.026}
		{A. Skripka, {\it Estimates and trace formulas for unitary and resolvent comparable perturbations}, Adv. Math. {311} (2017) 481--509.}
		
		\bibitem{S21}
		\href{https://doi.org/10.1063/5.0017648}
		{A.  Skripka, {\it Lipschitz estimates for functions of Dirac and Schr\"{o}dinger operators}, {J. Math. Phys.} {62} (1) (2021) Paper No. 013506, 28 pp.}
		
		
		\bibitem{ST19}
		\href{https://link.springer.com/book/10.1007/978-3-030-32406-3}
		{A. Skripka, \&  A. Tomskova, {\em Multilinear Operator Integrals: Theory and Applications.} Lecture Notes in Math. 2250, Springer International Publishing (2019) XI+192 pp.
		}
		
		\bibitem{SolomyakStenkin1971}
		\href{https://doi.org/10.1007/978-1-4757-0013-8_5}
		{M. Solomyak, \& V. Sten’kin, {\it On one class of Stieltjes multiple-integral operators.} Linear Operators and Operator Equations, (1971)  pp. 99–108.}
		
		\bibitem{Stenkin1977}
		V. Sten’kin, {\it Multiple operator integrals}, Izvestiya Vysshikh Uchebnykh Zavedenii. Matematika {4} (1977)  pp. 102–115.
		
		\bibitem{Sui11}
		\href{https://doi.org/10.1016/j.jfa.2010.12.012}
		{W.~D. van Suijlekom, {\it Perturbations and operator trace functions}, {J. Funct. Anal.} {260} (2011) 2483--2496.}
		
		
		\bibitem{SZ}
		\href{https://smf.emath.fr/sites/default/files/2023-12/smf_ast_445__sample.pdf}
		{F. Sukochev, \& D. Zanin, {\it The Connes character formula for locally compact spectral triples}, { Ast\'erisque} { 445} (2023) 150 pp.}
		
		\bibitem{widom}
		H. Widom, {\em When are differentiable functions differentiable?}, Linear and
		Complex Analysis Problem Book, 199 Research Problems. Lecture Notes in Mathematics, vol. 1043, Springer-Verlag, Berlin (1984) pp. 184–188.
		
		
		\bibitem{Yafaev92}
		{D. R. Yafaev, 
			{\em Mathematical scattering theory. General theory.}
			Translations of Mathematical Monographs, 105. American Mathematical Society, Providence, RI (1992).}
		
		\bibitem{Yafaev}
		\href{https://doi.org/10.1112/S0024609305004911}
		{D. R. Yafaev, {\it A trace formula for the Dirac operator}, { Bull. London Math. Soc.} { 37} (6) (2005) 908--918.}
		
		
	\end{thebibliography}
\end{document}